\documentclass[a4paper,12pt]{amsart}
\usepackage{amssymb,amsfonts,enumerate,verbatim}
\usepackage{hyperref}
\usepackage[all]{xy}
\usepackage{pstricks}
\usepackage{mathptmx}

\usepackage{multicol}


\makeindex

\newcommand{\C}{{\mathbb C}}
\newcommand{\R}{{\mathbb R}}
\newcommand{\Z}{{\mathbb Z}}
\newcommand{\N}{{\mathbb N}}

\newcommand{\del}{\partial}
\newcommand{\m}{^{-1}}
\newcommand{\im}{\operatorname{im}}
\renewcommand{\phi}{\varphi}
\newcommand{\eps}{\varepsilon}

\renewcommand{\div}{\operatorname{div}}

\newcommand{\dom}{\operatorname{dom}}
\newcommand{\Bad}{B^\mathrm{ad}}
\newcommand{\BAPS}{B_\mathrm{APS}}
\newcommand{\ad}{\mathrm{ad}}

\newcommand{\ind}{\operatorname{ind}}
\newcommand{\<}{\left\langle}       
\renewcommand{\>}{\right\rangle}

\newcommand{\id}{{\operatorname{id}}}

\newcommand{\supp}{\operatorname{supp}}

\newcommand{\dM}{{\partial M}}

\newcommand{\Dnorm}[1]{\| #1 \|_{D}}
\newcommand{\LzdM}[1]{\| #1 \|_{L^2(\dM)}}
\newcommand{\Lz}[2]{\| #1 \|_{L^2(#2)}}
\newcommand{\LzM}[1]{\| #1 \|_{L^2(M)}}

\newcommand{\LzMp}[2]{( #1 , #2 )_{L^2(M)}}

\newcommand{\LzZq}[1]{\| #1 \|_{L^2(Z_{[0,\rho]})}}

\newcommand{\HsdM}[1]{\| #1 \|_{H^s(\dM)}}

\newcommand{\HkM}[1]{\| #1 \|_{H^k(M)}}

\newcommand{\HlD}[1]{\| #1 \|_{H^1_D(M)}}

\newcommand{\HhdM}[1]{\| #1 \|_{H^{1/2}(\dM)}}
\newcommand{\HmhdM}[1]{\| #1 \|_{H^{-1/2}(\dM)}}
\newcommand{\HcdM}[1]{\| #1 \|_{\check{H}(A)}}
\newcommand{\HddM}[1]{\| #1 \|_{\hat{H}(A)}}
\newcommand{\Hlloc}{H^1_\mathrm{loc}}
\newcommand{\Hkloc}{H^k_\mathrm{loc}}

\newcommand{\Cuc}{C^\infty_c}
\newcommand{\Cucc}{C^\infty_{cc}}
\newcommand{\Cu}{C^\infty}

\newcommand{\D}{\mathcal D}

\newcommand{\Dmax}{D_{\mathrm{max}}}

\newcommand{\nor}{\mathrm{nor}}
\newcommand{\abs}{\mathrm{abs}}
\newcommand{\rel}{\mathrm{rel}}

\newcommand{\Hfin}{H^{\operatorname{fin}}}
\newcommand{\loc}{\operatorname{loc}}
\newcommand{\Zrq}{Z_{[0,\rho]}}

\newtheorem{thm}{Theorem}[section]
\newtheorem{lemma}[thm]{Lemma}
\newtheorem{prop}[thm]{Proposition}
\newtheorem{cor}[thm]{Corollary}
\newtheorem{add}[thm]{Addendum}

\theoremstyle{definition}
\newtheorem{remark}[thm]{Remark}
\newtheorem{remarks}[thm]{Remarks}
\newtheorem{definition}[thm]{Definition}

\newtheorem{notation}[thm]{Notation}
\newtheorem{example}[thm]{Example}
\newtheorem{examples}[thm]{Examples}

\newtheorem{stase}[thm]{Standard Setup}
\newtheorem{facts}[thm]{Facts}

\newcommand{\dref}[1]{Definition~\ref{#1}}
\newcommand{\cref}[1]{Corollary~\ref{#1}}
\newcommand{\lref}[1]{Lemma~\ref{#1}}
\newcommand{\pref}[1]{Proposition~\ref{#1}}
\newcommand{\tref}[1]{Theorem~\ref{#1}}
\newcommand{\sref}[1]{Section~\ref{#1}}
\newcommand{\eref}[1]{Example~\ref{#1}}

\long\def\symbolfootnote[#1]#2{\begingroup%
\def\thefootnote{\fnsymbol{footnote}}\footnote[#1]{#2}\endgroup} 

\begin{document}


\title
[Boundary Value Problems]
{Boundary Value Problems for Elliptic Differential Operators of First Order}

\author{Christian B{\"a}r}
\author{Werner Ballmann}

\address{
Universit{\"a}t Potsdam\\
Institut f\"ur Mathematik\\
Am Neuen Palais 10\\
Haus 8\\
14469 Potsdam\\
Germany
}
\address{
Max Planck Institute for Mathematics\\
Vivatsgasse 7\\
53111 Bonn\\
Germany
}
 
\email{baer@math.uni-potsdam.de} 
\email{hwbllmnn@mpim-bonn.mpg.de}



\thanks{}

\date{\today}

\begin{abstract}
We study boundary value problems for linear elliptic differential operators of order one.
The underlying manifold may be noncompact, but the boundary is assumed to be compact.
We require a symmetry property of the principal symbol of the operator along the boundary.
This is satisfied by Dirac type operators, for instance.

We provide a selfcontained introduction to (nonlocal) elliptic boundary conditions,
boundary regularity of solutions, and index theory.
In particular, we simplify and generalize the traditional theory
of elliptic boundary value problems for Dirac type operators.
We also prove a related decomposition theorem, a general version of Gromov and Lawson's relative index theorem and a generalization of the cobordism theorem.
\end{abstract}

\maketitle


\newpage
\tableofcontents


\newpage
\section{Introduction}
\label{secIntro}


In their attempt to generalize Hirzebruch's signature theorem to compact manifolds
with boundary, Atiyah, Patodi, and Singer arrived at a boundary condition
which is nonlocal in nature and involves the spectrum
of an associated selfadjoint differential operator on the boundary.
In fact, they obtained an index theorem for a certain class of first order elliptic differential 
operators on compact manifolds with boundary \cite[Theorem 3.10]{APS}.
For the standard differential operators of first order encountered in Riemannian geometry, 
their assumption means that a sufficiently small collar about the boundary is cylindrical,
that is, isometric to the Riemannian product of an interval times the boundary.
In many applications, this is a completely satisfactory assumption.
They also discuss the $L^2$-theory for the natural extension of the operator
to the noncompact manifold which is obtained
by extending the cylinder beyond the boundary to a one-sided infinite cylinder.
The work of Atiyah, Patodi, and Singer lies at the heart of many investigations
concerning boundary value problems and $L^2$-index theory
for first order elliptic differential operators, and this includes the present article.

The original motivation for our present studies came from the relative index theorem
of Gromov and Lawson, see \cite[Thm. 4.18]{GL} or \tref{relind} below.
After the decomposition theorem in \cite[Thm.~23.3]{BW} and \cite[Thm.~4.3]{BL1}
was used in \cite{BB1} and \cite{BB2} to obtain index theorems for first order
geometric differential operators on certain noncompact Riemannian manifolds,
we observed that the decomposition theorem could also be used
for a short proof of the relative index theorem.
The drawback of this argument is, however, that the proof of the decomposition theorem in the above references \cite{BW} and \cite{BL1} involves a heavy technical machinery so that, as a whole, the proof of the relative index theorem would not be simplified.
Therefore, our first objective is a simplification of the theory of boundary value problems
for (certain) first order elliptic differential operators,
and the main result of this endeavor is formulated in \tref{tell}
(together with Addendum~\ref{addbad}) and \tref{treg} below.
Another objective is the $L^2$-index theory for noncompact manifolds, see \tref{fred}.
We arrive, finally, at simple proofs of the decomposition and relative index theorems.

To formulate our main results, we start by fixing the setup for our investigations.
We consider Hermitian vector bundles $E,F$ over a manifold $M$
with compact boundary $\partial M$ and a differential operator $D$ from $E$ to $F$.
We do not equip $M$ with a Riemannian metric
but we assume that $M$ is endowed with a smooth volume element $\mu$.
Then functions can be integrated over $M$ and the spaces $L^2(M,E)$ and $L^2(M,F)$
of square-integrable sections of $E$ and $F$ are defined.
By $D_{cc}$ we denote $D$, considered as an unbounded operator
from $L^2(M,E)$ to $L^2(M,F)$ with domain $C_{cc}^\infty(M,E)$,
the space of smooth sections of $E$ with compact support
in the interior of $M$\symbolfootnote[2]{
Throughout the article, the indices $c$ and $cc$ indicate compact support in $M$
and in the interior of $M$, respectively.
In the index at the end of the article, the reader finds much of the standard
and all of the nonstandard notation.}.

We denote by $D^*$ the formal adjoint of $D$.
The {\em maximal extension}\index{maximal extension}
$D_{\max}$\index{1Dmax@$D_{\max}$} of $D$ is the adjoint of $D^*_{cc}$
in the sense of functional analysis.
That is, $\dom(D_{\max})$ is the space of all $\Phi\in L^2(M,E)$
such that there is a section $\Psi\in L^2(M,F)$
with $(\Phi,D^*\Xi)_{L^2(M)}=(\Psi,\Xi)_{L^2(M)}$
for all $\Xi\in C^\infty_{cc}(M,F)$.
Then we set $D_{\max}\Phi:=\Psi$.
The {\em graph norm}\index{graph norm} $\|\cdot\|_{D_{\max}}$ of $D_{\max}$,
defined by \index{1normgraph@$\Dnorm{\cdot}$}
\[
  \|\Phi\|_D^2 := \LzM{\Phi}^2 + \LzM{D_{\max}\Phi}^2 ,
\]
turns $\dom(D_{\max})$ into a Hilbert space.
Clearly, $C^\infty_c(M,E)\subset\dom(D_{\max})$.

\begin{definition}\label{complete}\index{complete operator}
We say that $D$ is {\em complete} if the subspace
of compactly supported sections in $\dom(D_{\max})$ is dense in $\dom(D_{\max})$
with respect to the graph norm of $D_{\max}$.
\end{definition}

By definition, completeness holds if $M$ is compact.
In the noncompact case, completeness signifies that square integrability
is a decent boundary condition at infinity for $D$.
In practice, completeness can often be checked using the following theorem:

\begin{thm}\label{cherwolf}
Suppose that $M$ carries a complete Riemannian metric
with respect to which the principal symbol $\sigma_D$ of $D$ satisfies an estimate
\[
  |\sigma_D(\xi)| \le C(\mathrm{dist}(x,\dM))\cdot|\xi|
\]
for all $x\in M$ and $\xi\in T_x^*M$,
where $C:[0,\infty)\to\R$ is a positive monotonically increasing continuous function with 
\[
  \int_0^\infty \frac{dr}{C(r)} = \infty .
\]
Then $D$ and $D^*$ are complete.
\end{thm}

This theorem applies, for instance, if $C$ is a constant.

\begin{remark}\label{measuregen}
Note that we do not assume that $\mu$ is the volume element
induced by the Riemannian metric.
In fact, the principal symbols of $D$ and $D^*$ do not depend on the choice of $\mu$,
although $D^*$ does. 
\end{remark}

Fix a vector field $T$\index{1T@$T$} of $M$ along $\partial M$
pointing into the interior of $M$ and let $\tau$\index{1Tau@$\tau$}
be the {\em associated one-form} along $\partial M$,\index{associated one-form}
that is, $\tau(T)=1$and $\tau|_{T\partial M}=0$.
Set $\sigma_0:=\sigma_D(\tau)$.\index{1Sigma0@$\sigma_0$}
Note that $\sigma_0$ is invertible if $D$ is elliptic.

\begin{definition}\label{defbs}
We say that $D$ is {\em boundary symmetric}\index{boundary symmetric}
(with respect to $T$) if $D$ is elliptic and
\begin{equation}
  \sigma_0(x)^{-1}\circ\sigma_D(\xi) : E_x \to E_x
  \quad\text{ and}\quad
  \sigma_D(\xi)\circ\sigma_0(x)^{-1} : F_x \to F_x \label{BS}
\end{equation}
are skew-Hermitian, for all $x\in\partial M$ and $\xi\in T^*_x\partial M$.
Here we identify $\xi$ with its extension to $T_xM$ which satisfies $\xi(T)=0$.
\end{definition}

It is somewhat hidden in this definition, but will become clear in the text,
that $D$ is boundary symmetric if and only if $D^*$ is boundary symmetric.
Clearly, boundary symmetry of $D$ does not depend on the choice of $\mu$.

We will see in \lref{nf} that the boundary symmetry of $D$ implies
existence of essentially selfadjoint elliptic first order differential operators $A$
on $E|_{\partial M}$ and $\tilde A$ on $F|_{\partial M}$
with symbols given by $\sigma_0^{-1}\circ\sigma_D$
and $(\sigma_D\circ\sigma_0^{-1})^*$, respectively.
We will call such operators {\em  adapted}\index{adapted operator} (with respect to the choice of $T$).
Note that we may add zero-order terms to adapted operators
without violating the requirement on their principal symbols.
\lref{nf} also shows that, in a collar about $\dM$,
$D$ and $D^*$ can be represented in the normal form\index{normal form near boundary}
\begin{equation}
  D = \sigma_t \left( \frac{\partial}{\partial t} + A + R_t \right)
  \quad\text{and}\quad
  D^* = -\sigma_t^* \left( \frac{\partial}{\partial t} + \tilde A + \tilde R_t \right) ,
  \label{norfor}
\end{equation}
where $t$ is an inward coordinate in the collar with $\partial/\partial t=T$
along $\dM=t^{-1}(0)$ and where, for $0\le t<r$,
\begin{enumerate}[(i)]
\item
$\sigma(t,x)=\sigma_t(x)=\sigma_D(dt|_{(t,x)})$\index{1Sigmat@$\sigma_t$}
for all $x\in\dM$;
\item
$R_t$ and $\tilde R_t$ are differential operators of $E|_{\dM}$ and $F|_{\dM}$
of order at most one, respectively,\index{1Rt@$R_t$}\index{1Rttilde@$\tilde R_t$}
whose coefficients depend smoothly on $t$ and
such that $R_0$ and $\tilde R_0$ are of order zero.
\end{enumerate}
Since $dt=\tau$ along $\dM$, the notation $\sigma_t$ is in peace
with the notation $\sigma_0=\sigma_D(\tau)$ further up.

To put the normal form~\eqref{norfor} into perspective,
we mention that Atiyah, Patodi, and Singer consider the case
where $\sigma$ does not depend on $t$ and where $R_t=\tilde R_t=0$.

\begin{stase}\label{stase}\index{standard setup}
Our standard setup, which we usually assume, is as follows:
\begin{itemize}
\item $M$ is a manifold with compact boundary $\dM$ (possibly empty);
\item $\mu$ is a smooth volume element on $M$;
\item $E$ and $F$ are Hermitian vector bundles over $M$;
\item $D$ is an elliptic differential operator from $E$ to $F$ of order one;
\item $T$ is an interior vector field along $\dM$;
\item $D$ is boundary symmetric with respect to $T$;
\item $A$ and $\tilde A$ are  adapted differential operators
on $E|_{\partial M}$ and $F|_{\partial M}$ \\ for $D$ and $D^*$, respectively.
\end{itemize}
\end{stase}

\begin{example}\label{introdir}
Let $M$ be a Riemannian manifold with boundary
and $\mu$ be the associated Riemannian volume element.
Then Dirac type operators over $M$ are elliptic and boundary symmetric
with respect to the interior unit normal field of $\dM$,
compare Example~\ref{exabso} (a).
\end{example}

For $I\subset\R$ and $s\in\R$,
we denote by $H^s_I(A)$\index{1HAsI@$H^s_I(A)$} the closed subspace
of the Sobolev space $H^s(\partial M,E)$ spanned by the eigenspaces of $A$
with eigenvalue in $I$.
We denote by $Q_I$\index{1QI@$Q_I$} the (spectral) projection
of $H^s(\partial M,E)$ onto $H^s_I(A)$ with kernel $H^s_{\R\setminus I}(A)$.
For $a\in\R$, we let
\begin{equation}\label{checkaa}
  \check H(A):=H^{1/2}_{(-\infty,a)}(A)\oplus H^{-1/2}_{[a,\infty)}(A) .
  \index{1HAcheck@$\check{H}(A)$}
\end{equation}
Since, for any finite interval $I$, the space $H^s_I(A)$ is finite-dimensional and contained
in $C^\infty(\dM,E)$, different choices of $a$ lead to the same topological vector space.

\begin{thm}\label{md}
Assume the Standard Setup~\ref{stase} and that $D$ and $D^*$ are complete.
Then
\begin{enumerate}
\item\label{mdd}
$C^\infty_c(M,E)$ is dense in $\dom(D_{\max})$;
\item\label{mdt}
the trace map $\mathcal R$ on $C^\infty_c(M,E)$, $\mathcal R\Phi:=\Phi|_{\partial M}$,
extends to a surjective bounded linear map $\mathcal R:\dom(D_{\max})\to\check H(A)$;
\item\label{mdr}
$\Phi\in\dom(D_{\max})$ is in $H^{k+1}_{\loc}(M,E)$
if and only if $D\Phi$ is in $H^{k}_{\loc}(M,F)$
and $Q_{[0,\infty)}(\mathcal R\Phi)$ is in $H^{k+1/2}(\dM,E)$,
for any integer $k\ge0$.
\end{enumerate}
\end{thm}

\begin{remarks}\label{comreg}
(a) Completeness is irrelevant for Assertions \eqref{mdt} and \eqref{mdr},
and a similar remark applies to the boundary regularity results below;
compare Remark~\ref{cylicom} below.

(b) In the case where $\dM=\emptyset$, Assertion \eqref{mdd} of \tref{md} says
that $D_{\max}$ is equal to the closure of $D_{cc}$.
If $D$ is {\em formally selfadjoint}, i.e., if $D=D^*$,
this means that $D$ is {\em essentially selfadjoint}.
 
(c) Assertion \eqref{mdt} of \tref{md} implies that, as a topological vector space,
$\check H(A)$ does not depend on the choice of $A$
(as long as its symbol is given by $\sigma_0^{-1}\circ\sigma_D$).
Carron pointed out to us a direct proof of this fact:
If $A'$ is a further operator of the required kind,
then the difference $\delta=Q_{[a,\infty)}-Q'_{[a,\infty)}$
of the corresponding spectral projections 
is a pseudo-differential operator of order $-1$,
thus $\delta$ maps $H^{-1/2}(\dM,E)$ to $H^{1/2}(\dM,E)$.
\end{remarks}

\tref{md} implies that, for any closed subspace $B\subset\check H(A)$,
the restriction $D_{B,\max}$ of $D_{\max}$ to
\begin{equation}
  \dom(D_{B,\max}) := \{ \Phi \in \dom(D_{\max}) \mid \mathcal R(\Phi) \in B \}
\end{equation}
is a closed extension of $D_{cc}$.
We will see in \pref{closedext} that any closed operator
between $D_{cc}$ and $D_{\max}$ is of this form.
In particular, the {\em minimal extension}\index{minimal extension}
$D_{\min}$\index{1Dmin@$D_{\min}$} of $D$,
that is, the closure of $D_{cc}$,
is given by the Dirichlet boundary condition, that is, by $D_{0,\max}$.
We arrive at the following

\begin{definition}
A \emph{boundary condition}\index{boundary condition} for $D$
is a closed subspace of $\check H(A)$.
\end{definition}

As we already mentioned above, we have
\begin{equation}
  D_{cc} \subset D_{\min} \subset D_{B,\max} \subset D_{\max}
  \label{dddd}
\end{equation}
for any boundary condition $B\subset\check H(A)$,
explaining the terminology {\em minimal} and {\em maximal} extension.

We will see in Subsection~\ref{suseabc} that, for any boundary condition $B$,
the adjoint operator of $D_{B,\max}$ is the closed extension $D^*_{\Bad,\max}$
of $D^*_{cc}$ with boundary condition
\begin{equation}
  \Bad = \{ \psi \in \check H(\tilde A) \mid \text{$\beta(\phi,\psi)
  = 0$ for all $\phi\in B$} \} ,
  \index{1Bad@$\Bad$}
\end{equation}
where $\beta$\index{1Beta@$\beta$} denotes the natural extension
of the sesqui-linear form
\begin{equation}
  \beta (\phi,\psi) := (\sigma_0\phi,\psi)_{L^2(\dM)}
\end{equation}
on $C^\infty(\partial M,E)\times C^\infty(\partial M,F)$
to $\check H(A)\times\check H(\tilde A)$,
compare \lref{duality}.

\begin{definition}\label{ellreg}
A boundary condition $B$ is said to be
\begin{enumerate}[(i)]
\item\label{elli}
{\em elliptic}\index{elliptic boundary condition} if
\[
  \dom(D_{B,\max}) \subset H^1_{\loc}(M,E)
  \quad\text{and}\quad
  \dom(D^*_{B^{\ad},\max}) \subset H^1_{\loc}(M,F) ;
\]
\item\label{regu}
{\em $m$-regular}\index{regular boundary condition},
where $m\ge0$ is an integer, if
\begin{align*}
  D_{\max}\Phi\in H^k_{\loc}(M,F) &\Longrightarrow \Phi\in H^{k+1}_{\loc}(M,E) , \\
  D^*_{\max}\Psi\in H^k_{\loc}(M,E) &\Longrightarrow \Psi\in H^{k+1}_{\loc}(M,F)
\end{align*}
for all $\Phi\in \dom(D_{B,\max})$, $\Psi\in \dom(D^*_{\Bad,\max})$, and $0\le k\le m$.
\item
{\em $\infty$-regular} if it is $m$-regular for all integers $m\ge0$.
\end{enumerate}
\end{definition}

We recall that the implications in \eqref{regu} hold in the interior of $M$
for all integers $k\ge0$,
by the standard interior regularity theory for elliptic differential operators.

\begin{remarks}\label{ell2}
(a) By definition, we have $H^0_{\loc}(M,E)=L^2_{\loc}(M,E)$, and similarly for $F$.
Hence a boundary condition is elliptic if and only if it is $0$-regular.

(b) By \tref{md}, Remark~\ref{comreg}, and the above,
a boundary condition $B$ is elliptic if and only if
\[
  B\subset H^{1/2}(\dM,E)
  \quad\text{and}\quad \Bad\subset H^{1/2}(\dM,F) .
\]
\end{remarks}

To state our main theorem on boundary regularity, we introduce the notation
$U^s:=U\cap H^s(\dM,E)$, for any $U\subset\cup_{s'\in\R}H^{s'}(\dM,E)$ and $s\in\R$.

\begin{thm}\label{tell}
Assume the Standard Setup~\ref{stase} and that $D$ and $D^*$ are complete.
Let $B\subset\check H(A)$ be a boundary condition.
Then $B$ is elliptic if and only if, for some (and then any) $a\in\R$,
there are $L^2$-orthogonal decompositions
\begin{equation}
  L^2_{(-\infty,a)}(A) = V_-\oplus W_- 
  \quad\text{and}\quad
  L^2_{[a,\infty)}(A) = V_+\oplus W_+
  \label{dell}
\end{equation}
such that
\begin{enumerate}[(i)]
\item\label{well}
$W_-$ and $W_+$ are finite-dimensional and contained in $H^{1/2}(\dM,E)$;
\item\label{gell}
there is a bounded linear map $g:V_-\to V_+$ with
\[
  g(V_-^{1/2})\subset V_+^{1/2}
  \quad\text{and}\quad
  g^*(V_+^{1/2})\subset V_-^{1/2}
\]
\end{enumerate}
such that
\begin{equation}
  B = W_+ \oplus \{ v+gv \mid v \in V_-^{1/2} \} .
  \label{bell}
\end{equation}
\end{thm}

\begin{add}\label{addbad}
Assume the Standard Setup~\ref{stase} and that $D$ and $D^*$ are complete.
Let $B$ be an elliptic boundary condition as in \tref{tell}.
Then
\begin{equation}
  \Bad = (\sigma_0^{-1})^* \left(W_- \oplus \{ v-g^*v \mid v \in V_+^{1/2} \} \right) .
  \label{bad}
\end{equation}
\end{add}

\begin{remark}\label{rembad}
For boundary conditions $B$ and $\Bad$ as in \tref{tell} and Addendum~\ref{addbad},
$\sigma_0^*(\Bad)$ is the $L^2$-orthogonal complement of $B$ in $H^{1/2}(\dM,E)$,
by the orthogonality of the decomposition \eqref{dell}, \eqref{bell}, and \eqref{bad}.
\end{remark}

\begin{thm}\label{treg}
Assume the Standard Setup~\ref{stase} and that $D$ and $D^*$ are complete.
Let $B$ be an elliptic boundary condition as in \tref{tell} and $m\ge0$ be an integer.
Then $B$ is $m$-regular if and only if
\begin{enumerate}[(i)]
\item $W_{\pm}\subset H^{m+1/2}(\dM,E)$;
\item\label{greg}  
$g(V_-^{m+1/2})\subset V_+^{m+1/2}$
and $g^*(V_+^{m+1/2})\subset V_-^{m+1/2}$.
\end{enumerate}
\end{thm}

\begin{examples}\label{exabc}
(a)
The boundary condition\index{Atiyah-Patodi-Singer boundary condition}
\begin{equation}
  \BAPS := H^{1/2}_{(-\infty,0)}(A)
  \label{baps}
\end{equation}
of Atiyah, Patodi, and Singer is $\infty$-regular.
More generally, given any $a\in\R$,
\begin{equation}
  B(a) = H^{1/2}_{(-\infty,a)}(A)
  \label{Ba}
\end{equation}
is an $\infty$-regular boundary condition with
\begin{equation}
  B(a)^{\rm ad} = (\sigma_0^{-1})^* \left( H^{1/2}_{[a,\infty)}(A) \right) .
  \label{bell0}
\end{equation}
Boundary conditions of this type are called
{\em generalized Atiyah-Patodi-Singer boundary conditions}.
From the point of view of Theorems~\ref{tell},
they are the most basic examples of elliptic or $\infty$-regular boundary conditions.

(b) Elliptic boundary value problems in the classical sense of 
Lopatinsky and Shapiro are $\infty$-regular, see \cref{lopashap}.

(c) {\em Chiral boundary conditions}:\index{chiral boundary conditions}
Let $G$ be a chirality operator of $E$ along $\dM$,
that is, $G$ is a field of unitary involutions of $E$ along $\dM$
which anticommutes with all $\sigma_A(\xi)$, $\xi\in T^*\dM$.
Let $E^\pm$ be the subbundle of $E$ over $\dM$
which consists of the eigenbundle of $G$ for the eigenvalue $\pm1$.
Then the (local) boundary condition requiring that $\mathcal R\Phi$
be a section of $E^\pm$ is $\infty$-regular,
see \cref{cor:involution}.

Suppose for example that $M$ is Riemannian,
that $D$ is a  Dirac operator in the sense of Gromov and Lawson,
and that $T$ is the interior unit normal vector field along $\dM$.
Then Clifford multiplication with $iT$ defines a field $G$ of involutions as above.
This type of boundary condition goes under the name {\em MIT bag condition}.
If $M$ is even dimensional, there is another choice for $G$,
namely Clifford multiplication with the complex volume form,
compare \cite[Ch.~I, \S~5]{LM}.

(d) {\em Transmission conditions}:\index{transmission condition}
For convenience assume $\dM=\emptyset$.
Let $N$ be a closed and two-sided hypersurface in $M$.
Cut $M$ along $N$ to obtain a manifold $M'$
whose boundary $\dM'$ consists of two disjoint copies $N_1$ and $N_2$ of $N$.
There are natural pull-backs $\mu'$, $E'$, $F'$, and $D'$
of $\mu$, $E$, $F$, and $D$ from $M$ to $M'$.
Assume that there is an interior vector field $T$ along $N=N_1$
such that $D'$ is boundary symmetric with respect to $T$ along $N_1$.
(If, for instance, $D$ is of Dirac type, we may choose $T$ to be the interior
normal vector field of $M'$ along $N_1$.)
Then $D'$ is also boundary symmetric with respect to the interior
vector field $-T$ of $M'$ along $N=N_2$.
The {\em transmission condition}
\begin{equation*}
  B:= \left\{ (\phi,\psi) \in H^{1/2}(N,E) \oplus H^{1/2}(N,E)
  \mid \varphi = \psi \right\}
\end{equation*}
reflects the fact that $H^1_{\loc}$-sections of $E$ have a well-defined trace along $N$.
The natural pull-back to $M'$ yields a 1-1 correspondence between $H^1_{\loc}$-sections
of $E$ and $H^1_{\loc}$-sections of $E'$ with boundary values in $B$.
This boundary condition is basic for the Fredholm theory of $D$
in the case where $M$ is noncompact.
We will see in Example~\ref{ex:TransCond}
that $B$ is an $\infty$-regular boundary condition for $D'$.
\end{examples}

\begin{definition}\label{coer}
We say that $D$ is {\em coercive at infinity}\index{coercive at infinity}
if there is a compact subset $K\subset M$ and a constant $C$ such that
\begin{equation}\label{coer2}
  \| \Phi \|_{L^2(M)}
  \le C \| D\Phi \|_{L^2(M)} ,
\end{equation}
for all smooth sections $\Phi$ of $E$ with compact support in $M\setminus K$.
\end{definition}

As for completeness, coercivity at infinity is automatic if $M$ is compact.

\begin{thm}\label{fred}
Assume the Standard Setup~\ref{stase}
and that $D$ and $D^*$ are complete.
Let $B\subset\check H(A)$ be an elliptic boundary condition.

Then $D$ is coercive at infinity if and only if $D_{B,\max}$
has finite-dimensional kernel and closed image.
In this case
\begin{align*}
  \ind D_{B,\max} 
  &= \dim\ker D_{B,\max} - \dim\ker D^*_{\Bad,\max} \\
  &= \ind D_{B(a)} + \dim W_+ - \dim W_- \in \Z \cup \{-\infty\} ,
\end{align*}
where we choose the representation of $B$ as in \tref{tell}.

In particular, $D_{B,\max}$ is a Fredholm operator
if and only if $D$ and $D^*$ are coercive at infinity.
\end{thm}

Given an elliptic or regular boundary condition $B$ as in \tref{tell} above,
we obtain a continuous one-parameter family  $B_s$, $1\ge s\ge0$,
of elliptic or regular boundary conditions by substituting $sg$ for $g$.
Then $B_1=B$ and $B_0=W_+\oplus V_-^{1/2}$.
The proof of the second index formula above relies on the constancy
of $\ind D_{B_t,\max}$ under this deformation, given that $D$ is coercive at infinity.
It is clear that, in $H^{1/2}_{(-\infty,a)}(A)\oplus H^{1/2}_{[a,\infty)}(A)$,
any sum $W_-\oplus W_+$ of subspaces $W_-$ and $W_-$
of dimensions $k_-:=\dim W_-<\infty$ and $k_+:=\dim W_+<\infty$
can be deformed into any other such sum
as long as the latter has the same pair of dimensions $k_-$ and $k_+$.
We conclude that for fixed $a$, the pair $(k_-,k_+)$ of dimensions
is a complete invariant\symbolfootnote[2]{The pair $(k_-,k_+)$
depends on the choice of $a$, though.}
for the homotopy classes of elliptic boundary conditions,
and similarly for regular boundary conditions.

Suppose now for convenience that $\dM=\emptyset$ is empty,
and let $M=M'\cup M''$ be a decomposition of $M$ into two pieces
with common boundary $N=\dM'=\dM''$, a compact hypersurface.
Let $E$ and $F$ be Hermitian vector bundles over $M$
and $D$ a first-order elliptic differential operator from $E$ to $F$.
Denote by $E'$, $F'$, and $D'$ the restrictions of $E$, $F$, and $D$ to $M'$,
and analogously for $M''$.
Assume that $D'$ is  boundary symmetric
with respect to an interior normal field $T$ of $M'$ along $N$.
Then $D''$ is boundary symmetric with respect to $-T$.
Choose an adapted operator $A$ for $D'$ as in the Standard Setup~\ref{stase}.
Then $-A$ is an adapted operator for $D''$ as in the Standard Setup~\ref{stase}
and will be our preferred choice.
With respect to this choice,
the boundary condition of Atiyah, Patodi, and Singer for $D''$
is given by $H^{1/2}_{(-\infty,0)}(-A)=H^{1/2}_{(0,\infty)}(A)$.

\begin{thm}[Decomposition Theorem]\label{deco}\index{decomposition theorem}
Let $M=M'\cup M''$ and notation be as above.
Assume that $D'$ is  boundary symmetric
with respect to an interior normal field $T$ of $M'$ along $N$
and choose an adapted operator $A$ as above.

Then $D$ and $D^*$ are complete and coercive at infinity if and only if
$D'$ and $D''$ and their formal adjoints are complete and coercive at infinity.
In this case, $D_{\max}$, $D'_{B',\max}$, and $D_{B'',\max}$
are Fredholm operators and their indices satisfy
\[
  \ind D_{\max} = \ind D'_{B',\max} + \ind D''_{B'',\max} ,
\]
where $B'=B(a)=H^{1/2}_{(-\infty,a)}(A)$ and $B''=H^{1/2}_{[a,\infty)}(A)$ or,
more generally, where $B'$ is elliptic and $B''$ is the $L^2$-orthogonal complement
of $B'$ in $H^{1/2}(A)$.
\end{thm}

\begin{definition}\label{def:agree}
For $i=1,2$, let $M_i$ be manifolds, $E_i$ and $F_i$ Hermitian vector bundles over $M_i$, and $D_i:\Cu(M_i,E_i) \to \Cu(M_i,F_i)$ be differential operators.
Let $K_i\subset M_i$ be closed subsets.
Then we say that {\em $D_1$ outside $K_1$ agrees with $D_2$ outside $K_2$}
if there are a diffeomorphism $f:M_1\setminus K_1 \to M_2\setminus K_2$
and smooth fiberwise linear isometries
$\mathcal{I}_E:E_1|_{M_1\setminus K_1} \to E_2|_{M_2\setminus K_2}$ and
$\mathcal{I}_F:F_1|_{M_1\setminus K_1} \to F_2|_{M_2\setminus K_2}$ such that 
$$
\xymatrix{
E_1|_{M_1\setminus K_1} \ar[d] \ar[r]^{\mathcal{I}_E} & E_2|_{M_2\setminus K_2} \ar[d] && F_1|_{M_1\setminus K_1} \ar[d] \ar[r]^{\mathcal{I}_F}& F_2|_{M_2\setminus K_2} \ar[d]\\
M_1\setminus K_1 \ar[r]^f & M_2\setminus K_2 && M_1\setminus K_1\ar[r]^f & M_2\setminus K_2
}
$$
commute and 
\begin{equation}
  \mathcal{I}_F\circ (D_1\Phi) \circ f^{-1} = D_2(\mathcal{I}_E\circ\Phi\circ f^{-1})
\end{equation}
for all smooth sections $\Phi$ of $E_1$ over $M_1\setminus K_1$.
\end{definition}

\begin{center}
\begin{pspicture}(2,-6.5)(6,3.3)
\psbezier[linewidth=0.04](7.006123,3.244471)(6.5488772,2.5644712)(4.9584584,0.8244712)(3.686123,0.64447117)
\psbezier[linewidth=0.04](3.786123,-0.5555288)(5.206123,-0.09552883)(6.666123,0.20447117)(8.126123,-0.8155288)
\psbezier[linewidth=0.04](7.126123,3.0244713)(6.8461227,2.4044712)(6.568243,2.125087)(6.586123,1.1644711)(6.604003,0.20385532)(7.646123,0.04447117)(8.166123,-0.23552883)
\psbezier[linewidth=0.04](3.726544,0.66447115)(3.4524388,0.42447117)(3.386123,-0.23552883)(3.806123,-0.5555288)
\usefont{T1}{ptm}{m}{n}
\rput(5.397529,0.54947114){$M_1 \setminus K_1$}
\psbezier[linewidth=0.04](3.686123,0.64447117)(2.7061229,0.64447117)(1.5322459,2.4071653)(0.76612294,1.7644712)(0.0,1.1217772)(0.66296667,-2.7065866)(1.626123,-2.9755287)(2.5892792,-3.244471)(3.026123,-0.9355288)(3.766123,-0.5555288)
\psbezier[linewidth=0.04](1.1661229,0.5044712)(1.1661229,-0.29552883)(2.606123,-0.39552882)(2.606123,0.40447116)
\psbezier[linewidth=0.04](1.2661229,0.18447118)(1.746123,0.66447115)(2.246123,0.42447117)(2.4661229,0.050428618)
\psbezier[linewidth=0.04](1.246123,-1.3355289)(1.246123,-2.1355288)(2.606123,-1.9955288)(2.606123,-1.1955289)
\psbezier[linewidth=0.04](1.366123,-1.6555288)(1.5261229,-1.4349228)(1.746123,-1.1355288)(2.406123,-1.6397712)
\usefont{T1}{ptm}{m}{n}
\rput(1.7775292,-0.6905288){$K_1$}

\rput(0,-4){
\psbezier[linewidth=0.04](7.646123,2.4144711)(7.1888776,1.7344712)(5.5984583,-0.0055288305)(4.3261228,-0.18552883)
\psbezier[linewidth=0.04](4.426123,-1.3855288)(5.8461227,-0.9255288)(7.306123,-0.6255288)(8.766123,-1.6455288)
\psbezier[linewidth=0.04](7.766123,2.1944711)(7.486123,1.5744711)(7.2082434,1.295087)(7.226123,0.33447117)(7.244003,-0.6261447)(8.286123,-0.78552884)(8.806123,-1.0655289)
\psbezier[linewidth=0.04](4.3665442,-0.16552883)(4.0924387,-0.40552884)(4.026123,-1.0655289)(4.446123,-1.3855288)
\usefont{T1}{ptm}{m}{n}
\rput(6.037529,-0.28052884){$M_2 \setminus K_2$}
\psbezier[linewidth=0.04](4.3261228,-0.18552883)(3.346123,-0.18552883)(1.6122459,0.7771652)(0.80612296,0.17447117)(0.0,-0.42822286)(0.3629667,-1.9565865)(1.366123,-2.1855288)(2.3692791,-2.4144711)(3.666123,-1.7655288)(4.406123,-1.3855288)
\psbezier[linewidth=0.04](0.786123,-0.8655288)(0.786123,-1.6655288)(2.666123,-1.6255288)(2.666123,-0.82552886)
\psbezier[linewidth=0.04](0.92612296,-1.1855289)(1.846123,-0.08552883)(2.526123,-1.2655288)(2.386123,-1.2655288)
\usefont{T1}{ptm}{m}{n}
\rput(3.3975291,-0.72052884){$K_2$}
}

\psarc[linewidth=0.04,arrowsize=0.05cm 2.0,arrowlength=1.4,arrowinset=0.4]{->}(6.4,-2.1){2.1}{140}{220}
\rput(4.7,-2.1){$f$}
\end{pspicture} 
\nopagebreak 

\emph{Fig.~1}
\end{center}

If $M_i$ are Riemannian and $D_i$ are Dirac type operators, then $f$ is necessarily
an isometry because the principal symbol of a Dirac type operator determines the
Riemannian metric on the underlying manifold via the Clifford relations \eqref{Cliff1}
and \eqref{Cliff2}.

To each Dirac type operator $D$ over $M$ there is an associated $1$-density $\alpha_D$\index{1AzlphaD@$\alpha_D$} on $M$, the \emph{index density}\index{index density}, see \cite[Ch.~4]{BGV}.
At any point $x\in M$ the value of $\alpha_D(x)$ can be computed in local coordinates from the coefficients of $D$ and their derivatives at $x$.

We are now ready to state a general version of Gromov and Lawson's relative index theorem:

\begin{thm}[Relative Index Theorem]\label{relind}\index{relative index theorem}
Let $M_1$ and $M_2$ be complete Riemannian manifolds without boundary.
Let $D_i: \Cu(M_i,E_i) \to \Cu(M_i,F_i)$ be Dirac type operators which agree outside compact subsets $K_1\subset M_1$ and $K_2\subset M_2$.

Then $D_{1,\max}$ is a Fredholm operator if and only if
$D_{2,\max}$ is a Fredholm operator.
In this case,
\[
  \ind D_{1,\max} - \ind D_{2,\max}
  = \int_{K_1} \alpha_{D_1} - \int_{K_2} \alpha_{D_2} .
\]
\end{thm}

The reason for the restriction to Dirac type operators in \tref{relind} is mostly for convenience.
We need that the operators are boundary symmetric along an auxiliary hypersurface $N_i\subset M_i$.
This is automatic for Dirac type operators.
Moreover, we need the local version of the Atiyah-Singer index theorem.

Finally, we show

\begin{thm}[Cobordism Theorem]\label{cobothm}
Let $M$ be a complete and connected Riemannian manifold,
$E\to M$ be a Hermitian vector bundle,
and $D:\Cu(M,E)\to\Cu(M,E)$ be a formally selfadjoint differential operator of Dirac type.
Let $A$ be an adapted boundary operator for $D$
with respect to the interior unit normal vector field,
and let $E=E^+\oplus E^-$ be the splitting into the eigensubbundles
of the involution $i\sigma_0$ for the eigenvalues $\pm 1$.
With respect to this splitting, we write
$$
A = 
\begin{pmatrix}
A_{++} & A^- \cr A^+ & A_{--}
\end{pmatrix} .
$$
Then, if $D$ is coercive at infinity,
\[
  \ind A^+ = \ind A^- = 0 .
\]
\end{thm}

Originally, the cobordism theorem was formulated for compact manifolds $M$ with boundary and showed the cobordism invariance of the index. 
This played an important role in the original proof of the Atiyah-Singer index theorem, compare e.g.\ \cite[Ch.~XVII]{Pa} and \cite[Ch.~21]{BW}.
In this case, one can also derive the cobordism invariance from Roe's index theorem for partitioned manifolds \cite{R,Hi}.
We replace compactness of the bordism by the weaker assumption of coercivity of $D$.
Our proof is comparatively simple and makes no use of the Calder\'on projector.

{\sc Bibliographical Notes:}
\tref{cherwolf} extends results of Wolf \cite{Wo} and Chernoff \cite{Ch}
to operators which are not necessarily essentially selfadjoint
and live on manifolds with possibly nonempty boundary.
In the case where $\sigma_D$ is uniformly bounded,
our proof consists of an adaptation of the argument in the proof of Thm. II.5.7 in \cite{LM}.
The seemingly weaker assumption of Chernoff on the growth of $\sigma_D$
as stated in \ref{cherwolf} follows from an appropriate conformal change
of the underlying Riemannian metric of the manifold, see Section~\ref{sectcom}.

\tref{md} extends Propositions 2.30 and 5.7 of \cite{BBC}
(where the higher regularity part $k>0$ is not discussed).
The higher regularity part of \tref{md} and \tref{treg}
generalize the Regularity Theorem 6.5 of \cite{BL2}.
Our proof is rather elementary.
It is worth mentioning that one of the main points in \cite{BBC}
is the low regularity of the given data,
whereas we assume throughout that the data are smooth.

In the articles \cite{BF1}, \cite{BF2}, and \cite{BFO} of Booss-Bavnbek et al.,
the space $\dom D_{\max}/\dom D_{\min}$ of abstract boundary values
of $D_{\max}$ is investigated and identified with the space $\check H(A)$
as in \eqref{checkaa}, see Proposition 7.15 of \cite{BF2}.
Conversely, it is immediate from \tref{md} and \cref{domdmin}
that $\check H(A)$ is topologically isomorphic to $\dom D_{\max}/\dom D_{\min}$.

A preliminary and unpublished version of our \tref{tell} was taken up in \cite{BBC},
where it is proved for Dirac operators in the sense of Gromov and Lawson.
Conversely, the presentation and results in \cite{BBC} have influenced our discussion.
In particular, in the text we interchange definition and characterization
of elliptic boundary conditions, compare Subsection~\ref{susebc}.
In contrast to \cite{BBC}, our arguments do not involve the Calder\'on projection
and its regularity properties.

Similar, but more special boundary conditions have been considered in \cite{BC}.
The emphasis there lies on Dirac type operators with coefficients of low regularity.
Interior and boundary regularity of solutions and certain Fredholm properties are derived.

The first part of \tref{fred} generalizes a result of Anghel \cite[Thm.~2.1]{An}
to manifolds with boundary and operators
which are not necessarily essentially selfadjoint.
Variants of the index formulas in the second part of \tref{fred}
are also contained in \cite{BW} and \cite{BBC}.
\tref{relind} corresponds to Theorem 3.1 of \cite{Ca}.
At the cost of introducing a bit more of (elementary) functional analysis,
we could replace our assumption on the Fredholmness
by the weaker assumption of non-parabolicity used in \cite{Ca}.

There is a vast literature on elliptic boundary value problems based on pseudo-differential techniques, see e.g.\ \cite{B,BLZ,G2,RS,S1,S2} and the references therein.
Not all of the boundary conditions which we consider can be treated with pseudo-differential operators.
Our main objective however, is to provide a more elementary and simplified approach using standard functional analysis only.
A further reference for an elementary approach to boundary value problems
of Dirac type operators is \cite{FS};
however, the proof of the main result in \cite{FS}, Theorem 4, is not complete.

{\sc Prerequisites:}
The reader should be familiar with basic differential geometric concepts such as manifolds and vector bundles.
The functional analysis of Hilbert and Banach spaces, selfadjoint operators and their spectrum is assumed.
From the field of partial differential equations, we only require knowledge about linear differential operators, principal symbols, and the standard interior elliptic regularity theory.
No previous knowlegde of boundary value problems is necessary.

{\sc Structure of the Article:}
Most of this is clear from the table of contents.
We just give references to the places,
where the results stated in the introduction are proved:
\tref{cherwolf} in Section~\ref{sectcom},
\tref{md} in Section~\ref{secMaxDom},
\tref{tell} and Addendum~\ref{addbad} in Subsection~\ref{susebc},
\tref{treg} in Subsection~\ref{suseregu},
\tref{fred} in Subsections~\ref{susefrepro} and \ref{susefp},
the decomposition and relative index theorems~\ref{deco} and \ref{relind}
in Subsection~\ref{suserelind}.
and, finally, the cobordism theorem in Subsection \ref{coth}.

{\sc Acknowledgments:}
We would like to thank Bernd Ammann, Bernhelm Booss-Bavnbek, and Gilles Carron for helpful remarks and suggestions.
We also gratefully acknowledge the support by the Max Planck Institute for Mathematics,
the Hausdorff Center for Mathematics, the Sonderforschungsbereich~647, and the Erwin Schr\"odinger Institute for Mathematical Physics.
W.~B.\ thanks the University of Potsdam for its hospitality.


\section{Preliminaries}
\label{secPre}


\subsection{Measured manifolds}
We will consider differential operators which live on manifolds with boundary
(possibly empty).
In general, the manifolds will not be Riemannian,
but we will assume that they are equipped with a smooth {\em volume element}.
By this we mean a nowhere-vanishing smooth one-density.\index{volume element}
If $M$ is oriented, this is essentially the same thing as a nowhere vanishing $n$-form, where $n=\dim(M)$.
The volume element is needed to define the integral of functions over $M$.

\begin{definition}
A \emph{measured manifold}\index{measured manifold} is a pair
consisting of a manifold $M$ (possibly with boundary) 
and a smooth volume element $\mu$ on $M$.
\end{definition}

Let $M$ be a manifold with nonempty boundary $\dM$.
Let $T$ be a smooth vector field on $M$ along $\dM$
pointing into the interior of $M$.\index{1T@$T$}
In particular, $T$ does not vanish anywhere.
\begin{center}
\begin{pspicture}(0,-2.5)(8,2.5)
\psellipse[linewidth=0.04,dimen=outer](7.4138293,0.029666925)(0.5,1.46)
\psbezier[linewidth=0.04](7.3938293,1.5096669)(6.4738293,1.589667)(1.9938294,2.6296668)(1.2338294,2.249667)(0.47382936,1.8696669)(0.0,-1.6709992)(1.0738294,-2.1503332)(2.1476588,-2.6296668)(6.6338296,-1.3503331)(7.3938293,-1.430333)
\psbezier[linewidth=0.04](1.2938293,0.12966692)(1.2938293,-0.6703331)(3.6938293,-0.71033305)(3.6938293,0.089666925)
\psbezier[linewidth=0.04](1.5138294,-0.21251442)(2.6338294,0.60966694)(3.0738294,0.15500201)(3.4138293,-0.31033307)
\psline[linewidth=0.04cm,arrowsize=0.05291667cm 2.0,arrowlength=1.4,arrowinset=0.4]{->}(7.073829,1.1296669)(6.1738296,1.329667)
\psline[linewidth=0.04cm,arrowsize=0.05291667cm 2.0,arrowlength=1.4,arrowinset=0.4]{->}(7,0.7496669)(6.093829,0.80966693)
\psline[linewidth=0.04cm,arrowsize=0.05291667cm 2.0,arrowlength=1.4,arrowinset=0.4]{->}(6.95,0.28966692)(6.073829,0.2696669)
\psline[linewidth=0.04cm,arrowsize=0.05291667cm 2.0,arrowlength=1.4,arrowinset=0.4]{->}(6.9138293,0.029666925)(6.1538296,-0.29033306)
\psline[linewidth=0.04cm,arrowsize=0.05291667cm 2.0,arrowlength=1.4,arrowinset=0.4]{->}(6.9738293,-0.69033307)(6.4138293,-1.1303331)
\psline[linewidth=0.04cm,arrowsize=0.05291667cm 2.0,arrowlength=1.4,arrowinset=0.4]{->}(7.093829,-1.050333)(6.4738293,-1.3503331)
\psline[linewidth=0.04cm,arrowsize=0.05291667cm 2.0,arrowlength=1.4,arrowinset=0.4]{->}(6.9338293,-0.29033306)(6.3538294,-0.69033307)
\psline[linewidth=0.04cm,arrowsize=0.05291667cm 2.0,arrowlength=1.4,arrowinset=0.4]{->}(7.8538294,0.18966693)(7.1338296,0.18966693)
\psline[linewidth=0.04cm,arrowsize=0.05291667cm 2.0,arrowlength=1.4,arrowinset=0.4]{->}(7.8338294,0.6896669)(7.1338296,0.54966694)
\psline[linewidth=0.04cm,arrowsize=0.05291667cm 2.0,arrowlength=1.4,arrowinset=0.4]{->}(7.7138295,1.1296669)(7.113829,1.0096669)
\psline[linewidth=0.04cm,arrowsize=0.05291667cm 2.0,arrowlength=1.4,arrowinset=0.4]{->}(7.8738294,-0.11033308)(7.1338296,-0.07033308)
\psline[linewidth=0.04cm,arrowsize=0.05291667cm 2.0,arrowlength=1.4,arrowinset=0.4]{->}(7.8538294,-0.43033308)(7.113829,-0.33033308)
\psline[linewidth=0.04cm,arrowsize=0.05291667cm 2.0,arrowlength=1.4,arrowinset=0.4]{->}(7.7738295,-0.8103331)(7.093829,-0.69033307)
\psline[linewidth=0.04cm,arrowsize=0.05291667cm 2.0,arrowlength=1.4,arrowinset=0.4]{->}(7.7138295,-1.030333)(7.093829,-0.9903331)
\psline[linewidth=0.04cm](7.613829,-1.2303331)(7.1738296,-1.1903331)
\psline[linewidth=0.04cm](7.573829,1.4096669)(7.2538295,1.3696669)
\psline[linewidth=0.04cm,arrowsize=0.05291667cm 2.0,arrowlength=1.4,arrowinset=0.4]{->}(7.45,1.48)(6.553829,1.709667)
\psline[linewidth=0.04cm,arrowsize=0.05291667cm 2.0,arrowlength=1.4,arrowinset=0.4]{->}(7.3338294,-1.390333)(6.613829,-1.430333)

\rput(3.8552356,1.074667){$M$}
\rput(7.655236,-1.785333){$\dM$}
\rput(6.0152354,-0.50533307){$T$}
\end{pspicture} 
\nopagebreak 

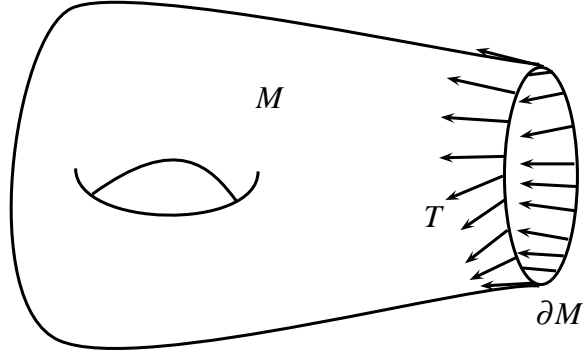
{\em Fig.~2}
\end{center}
At each $x\in\dM$ there is a unique $\tau(x)\in T_x^*M$\index{1Tau@$\tau$}
such that
\begin{equation}
  \<\tau(x),T(x)\> = 1 \quad\text{and}\quad \tau(x)|_{T_x\dM} = 0 .
  \label{tau}
\end{equation}
Here $\<\cdot,\cdot\>$ denotes the natural evaluation of $1$-forms on tangent vectors.
We call $\tau$ the {\em one-form associated to $T$}.\index{associated one-form}

If $\mu$ is a smooth volume element on $M$,
then $\dM$ inherits a smooth volume element $\nu$ such that along $\dM$
we have $\mu=|\tau|\otimes\nu$, 
i.e.,\index{volume element induced on boundary}
\begin{equation}
  \nu(X_1,\ldots,X_{n-1}) = \mu(T,X_1,\ldots,X_{n-1}) 
\label{eq:munu}
\end{equation}
for all vector fields $X_1,\ldots,X_{n-1}$ on $\dM$.
This turns $\dM$ into a measured manifold.

\begin{remark}\label{rem:Riemann}
If $M$ is a Riemannian manifold with boundary,
then the Riemannian volume element $\mu$ turns $M$
canonically into a measured manifold.
The natural choice for $T$ is the interior unit normal vector field along $\dM$. 
The induced volume element $\nu$ on $\dM$ then coincides
with the volume element of the induced Riemannian metric on $\dM$.

Conversely, given a smooth volume element $\mu$ on $M$
and an interior vector field $T$ along $\dM$,
one can always find a Riemannian metric on $M$ inducing $\mu$ and $T$ in this way.
\end{remark}

\begin{notation}
Throughout this article we will write\index{1ZI@$Z_I$}
$$
Z_I := I \times \dM
$$
where $I\subset\R$ is any interval.
We think of $Z_I$ as a cylinder over $\dM$.
\end{notation}

\begin{lemma}\label{adapted}
Let $(M,\mu)$ be a measured manifold with compact boundary
and let $\tau$ be the one-form  associated to an interior vector field $T$.

Then there is a neighborhood $U$ about $\dM$ in $M$,
a constant $r>0$, and a diffeomorphism $\Psi=(t,\psi):U \to Z_{[0,r)}$
such that
\begin{align}
  \dM &= t^{-1}(0) , \label{NN1} \tag{i} \\
  \psi|_\dM &=\id_\dM , \label{NN2} \tag{ii} \\
  d\Psi(T) &=\del/\del t \hspace{.95cm}\text{along $\dM$} , \label{NN3} \tag{iii} \\
  \tau &= dt  \hspace{1.45cm}\text{along $\dM$} , \label{NN4} \tag{iv} \\
  \Psi_*(\mu) &= |dt|\otimes \nu \hspace{.5cm}\text{on $Z_{[0,r)}$} . \label{NN5} \tag{v}
\end{align}
\end{lemma}

\begin{center}
\begin{pspicture}(-12,-2.5)(0,3.5)
\psscalebox{-1 1}
{
\pscustom*[linecolor=lightgray]{
\psellipticarc[linewidth=0.04,dimen=outer](7.4138293,0.029666925)(1.3,2){115}{242}
\psline(7.4138293,-1.42)(7.4138293,1.46)
}

\psellipse[linewidth=0.04,dimen=outer,fillstyle=solid,fillcolor=white](7.4138293,0.029666925)(0.5,1.46)
\pscustom*[linecolor=lightgray]{
\psellipticarc[linewidth=0.04,dimen=outer,linecolor=lightgray](7.4138293,0.029666925)(0.5,1.46){225}{125}
\pscurve(6.9,1.2)(7.15,0.1)(7.08,-1)
}
\psellipse[linewidth=0.04,dimen=outer](7.4138293,0.029666925)(0.5,1.46)
\psbezier[linewidth=0.04](7.3938293,1.5096669)(6.4738293,1.589667)(1.9938294,2.6296668)(1.2338294,2.249667)(0.47382936,1.8696669)(0.0,-1.6709992)(1.0738294,-2.1503332)(2.1476588,-2.6296668)(6.6338296,-1.3503331)(7.3938293,-1.430333)
\psbezier[linewidth=0.04](1.2938293,0.12966692)(1.2938293,-0.6703331)(3.6938293,-0.71033305)(3.6938293,0.089666925)
\psbezier[linewidth=0.04](1.5138294,-0.21251442)(2.6338294,0.60966694)(3.0738294,0.15500201)(3.4138293,-0.31033307)
\psline[linewidth=0.04cm,arrowsize=0.05291667cm 2.0,arrowlength=1.4,arrowinset=0.4]{->}(7.073829,1.1296669)(6.438296,1.329667)
\psline[linewidth=0.04cm,arrowsize=0.05291667cm 2.0,arrowlength=1.4,arrowinset=0.4]{->}(7,0.7496669)(6.23829,0.80966693)
\psline[linewidth=0.04cm,arrowsize=0.05291667cm 2.0,arrowlength=1.4,arrowinset=0.4]{->}(6.95,0.28966692)(6.133829,0.2196669)
\psline[linewidth=0.04cm,arrowsize=0.05291667cm 2.0,arrowlength=1.4,arrowinset=0.4]{->}(6.9138293,0.029666925)(6.1538296,-0.29033306)
\psline[linewidth=0.04cm,arrowsize=0.05291667cm 2.0,arrowlength=1.4,arrowinset=0.4]{->}(6.9738293,-0.69033307)(6.3138293,-1.0303331)
\psline[linewidth=0.04cm,arrowsize=0.05291667cm 2.0,arrowlength=1.4,arrowinset=0.4]{->}(7.093829,-1.050333)(6.4738293,-1.3503331)
\psline[linewidth=0.04cm,arrowsize=0.05291667cm 2.0,arrowlength=1.4,arrowinset=0.4]{->}(6.9338293,-0.29033306)(6.2538294,-0.69033307)
\psline[linewidth=0.04cm,arrowsize=0.05291667cm 2.0,arrowlength=1.4,arrowinset=0.4]{->}(7.8538294,0.18966693)(7.1338296,0.18966693)
\psline[linewidth=0.04cm,arrowsize=0.05291667cm 2.0,arrowlength=1.4,arrowinset=0.4]{->}(7.8338294,0.6896669)(7.1,0.54966694)
\psline[linewidth=0.04cm,arrowsize=0.05291667cm 2.0,arrowlength=1.4,arrowinset=0.4](7.7138295,1.1296669)(7.09,1.0096669)
\psline[linewidth=0.04cm,arrowsize=0.05291667cm 2.0,arrowlength=1.4,arrowinset=0.4]{->}(7.8738294,-0.11033308)(7.1338296,-0.07033308)
\psline[linewidth=0.04cm,arrowsize=0.05291667cm 2.0,arrowlength=1.4,arrowinset=0.4]{->}(7.8538294,-0.43033308)(7.113829,-0.33033308)
\psline[linewidth=0.04cm,arrowsize=0.05291667cm 2.0,arrowlength=1.4,arrowinset=0.4]{->}(7.7738295,-0.8103331)(7.093829,-0.69033307)
\psline[linewidth=0.04cm,arrowsize=0.05291667cm 2.0,arrowlength=1.4,arrowinset=0.4]{->}(7.7138295,-1.030333)(7.093829,-0.9903331)
\psline[linewidth=0.04cm](7.613829,-1.2303331)(7.1738296,-1.1903331)
\psline[linewidth=0.04cm](7.573829,1.4096669)(7.2538295,1.3696669)
\psline[linewidth=0.04cm,arrowsize=0.05291667cm 2.0,arrowlength=1.4,arrowinset=0.4]{->}(7.45,1.48)(6.653829,1.709667)
\psline[linewidth=0.04cm,arrowsize=0.05291667cm 2.0,arrowlength=1.4,arrowinset=0.4]{->}(7.3338294,-1.390333)(6.513829,-1.430333)

\psellipse[linewidth=0.02,dimen=outer,fillstyle=solid,fillcolor=lightgray](9,0)(0.5,1.46)
\psline*[linecolor=lightgray](9,1.46)(10,1.46)(10,-1.46)(9,-1.46)
\psline[linewidth=0.04,dimen=outer](9,1.46)(10,1.46)
\psline[linewidth=0.04,dimen=outer](9,-1.46)(10,-1.46)
\psellipse[linewidth=0.04,dimen=outer,fillstyle=solid,fillcolor=white](10,0)(0.5,1.46)
\pscustom*[linecolor=lightgray]{
\psellipticarc[linewidth=0.04,dimen=outer,linecolor=lightgray](10,0)(0.5,1.46){220}{130}
\pscurve(9.56,0.9)(9.7,0)(9.65,-0.9)
}
\pscurve[linewidth=0.04,dimen=outer](9.65,0.9)(9.7,0)(9.65,-0.9)
\psellipse[linewidth=0.04,dimen=outer](10,0)(0.5,1.46)

\psline[linewidth=0.04cm](8.68,1)(10.32,1)
\psline[linewidth=0.04cm](8.55,0.5)(9.55,0.5)
\psline[linewidth=0.04cm](9.7,0.5)(10.43,0.5)
\psline[linewidth=0.04cm](8.5,0)(9.54,0)
\psline[linewidth=0.04cm](9.7,0)(10.46,0)
\psline[linewidth=0.04cm](8.68,-1)(10.32,-1)
\psline[linewidth=0.04cm](8.55,-0.5)(9.55,-0.5)
\psline[linewidth=0.04cm](9.7,-0.5)(10.43,-0.5)

\pscurve{->}(7,1.8)(8,2.3)(9,1.8)
}
\rput(-3.8552356,1.074667){$M$}
\rput(-6.6,-0.50533307){\psframebox*[framearc=.3]{$U$}}
\rput(-9.5,-1.9){\psframebox*[framearc=.3]{$Z_{[0,r)}$}}
\rput(-8,2.6){$\Psi$}
\end{pspicture}
\nopagebreak 

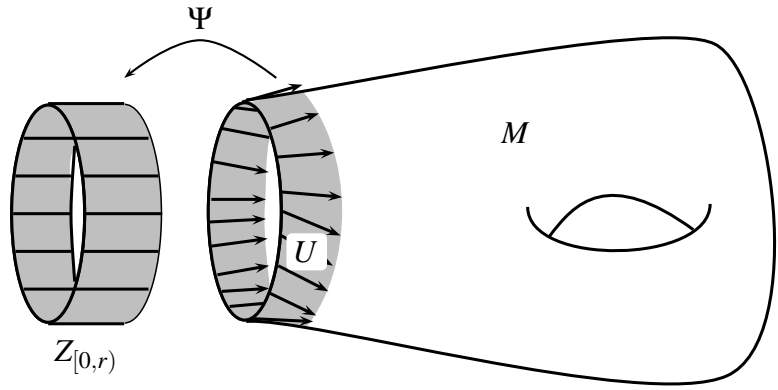
{\em Fig.~3}
\end{center}

\begin{proof}
Extend $T$ to a smooth vector field $T_1$ without zeros in a neighborhood of $\dM$ in $M$.
Solve for a smooth real function $f$ on a possibly smaller neighborhood such that
\begin{equation}
  0 = \div(f\cdot T_1) = df(T_1) + f\cdot \div(T_1)
  \label{eq:DivFrei}
\end{equation}
with $f_{|\dM} = 1$.
Note that the divergence is defined because $M$ carries a smooth volume element.
Since \eqref{eq:DivFrei} is an ordinary differential equation
along the integral curves of $T_1$, there is a unique solution $f$.

The vector field $fT_1$ is a smooth extension of $T$, which we denote again by $T$.
Let $\Phi$ be the flow of $T$.
Since $\dM$ is compact, there is a neighborhood $U$ of $\dM$ in $M$
and an $r>0$ such that 
\[
  Z_{[0,r)}  \to U, \quad
  (t,x) \mapsto \Phi_t(x) ,
\]
is a diffeomorphism.
Let $\Psi$ be the inverse of this diffeomorphism.
Properties \eqref{NN1}, \eqref{NN2}, and \eqref{NN3} are clear by construction
and they imply \eqref{NN4}.
Since $T$ is divergence free, its flow preserves $\mu$.
This shows \eqref{NN5}.
\end{proof}

\begin{definition}\label{def:adapted}
A diffeomorphism as in Lemma~\ref{adapted}
will be called {\em adapted to} $(M,\mu)$ and $T$.
\end{definition}

We will often identify a neighborhood $U$ of the boundary
with the cylinder $Z_{[0,r)}$ using an adapted diffeomorphism.
Property~\eqref{NN5} is not really necessary in our reasoning below,
but it simplifies the exposition.


\subsection{Vector bundles and differential operators}

Let $(M,\mu)$ be a measured manifold (possibly with boundary)
and $E \to M$ be a Hermitian vector bundle over $M$.

The space of smooth sections of $E$ is denoted by $\Cu(M,E)$.\index{1CME@$\Cu(M,E)$}
Here smoothness means smoothness up to the boundary, i.e.,
in local coordinates all derivatives have continuous extensions to the boundary.
Smooth sections with compact support
form the space $\Cu_c(M,E)$.\index{1CcME@$\Cuc(M,E)$}
Note that the support $\Phi\in\Cu_c(M,E)$ may intersect the boundary $\dM$.
The space of smooth sections with compact support contained in the interior of $M$
is denoted by $\Cu_{cc}(M,E)$.\index{1CccME@$\Cucc(M,E)$}
Obviously, we have 
\[
  \Cu_{cc}(M,E) \subset \Cu_c(M,E) \subset \Cu(M,E) .
\] 
The $L^2$-scalar product\index{L2-scalar product@$L^2$-scalar product}
of $\Phi,\Psi\in\Cu_c(M,E)$ is defined by 
\[
  (\Phi,\Psi)_{L^2(M)} := \int_M \<\Phi,\Psi\> d\mu ,
  \index{1L2scalarproduct@$(\cdot,\cdot)_{L^2(M)}$}
\]
the corresponding norm by
\[
  \|\Phi\|_{L^2(M)}^2 = (\Phi,\Phi)_{L^2(M)} = \int_M |\Phi|^2 d\mu .
  \index{1normL2@$\LzM{\cdot}$}
\]
Here $\<\cdot,\cdot\>$  and $|\cdot |$ denote the Hermitian product and norm of $E$.
The completion of $\Cu_c(M,E)$ with respect to $\LzM{\cdot}$
is denoted by $L^2(M,E)$.\index{1L2ME@$L^2(M,E)$}
Its elements are called {\em square integrable}\index{square integrable section} sections.
An alternative, but equivalent, definition of $L^2(M,E)$ would be the space
of all measurable sections with finite $L^2$-norm
modulo sections vanishing almost everywhere.
Measurable sections whose restrictions to compact subsets have finite $L^2$-norm
are called {\em locally square integrable}.\index{locally square integrable section}
The space of locally square integrable sections modulo sections vanishing almost
everywhere is denoted by $L^2_{\loc}(M,E)$.\index{1L2locME@$L^2_{\loc}(M,E)$}
If $M$ is compact, then, of course, $L^2(M,E)=L^2_{\loc}(M,E)$.
The $L^2$-norm and hence $L^2(M,E)$ depend on the volume element $\mu$,
while $L^2_{\loc}(M,E)$ does not because, over compact subsets of $M$,
any two smooth volume elements on $M$ can be bounded by each other.

Let $E\to M$ and $F\to M$ be Hermitian vector bundles over $M$
and $D$ be a (linear) differential operator from $E$ to $F$.
Associated with $D$, there is a unique differential operator 
\[
  D^*:\Cu(M,F) \to \Cu(M,E) ,
\]
called the {\em formal adjoint}\index{formal adjoint operator} of $D$,
such that 
\begin{equation}
  \int_M \<D\Phi,\Psi\> d\mu = \int_M \<\Phi,D^*\Psi\> d\mu
\label{eq:adjungierterOperator}
\end{equation}
for all sections $\Phi\in\Cucc(M,E)$ and $\Psi\in\Cucc(M,F)$.
Locally, $D^*$ is obtained from $D$ by integration by parts.
Clearly, we have $D^{**}=D$.

Let $D_c$ and $D_{cc}$ be the operator $D$,
considered as an unbounded operator on the Hilbert space $L^2(M,E)$
with domain $\dom(D_c)=\Cuc(M,E)$ and $\dom(D_{cc})=\Cucc(M,E)$, respectively,
and similarly for $D^*$.
Note that $\dom(D_{c})$ and $\dom(D_{cc})$ are dense in $L^2(M,E)$
and that $\dom(D_{cc})$ is contained in $\dom(D_{c})$.
We write $D_{cc}\subset D_{c}$ to express the latter fact.

Suppose that, as additional data, we are given a Riemannian metric $g$ on $M$
and a Hermitian connection $\nabla$ on $E$.
For any $\Phi\in\Cu(M,E)$ and integer $k\ge0$,
we then get the $k^{th}$ covariant derivative
\[
  \nabla^k\Phi \in
  \Cu(M,\underbrace{T^*M\otimes \ldots\otimes T^*M}_{\mbox{$k$ times}}\otimes E) ,
\]
where both, the Levi-Civita connection of $g$ and $\nabla$,
are used in the definition of the higher covariant derivatives of $\Phi$.
Together with $g$, the Hermitian metric on $E$ induces a metric
on $T^*M\otimes \ldots\otimes T^*M\otimes E$
so that we obtain the formal adjoint $(\nabla^k)^*$
of the differential operator $\nabla^k$.

For a section $\Phi\in L^2_{\loc}(M,E)$, 
we call $\Psi\in L^2_{\loc}(M,T^*M\otimes \ldots\otimes T^*M\otimes E)$ the
{\em $k$-th weak covariant derivative}\index{weak covariant derivative} of $\Phi$ if 
\[
  \LzMp{\Psi}{\Xi} = \LzMp{\Phi}{(\nabla^k)^*\Xi}
\] 
for all $\Xi\in\Cucc(M,T^*M\otimes \ldots\otimes T^*M\otimes E)$.
If $\Psi$ exists, it is uniquely determined and we write $\Psi= \nabla^k\Phi$
(instead of $\Psi=\nabla^k_{\max}\Phi$).
The space of $\Phi\in L^2_{\loc}(M,E)$,
whose weak covariant derivatives up to order $k$ exist in $L^2_{\loc}(M,E)$,
is denoted by $\Hkloc(M,E)$.\index{1HklocME@$\Hkloc(M,E)$}
We have the inclusions
\[
  \Cu(M,E) \subset \Hkloc(M,E) \subset H^{k-1}_{\loc}(M,E)
  \subset \cdots \subset H^0_{\loc}(M,E) = L^2_{\loc}(M,E)
\]
and, by the Sobolev embedding theorem,
\[
  \bigcap_{k=0}^\infty H^k_{\loc}(M,E) = \Cu(M,E) .
\]
The space of $\Phi\in L^2(M,E)$,
whose weak covariant derivatives up to order $k$ exist in $L^2(M,E)$,
is denoted by $H^k(M,E)$.\index{1HkME@$H^k(M,E)$}
It is a Hilbert space with respect to the {\em Sobolev norm}\index{Sobolev norm}
$\HkM{\cdot}$\index{1normHk@$\HkM{\cdot}$}
defined by 
\begin{equation}
  \HkM{\Phi}^2 = \LzM{\Phi}^2 + \LzM{\nabla\Phi}^2 + \dots + \LzM{\nabla^k\Phi}^2 .
  \label{eq:DefSobo}
\end{equation}
If $M$ is noncompact, then $H^k(M,E)$ depends
on the choice of $g$ and $\nabla$
(given the smooth volume element $\mu$ and the Hermitian vector bundle $E$).
If $M$ is compact, then any two $H^k$-norms (for the same $k$)
are equivalent so that $H^k(M,E)$ is independent of these choices.

For $M$ compact (possibly with boundary),
the {\em Rellich embedding theorem}\index{Rellich embedding theorem}
\cite[Theorem~2.34, p.~55]{Au} says in particular that the embedding
\[
  H^{k+1}(M,E) \hookrightarrow H^k(M,E)
\]
is compact; in other words,
bounded sequences in $H^{k+1}(M,E)$ subconverge in $H^k(M,E)$.

We collect some of the spaces introduced so far in the following table:
\begin{center}
\begin{tabular}{|c|c|}
\hline
space of & notation \\
\hline
\hline
smooth sections & $\Cu(M,E)$ \\
\hline
smooth sections with compact support & $\Cuc(M,E)$ \\
\hline
smooth sections with compact support & \\
contained in the interior of $M$ & \raisebox{1.5ex}[-1.5ex]{$\Cucc(M,E)$} \\
\hline
square integrable sections & $L^2(M,E)$ \\
\hline
locally square integrable sections & $L^2_{\loc}(M,E)$ \\
\hline
sections in $L^2_{\loc}(M,E)$ with first $k$  & \\
weak derivatives in $L^2_{\loc}(M,E)$ & \raisebox{1.5ex}[-1.5ex]{$\Hkloc(M,E)$} \\
\hline
\end{tabular}
\nopagebreak

{\em Table~1}
\end{center}

We can restrict the bundle $E$ to $\dM$ and consider the corresponding spaces
such as $\Cu(\dM,E)$, $L^2(\dM,E)$ etc.
Further spaces of sections will be defined as needed.

Suppose from now on that $D$ is a differential operator from $E$ to $F$ of order one
and denote the {\em principal symbol}\index{principal symbol} of $D$ by $\sigma_D$.
For any $x\in M$, $\sigma_D(x):T_x^*M \to \mathrm{Hom}(E_x,F_x)$
is a linear map which is characterized by the property that
\begin{equation}
  D(f\Phi) = fD\Phi + \sigma_D(df)\Phi ,
\label{eq:DefHauptsymbol}
\end{equation}
for all $f\in\Cu(M)$ and $\Phi\in\Cu(M,E)$.
For all $\xi\in T^*M$,
we have\symbolfootnote[2]{In order to get rid of the sign in \eqref{eq:symboladjoint},
a factor $i$ is often included in the definition of the principal symbol.} 
\begin{equation}
\sigma_{D^*}(\xi) = - \sigma_D(\xi)^* .
\label{eq:symboladjoint}
\end{equation}
Equation \eqref{eq:adjungierterOperator} holds if the supports of $\Phi$ and $\Psi$
are compact and contained in the interior of $M$.
In case they meet the boundary, there is an additional boundary term
involving the principal symbol of $D$:

\begin{lemma}[Green's formula]\label{lem:Green}\index{Green's formula}
Let $(M,\mu)$ be a measured manifold with boundary and let $\tau$
be the one-form associated to an interior vector field.
Then
\[
  \int_M \<D\Phi,\Psi\> d\mu - \int_M \<\Phi,D^*\Psi\> d\mu
  = 
  -\int_{\dM}\<\sigma_D(\tau)\Phi,\Psi\>d\nu
\]
for all $\Phi\in \Cu(M,E)$ and $\Psi\in \Cu(M,F)$
such that $\supp(\Phi)\cap\supp(\Psi)$ is compact.
\end{lemma}

\begin{proof}
Choose a Riemannian metric on $M$ inducing $\mu$ and $T$
as in Remark~\ref{rem:Riemann}.
Now the lemma follows from the standard Green's formula
for Riemannian manifolds, see e.g.\ \cite[Prop.~9.1, p.~160]{Ta}.
\end{proof}

We say that $D$ is \emph{elliptic}\index{elliptic differential operator}
if $\sigma_D(\xi)$ is invertible for each nonzero covector $\xi$.
It the boundary of $M$ is empty and $D$ is elliptic (of order one), then
\begin{equation}
  D_{\max}\Phi\in H^k_{\loc}(M,F)
  \Longrightarrow
  \Phi\in H^{k+1}_{\loc}(M,E)
  \label{intelreg}
\end{equation}
for all $\Phi\in \dom(D_{\max})$,
by {\em interior elliptic regularity theory},\index{interior elliptic regularity}
see e.g. \cite[Ch.~III, \S~5]{LM}.


\section{Completeness}
\label{sectcom}


Let $(M,\mu)$ be a measured manifold with compact boundary. 
Let $E,F\to M$ be Hermitian vector bundles
and $D:\Cu(M,E)\to \Cu(M,F)$ be a differential operator of first order.
We start by generalizing Equation \eqref{eq:DefHauptsymbol}
to Lipschitz functions and sections in the maximal domain of $D$:

\begin{lemma}\label{product}
Let $\chi:M\to\R$ be a Lipschitz function with compact support
and $\Phi\in\dom(D_{\max})$.
Then $\chi\Phi\in\dom(D_{\max})$ and 
\[
  D_{\max}(\chi\Phi) = \sigma_D(d\chi) \Phi + \chi D_{\max}\Phi . 
\]
\end{lemma}

\begin{proof}
Suppose $\Psi\in H^1_{\loc}(M,F)$ has compact support in the interior of $M$.
Let $K$ be a compact subset in the interior of $M$
which contains the support of $\Psi$ in its interior.
Then there is a sequence of $\Psi_j$ in $\Cucc(M,F)$ with supports in $K$
which converge to $\Psi$ in the $H^1$-norm (over $K$).
In particular, $D^*\Psi$ is well defined and $\lim_{n\to\infty}D^*\Psi_n=D^*\Psi$
with respect to the $L^2$-norm.
We conclude that, for any $\Phi\in\dom(D_{\max})$,
\[
  (\Phi,D^*\Psi) = \lim_{n\to\infty}(\Phi,D^*\Psi_n)
  = \lim_{n\to\infty}(D_{\max}\Phi,\Psi_n)
  = (D_{\max}\Phi,\Psi) .
\]
Assume now that $\chi:M\to\R$ is a Lipschitz function with compact support, 
and let $\Phi\in\dom(D_{\max})$.
For any $\Psi\in\Cucc(M,F)$ we have $\chi\Psi\in H^1_{\loc}(M,F)$,
the support of $\chi\Psi$ is compact and contained in the interior of $M$,
and $D^*(\chi\Psi)=\sigma_{D^*}(d\chi)\Psi+\chi D^*\Psi$ in $L^2(M,F)$,
as we see by approximating $\chi$ in the $H^1$-norm by smooth functions
with compact support.
Hence
\begin{align*}
  (\chi\Phi,D^*\Psi)_{L^2(M)}
  &= (\Phi,\chi D^*\Psi)_{L^2(M)} \\
  &= (\Phi,D^*(\chi\Psi)- \sigma_{D^*}(d\chi)\Psi)_{L^2(M)} \\
  &= (\Dmax\Phi,\chi\Psi)_{L^2(M)} + (\sigma_{D}(d\chi)\Phi,\Psi)_{L^2(M)} \\
  &= (\chi \Dmax\Phi+ \sigma_{D}(d\chi)\Phi,\Psi)_{L^2(M)} .
  \qedhere
\end{align*}
\end{proof}

Recall from Definition~\ref{complete}
that $D$ is complete if and only if the space of compactly supported sections in $\dom(D_{\max})$ is dense in $\dom(D_{\max})$.

\begin{remark}\label{minmax}
If $\dM=\emptyset$ and $D$ is elliptic, then $D$ is complete if and only if
the minimal and maximal extensions of $D$ on $\Cuc(M,E)=\Cucc(M,E)$ coincide.
Namely, if the extensions coincide, then $\Cucc(M,E)$ is dense in $\dom(D_{\max})$.
Conversely, since $D$ is elliptic, $\dom(D_{\max})$ is contained in $H^1_{\loc}(M,E)$,
by interior elliptic regularity.
Furthermore, any compactly supported $H^1$-section can be approximated
by smooth sections with support contained in a fixed compact domain
in the $H^1$-norm and hence in the graph norm of $D$.
\end{remark}

\begin{thm}\label{comsup}
Let $(M,\mu)$ be a measured manifold with compact boundary
and $D:\Cu(M,E)\to \Cu(M,F)$ be a differential operator of first order.
Suppose that there exists a constant $C>0$
and a complete Riemannian metric on $M$ with respect to which 
\[
  |\sigma_D(\xi)| \leq C\,|\xi|
\]
for all $x\in M$ and $\xi\in T_x^*M$.
Then $D$ and $D^*$ are complete.
\end{thm}

\begin{proof}[Proof of \tref{comsup}]
Let $r:M\to\R$ be the distance function from the boundary, $r(x)=\mathrm{dist}(x,\dM)$.
Then $r$ is a Lipschitz function with Lipschitz constant $1$.
Choose $\rho\in\Cu(\R,\R)$ so that $0\le\rho\le 1$, $\rho(t)=0$ for $t\ge2$,
$\rho(t)=1$ for $t\le1$, and $|\rho'|\le 2$.
Set
\[
  \chi_m(x) := \rho\left(\frac{r(x)}{m}\right).
\]
Then $\chi_m$ is a Lipschitz function and we have almost everywhere
\[
  |d\chi_m(x)| \le \frac{2}{m} .
\]
Moreover, $\{\chi_m\}_m$ is a uniformly bounded sequence of functions
converging pointwise to $1$.

Now let $\Phi\in\dom(D_{\max})$.
Then $\| \chi_m\Phi - \Phi \|_{L^2(M)} \to 0$ as $m\to\infty$ by Lebesgue's theorem.
Furthermore, $\chi_m\Phi$ has compact support
and $\chi_m\Phi\in\dom(D_{\max})$ by \lref{product}. 
Since
\begin{align*}
  \| D_{\max}(\chi_m\Phi) &- D_{\max}\Phi \|_{L^2(M)} \\
  &\le \| (1-\chi_m)D_{\max}\Phi\|_{L^2(M)} + \| \sigma_D(d\chi_m)\Phi \|_{L^2(M)} \\
  &= \| (1-\chi_m)D_{\max}\Phi\|_{L^2(M)} + \frac{2C}{m}\,\| \Phi \|_{L^2(M)} 
  \to 0
\end{align*}
as $m\to\infty$, we conclude that $\chi_m\Phi \to \Phi$ in the graph norm of $\Dmax$.

The same discussion applies to $D^*$
because $|\sigma_{D^*}(\xi)|=|-\sigma_{D}(\xi)^*|=|\sigma_{D}(\xi)|$,
and the theorem is proved.
\end{proof}

\begin{example}\label{diracom}
If $D$ is of Dirac type with respect to a complete Riemannian metric,
see Example~\ref{exabso}~(a), then $D$ is complete.
Namely, by the Clifford relations \eqref{Cliff1}
we have $\sigma_D(\xi)^*\sigma_D(\xi) = |\xi|^2\cdot\id$
and hence $|\sigma_D(\xi)|=|\xi|$.
\end{example}

The assumption in \tref{comsup} that the principal symbol of $D$ is uniformly bounded
can be weakened to the condition considered by Chernoff
in \cite[Theorems~1.3 and 2.2]{Ch}:

\begin{proof}[Proof of \tref{cherwolf}]
Choose a smooth function $f:M\to\R$ with
\[
 C(\mathrm{dist}(x,\dM)) \leq f(x) \le 2\, C(\mathrm{dist}(x,\dM))
\]
for all $x\in M$.
Let $g$ denote the complete Riemannian metric as in the assumptions of \tref{cherwolf}.
Then $h:=f^{-2}g$ is also complete because the $h$-length
of a curve $c:[0,\infty)\to M$,
starting in $\dM$ and parametrized by arc-length with respect to $g$,
is estimated by
\begin{align*}
  L_h(c) 
  &= 
  \int_0^\infty \frac{|c'(t)|_g}{f(c(t))}\, dt \\
  &\ge 
  \frac12 \int_0^\infty \frac{|c'(t)|_g}{C(\mathrm{dist}(c(t),\dM))}\, dt \\
  &\ge
  \frac12 \int_0^\infty \frac{1}{C(t)}\, dt \\
  &=
  \infty .
\end{align*}
With respect to $h$,
the principal symbol $\sigma_D$ is uniformly bounded as required in \tref{comsup}.
\end{proof}

\begin{remark}\label{cylicom}
It will be convenient to assume further on that $D$ is complete.
However, in questions concerning boundary regularity,
assuming completeness is somewhat artificial because it is a property ``at infinity''.
But, in such questions, we can always pass to a complete
differential operator on vector bundles over the cylinder $Z_{[0,\infty)}$
which coincides with the given operator in a neighborhood of the boundary.
In this sense, assuming completeness causes no loss of generality
when studying the operator near the boundary.
\end{remark}


\section{Normal form}
\label{secnf}


Throughout this section,
let $(M,\mu)$ be a given measured manifold with compact boundary.
Let $T$ be an interior vector field along $\dM$ and $\tau$
be the associated one-form.
Identify a neighborhood $U$ of the boundary $\dM$ with a cylinder $Z_{[0,r)}$
via an adapted diffeomorphism as in \lref{adapted}.

Let $E,F\to M$ be Hermitian vector bundles.
Identify the restrictions of $E$ and $F$ to $Z_{[0,r)}$ as Hermitian vector bundles
with the pull-back of their restriction to $\dM$ with respect to the canonical projection
onto $\dM$ along the family of $t$-lines $(t,x)$, $0\le t<r$ and $x\in\dM$.
This can be achieved by using parallel transport along the $t$-lines
with respect to Hermitian connections on $E$ and $F$.
Then sections of $E$ and $F$ over $Z_{[0,r)}$ can be viewed
as $t$-dependent sections of $E$ and $F$ over $\dM$.
Using this identification we have, by \eqref{NN5} of \lref{adapted},
\[
  \int_{Z_{[0,r)}} |\Phi(p)|^2 d\mu(p)
  = \int_0^r \int_\dM |\Phi(t,x)|^2 d\nu(x)\, dt
\]
for any $\Phi\in L^2(Z_{[0,r)},E)$ or $\Phi\in L^2(Z_{[0,r)},F)$.

Let $D:\Cu(M,E) \to \Cu(M,F)$ be an elliptic differential operator
of first order and set
\begin{equation}\label{ddt}
  \sigma := \sigma_D(dt) .
\end{equation}
By \eqref{NN4} of \lref{adapted}, 
we have $\sigma_{(0,x)}=\sigma_D(\tau(x))$ for each $x\in\dM$.
Since $D$ is elliptic, $\sigma_{(t,x)}:E_{x}\to F_{x}$ 
is an isomorphism for each $(t,x)\in Z_{[0,r)}$.
We usually suppress the $x$-dependence in the notation
and write $\sigma_t$\index{1Sigmat@$\sigma_t$} instead of $\sigma_{(t,x)}$.

The main point about boundary symmetric operators as in Definition~\ref{defbs} is that,
in coordinates adapted to $(M,\mu)$ and $T$ as in \lref{adapted},
$D$ and $D^*$ admit good normal forms near the boundary:

\begin{lemma}[Normal form]\label{nf}\index{normal form near boundary}
Let $D:\Cu(M,E) \to \Cu(M,F)$ be an elliptic differential operator of first order.
If $D$ is boundary symmetric,
then there are formally selfadjoint elliptic differential operators
\[
  A:\Cu(\dM,E) \to \Cu(\dM,E) \index{1A@$A$}
  \quad\text{and}\quad
  \tilde A:\Cu(\dM,F) \to \Cu(\dM,F) \index{1Atilde@$\tilde A$}
\]
such that, over $Z_{[0,r)}$,
\begin{align}
  D
  &=
  \sigma_t\left(\frac{\del}{\del t} + A + R_t\right) ,
  \label{eq:Dnormalform}\\
  D^*
  &=
  -\sigma_t^*\left(\frac{\del}{\del t} + \tilde A + \tilde R_t\right) .
  \label{eq:Dstarnormalform}
\end{align}
The remainders
\[
  R_t:\Cu(\dM,E) \to \Cu(\dM,E) \index{1Rt@$R_t$}
  \quad\text{and}\quad
  \tilde R_t:\Cu(\dM,F) \to \Cu(\dM,F) \index{1Rttilde@$\tilde R_t$}
\]
are families of differential operators of order at most one
whose coefficients depend smoothly on $t\in [0,r)$.
They satisfy an estimate
\begin{align}
  \| R_t\Phi \|_{L^2(\dM)}
  &\le
  C \left( t \| A\Phi \|_{L^2(\dM)} + \| \Phi \|_{L^2(\dM)} \right) , \label{Rt}\\
  \| \tilde R_t\Psi \|_{L^2(\dM)}
  &\le
  C \left( t \| \tilde A\Psi \|_{L^2(\dM)} + \| \Psi \|_{L^2(\dM)} \right) , \label{Rtt}
\end{align}
for all $\Phi\in\Cu(\dM,E)$ and $\Psi\in\Cu(\dM,F)$.
\end{lemma}

\begin{proof}
For $x\in\dM$,
identify $T_x^*\dM$ with the subspace of $\xi\in T_x^*M$ such that $\xi(T)=0$.
In this sense, the principal symbol of the desired operator $A$ is given by
\begin{equation}
  \sigma_{A}(\xi)
  = \sigma_0(x)^{-1}\circ\sigma_D(\xi) ,
\label{eq:symbDT}
\end{equation}
by \eqref{eq:Dnormalform}, \eqref{Rt}, and the definition of $\sigma$.
Since $D$ is boundary symmetric, $\sigma_0(x)^{-1}\circ\sigma_D(\xi)$
is skew-Hermitian for all $x\in\dM$ and $\xi\in T_x^*\dM$.
Hence we can choose a formally selfadjoint differential operator
$A:\Cu(\dM,E) \to \Cu(\dM,E)$ of order one
with principal symbol as required by \eqref{eq:symbDT}.
Since the principal symbol is composed of invertible symbols, $A$ is elliptic.

Over $Z_{[0,r)}$, we have
\[
  D = \sigma_t \big( \frac{\del}{\del t} + \D_t \big) ,
\]
where $\D_t : \Cu(\dM,E) \to \Cu(\dM,E)$ 
is a family of elliptic differential operators of order one
whose coefficients depend smoothly on $t\in [0,r)$.
Hence
\[
  R_t := \D_t - A
\]
is a family of differential operators of order at most one
whose coefficients depend smoothly on $t$.
Since $\D_0$ and $A$ have the same principal symbol, $R_0$ is of order $0$.
Since $\dM$ is a closed manifold, we conclude that
\[
  \| R_t\Phi \|_{L^2(\dM)}
  \le C'  \big( t \| \Phi \|_{H^1(\dM)} + \| \Phi \|_{L^2(\dM)} \big)
\]
for some constant $C'$.
Now $A$ is elliptic of order one.
Hence, by standard elliptic estimates,
the $H^1$-norm is bounded by the graph norm of $A$;
that is, we have an estimate as claimed in \eqref{Rt}.
This finishes the proof of the assertions concerning $D$.

By \eqref{eq:symboladjoint},
the principal symbol of the desired operator $\tilde A$ is given by
\begin{align}
  \sigma_{\tilde A}(\xi)
  &= \sigma_{D^*}(\tau(x))^{-1}\circ\sigma_{D^*}(\xi)
  \label{symboltildea} \\
  &= (\sigma_{D}(\tau(x))^*)^{-1}\circ\sigma_{D}(\xi)^* \notag \\
  &= (\sigma_{0}(x)^*)^{-1}\circ\sigma_{D}(\xi)^* \notag \\
  &= \big(\sigma_{D}(\xi)\circ\sigma_{0}(x)^{-1}\big)^* , \notag
\end{align}
which is also skew-Hermitian.
Thus the analogous arguments as above show the assertions concerning $D^*$.
\end{proof}

\begin{remarks}
(a)
Conversely, it is immediate that an elliptic differential operator $D$
with a normal form as in \lref{nf} is boundary symmetric.
In particular, $D$ is boundary symmetric if and only if $D^*$ is boundary symmetric.
The latter is also obvious from Equation \eqref{eq:symboladjoint}
(as we see from the end of the above proof).

(b)
The operators $A$ and $\tilde A$ in \lref{nf} are not unique.
One can add symmetric zero-order terms to them by paying
with a corresponding change of the remainder-terms $R_t$ and $\tilde R_t$.
Equations \eqref{Rt} and \eqref{Rtt} will still be valid for the modified remainder terms.
\end{remarks}

We conclude this section with a few examples.

\begin{examples}\label{exabso}
(a) (Dirac type operators)
We say that the operator $D$ is of {\em Dirac type}\index{Dirac type operator}
if $M$ carries a Riemannian metric $\<\cdot,\cdot\>$ such that the principal symbol
of $D$ satisfies the Clifford relations\index{Clifford relations}
\begin{align}
  \sigma_D(\xi)^*\sigma_D(\eta) + \sigma_D(\eta)^*\sigma_D(\xi)
  &= 2\<\xi,\eta\> \cdot \id_{E_x} ,
  \label{Cliff1} \\
  \sigma_D(\xi)\sigma_D(\eta)^* + \sigma_D(\eta)\sigma_D(\xi)^*
  &= 2\<\xi,\eta\> \cdot \id_{F_x} ,
  \label{Cliff2}
\end{align}
for all $x\in M$ and $\xi,\eta \in T^*_xM$.
They easily imply that $D$ is elliptic and boundary symmetric
with respect to the interior normal field of the given Riemannian metric.
If $D$ is of Dirac type, then so is $D^*$.

The class of Dirac type operators contains in particular Dirac operators
on Dirac bundles as in  \cite[Ch.~II, \S~5]{LM}.
The classical Dirac operator on a spin manifold is an important special case.

(b)
A somewhat artificial example of a boundary symmetric operator
which is not of Dirac type can be constructed as follows.
Let $g$ and $g'$ be two Riemannian metrics on a manifold $M$ with spin structure.
Let $D:\Cu(M,E) \to \Cu(M,F)$ and $D':\Cu(M,E') \to \Cu(M,F')$
be the corresponding Dirac operators acting on spinors.
Assume that $g$ and $g'$ are conformal along the boundary,
i.e., $g'=fg$ for some smooth positive function $f$ on $\dM$.
Let $T$ be an interior normal vector field along $\dM$,
perpendicular to $\dM$ for $g$ and $g'$.
Then 
\[
\begin{pmatrix}
D & 0 \cr 0 & D'
\end{pmatrix}
+ V
: \Cu(M,E\oplus E') \to \Cu(M,F\oplus F')
\]
is boundary symmetric (with respect to $T$) but in general not of Dirac type.
Here $V$ may be an arbitrary zero-order term.

(c)
More importantly, let $D$ be a Dirac type operator.
If one changes $D$ in the interior of $M$ in such a way that it remains
an elliptic first-order operator, then $D$ is still boundary symmetric.
\end{examples}


\section{The model operator}
\label{secMOP}


Throughout this section, assume the Standard Setup~\ref{stase},
identify a neighborhood $U$ of the boundary $\dM$ with a cylinder $Z_{[0,r)}$
via an adapted diffeomorphism as in \lref{adapted},
and fix a normal form for $D$ and $D^*$ as in \lref{nf}.
Consider $A$ as an unbounded operator in the Hilbert space $L^2(\dM,E)$
with domain $\dom(A)=C^\infty(\dM,E)$.
The \emph{model operator}\index{model operator} associated to $D$ and $A$
is the operator
\begin{equation}\label{moop}\index{1D0@$D_0$}
  D_0 := \sigma_0 \left( \frac{\del}{\del t} + A \right)
\end{equation}
on the half-infinite cylinder $Z_{[0,\infty)}$.
Here $\sigma_0(x)=\sigma_D(\tau(x))$, compare \lref{nf}.
The coefficients of $D_0$ do not depend on $t$.
With respect to the product measure $\mu_0:=dt\otimes\nu$ on $Z_{[0,\infty)}$,
we have
\begin{equation}\label{d0ad}
  D_0^* 
  = - \sigma_0^* \big( \frac{\del}{\del t} - (\sigma_0^*)^{-1} A \sigma_0^* \big)
  \quad\text{and}\quad
  (\sigma_0^{-1}D_0)^*
  = -   \big( \frac{\del}{\del t} - A \big) .
\end{equation}
We will keep the above setup and will use the abbrevation $\Phi' := \partial \Phi/\partial t$.

\begin{lemma}\label{intpart}
For any $\Phi\in C^\infty_c(Z_{[0,\infty)},E)$, we have
\begin{align*}
\| \sigma_0^{-1}D_0\Phi \|_{L^2(Z_{[0,\infty)})}^2
&= 
\| \Phi' \|_{L^2(Z_{[0,\infty)})}^2
  + \| A\Phi \|_{L^2(Z_{[0,\infty)})}^2  
  - (A\Phi_0,\Phi_0)_{L^2(\dM)} ,
\end{align*}
where $\Phi_0$ denotes the restriction of $\Phi$ to $\{0\}\times\dM$.
\end{lemma}

\begin{proof}
We fix $t\in [0,\infty)$ and integrate over $\dM$:
\begin{align*}
  \LzdM{\sigma_0^{-1} D_0&\Phi}^2 \\
  &= ( \Phi'+A\Phi,\Phi'+A\Phi )_{L^2(\dM)} \\
  &= \LzdM{\Phi'}^2 + \LzdM{A\Phi}^2
  + ( \Phi',A\Phi )_{L^2(\dM)} + ( A\Phi,\Phi' )_{L^2(\dM)} \\
  &= \LzdM{\Phi'}^2 + \LzdM{A\Phi}^2
  + ( A\Phi',\Phi )_{L^2(\dM)} + ( A\Phi,\Phi' )_{L^2(\dM)} \\
  &= \LzdM{\Phi'}^2 + \LzdM{A\Phi}^2
  +  (A\Phi,\Phi )_{L^2(\dM)}' .
\end{align*}
Here we used that $A$ does not depend on $t$ and that it is formally selfadjoint.
Now we integrate this identity with respect to $t\in [0,\infty)$.
Since $\Phi$ vanishes for sufficiently large $t$,
the last term gives a boundary contribution only for $t=0$.
\end{proof}

\begin{lemma}\label{lem:relbnd}
There exist constants $r,C>0$ such that 
\[
  \LzZq{(D-D_0)\Phi} \le C(\rho\,\LzZq{D_0\Phi}+\LzZq{\Phi})
\]
for all $0<\rho <r$ and all $\Phi\in \Cu(Z_{[0,\infty)},E)$ with support
in $Z_{[0,\rho]}$.
\end{lemma}

\begin{proof}
We have, by \lref{nf},
\begin{align*}
  D-D_0 
  &= \sigma_t\left(\frac{\del}{\del t}+ A + R_t\right)
  - \sigma_0 \left(\frac{\del}{\del t}+A\right) 
  &= (\sigma_t -\sigma_0)\sigma_0^{-1}D_0 + \sigma_t R_t .
\end{align*}
Since $(\sigma_t -\sigma_0)\sigma_0^{-1} = \mathrm{O}(t)$
and $\sigma$ is uniformly bounded, it remains to estimate $R_t\Phi$.
By \eqref{Rt} it suffices to estimate
\[
  t\|A\Phi\|_{L^2(\dM)}=\|A(t\Phi)\|_{L^2(\dM)} .
\]
Now $t\Phi$ is in $C^\infty_c(Z_{[0,\infty)},E)$ and vanishes at $t=0$.
Therefore, by \lref{intpart},
\[
  \|\sigma_0^{-1} D_0(t\Phi) \|_{L^2(Z_{[0,\infty)})}^2
  = \| (t\Phi)' \|_{L^2(Z_{[0,\infty)})}^2
  + \| A(t\Phi) \|_{L^2(Z_{[0,\infty)})}^2 .
\]
Hence
\begin{align*}
  \LzdM{A(t\Phi)}
  &\leq \|\sigma_0^{-1} D_0(t\Phi) \|_{L^2(Z_{[0,\infty)})} \\
  &\leq C' \big( \|tD_0\Phi \|_{L^2(Z_{[0,\infty)})} + \|\Phi \|_{L^2(Z_{[0,\infty)})} \big) .
\end{align*}
The asserted inequality follows.
\end{proof}

Since $\dM$ is compact without boundary, the minimal and maximal extensions
of the operator $A$ coincide.
Hence $A$ is essentially selfadjoint in the Hilbert space $L^2(\dM,E)$.
For any $s\in\R$, the positive operator $(\id + A^2)^{s/2}$ is defined by functional calculus. 

\begin{definition}\label{def:Hs}
For any $s\in\R$, we define the {\em Sobolev $H^s$-norm}\index{Sobolev norm}
on $\Cu(\dM,E)$ by
\[
  \HsdM{\phi}^2 := \LzdM{(\id + A^2)^{s/2}\phi}^2 .\index{1normHs@$\HsdM{\cdot}$}
\]
We denote by $H^{s}(\dM,E)$\index{1HsdME@$H^{s}(\dM,E)$}
the completion of $\Cu(\dM,E)$ with respect to this norm.
\end{definition}

By standard elliptic estimates, this norm is equivalent to the Sobolev norms
defined in \eqref{eq:DefSobo} if $s\in\N$.
It is a nice feature of \dref{def:Hs} that it makes sense for all $s\in\R$.
The values $s=1/2$ and $s=-1/2$ will be of particular importance.
Let 
\[
  -\infty \leftarrow \cdots \leq \lambda_{-2} \leq \lambda_{-1}
  \leq \lambda_0 \leq \lambda_1 \leq \lambda_2 \leq \cdots \to +\infty
\] 
be the spectrum of $A$ with each eigenvalue being repeated
according to its (finite) multiplicity,
and fix a corresponding $L^2$-orthonormal basis $\phi_j$, $j\in\Z$, of eigensections of $A$.
Then, for $\phi = \sum_{j=-\infty}^\infty a_j\phi_j$, one has
\[
  \HsdM{\phi}^2 = \sum_{j=-\infty}^\infty |a_j|^2 (1+\lambda_j^2)^{s}.
\]

\begin{facts}
The following facts are basic in our considerations:
\begin{enumerate}[(i)]
\item 
$H^0(\dM,E) = L^2(\dM,E)$;
\item
if $s<t$, then $\HsdM{\phi} \le \|\phi\|_{H^t(\dM)}$ and,
by the Rellich embedding theorem,
the induced embedding $H^t(\dM,E) \hookrightarrow H^s(\dM,E)$ is compact;
\item
by the Sobolev embedding theorem,
$\bigcap_{s\in\R} H^s(\dM,E) = \Cu(\dM,E)$;
\item\label{perfpair}
for all $s\in\R$, the $L^2$-product $(\phi,\psi) = \int_\dM\<\phi,\psi\> d\nu$,
where $\phi,\psi\in\Cu(\dM,E)$, extends to a perfect pairing 
\[
  H^s(\dM,E) \times H^{-s}(\dM,E) \to \C
\]
and therefore renders $H^s(\dM,E)$ and $H^{-s}(\dM,E)$ as pairwise dual;
\item
for all $k\geq1$, the restriction map $\mathcal R:\Cuc(M,E) \to \Cu(\dM,E)$,
$\mathcal R(\Phi):=\Phi|_{\dM}$, 
extends by the \emph{trace theorem}\index{trace theorem} \cite[Thm.~5.22]{Ad}
to a continuous linear map
\[
  \mathcal R: H^{k}_{\loc}(M,E) \to H^{k-1/2}(\dM,E) .
\]
\end{enumerate}
\end{facts}

For $I \subset \R$, let $Q_I$\index{1QI@$Q_I$} be the spectral projection of the selfadjoint operator $A$,
\[
 Q_I : \sum_{j=-\infty}^\infty a_j\phi_j \mapsto \sum_{\lambda_j\in I} a_j\phi_j.
\]
Then $Q_I$ is an orthogonal projection in $L^2(\dM,E)$ and 
\[
  Q_I(H^s(\dM,E)) \subset H^s(\dM,E)
\]
for all $s\in\R$.
In particular, $Q_I(\Cu(\dM,E))\subset \Cu(\dM,E)$.
We abbreviate \index{1HAsI@$H^s_I(A)$}
\[
  H^s_I(A) := Q_I(H^s(\dM,E)).
\]
Fix $\Lambda\in\R$ and define \index{1normHAcheck@$\HcdM{\cdot}$}
\begin{equation}
  \HcdM{\phi}^2 := 
  \HhdM{Q_{(-\infty,\Lambda]}\phi}^2 + \HmhdM{Q_{(\Lambda,\infty)}\phi}^2.
\label{eq:defHc}
\end{equation}
This norm is, up to equivalence, independent of the choice of $\Lambda$.
Namely, if $\Lambda_1 < \Lambda_2$,
then the corresponding $\check{H}$-norms coincide
on the $L^2$-orthogonal complement of $Q_{[\Lambda_1,\Lambda_2]}(\Cu(\dM,E))$.
Now, the latter space is finite-dimensional so that any two norms on it are equivalent,
hence the claim.

The completion of $\Cu(\dM,E)$ with respect to $\HcdM{\cdot}$
will be denoted $\check{H}(A)$.\index{1HAcheck@$\check{H}(A)$}
In other words,
\begin{equation}\label{eq:Hcheck}
  \check{H}(A) =
  H^{1/2}_{(-\infty,\Lambda]}(A) \oplus H^{-1/2}_{(\Lambda,\infty)}(A).
\end{equation}
Similarly, we set \index{1normHAroof@$\HddM{\cdot}$}
\begin{equation}
  \HddM{\phi}^2 := 
  \HmhdM{Q_{(-\infty,\Lambda]}\phi}^2 + \HhdM{Q_{(\Lambda,\infty)}\phi}^2
\label{eq:defHd}
\end{equation}
and \index{1HAdach@$\hat{H}(A)$}
\begin{equation}\label{eq:Hdach}
  \hat{H}(A) :=
  H^{-1/2}_{(-\infty,\Lambda]}(A) \oplus H^{1/2}_{(\Lambda,\infty)}(A).
\end{equation}

While $H^s(\dM,E)$ is independent of $A$, the definitions
of $H^s_I(A)$, $\check{H}(A)$, and $\hat{H}(A)$ do depend on $A$.
We have 
\begin{equation}\label{chckhat}
  \hat H(A) = \check H(-A) .
\end{equation}
The $L^2$-product on $\Cu(\dM,E)$ uniquely extends to a perfect pairing 
\begin{equation}\label{omega}
\check H(A) \times \hat H(A) \to \C .
\end{equation}
Hence $\hat H(A)$ is canonically isomorphic to the dual space of $\check H(A)$
and conversely.
Let
\begin{equation}\label{hfinite}
  \Hfin(A) := \Big\{\phi = \sum_{j=-\infty}^\infty a_j\phi_j\,\Big|\, 
  a_j=0 \mbox{ for all but finitely many }j\Big\}
\end{equation}
be the space of ``finite Fourier series''.\index{1HAfinite@$\Hfin(A)$}
We have the inclusions
\begin{multline*}
  \Hfin(A) \subset \Cu(\dM,E) \\ \subset H^{1/2}(\dM,E) \subset
  \left\{ \begin{array}{c} L^2(\dM,E) \\ \hat H (A) \\ \check H (A)\end{array}\right\}
  \subset H^{-1/2}(\dM,E) .
\end{multline*}
The space $\Hfin(A)$ is dense in any of these spaces.
Sections $\Phi$ in 
\[
  L^2(Z_{[0,\infty)},E) = L^2([0,\infty),L^2(\dM,E))
\]
can be developed in the basis $(\phi_j)_j$ with coefficients depending on $t$,
\[
  \Phi(t,x) = \sum_{j=-\infty}^\infty a_j(t)\phi_j(x) .
\]
We fix a constant $r>0$
and a smooth cut-off function $\chi: \R \to \R$\index{1Chi@$\chi$} with 
\begin{equation}\label{chi}
  \text{$\chi(t)=1$ for all $t\le r/3$
  and $\chi(t)=0$ for all $t\ge 2r/3$.}
\end{equation}
We define, for $\phi\in \Hfin(A)$,
a smooth section $\mathcal E\phi$ of $E$ over $Z_{[0,\infty)}$ by\index{1E@$\mathcal E$}
\begin{equation}\label{exfi}
  (\mathcal E\phi)(t) := \chi(t)\cdot\exp(-t|A|)\phi .
\end{equation}
In other words, for $\phi(x) = \sum_{j=-\infty}^\infty a_j\phi_j(x)$ we have 
\[
  (\mathcal E\phi)(t,x)
  = \chi(t)\sum_{j=-\infty}^\infty a_j\cdot\exp(-t|\lambda_j|)\cdot\phi_j(x) .
\]
We obtain a linear map $\mathcal E:\Hfin(A) \to \Cuc(Z_{[0,\infty)},E)$.

\begin{lemma}\label{extfin}
There is a constant $C=C(\chi,A)>0$ such that
\[
  \|\mathcal E\phi \|_{D_0}^2 \le C \| \phi \|_{\check H(A)}^2
\]
for all $\phi\in \Hfin(A)$.
\end{lemma}

\begin{proof}
Without loss of generality we choose $\Lambda=0$ in \eqref{eq:Hcheck},
the definition of $\check H(A)$.
Since the eigenspaces of $A$ are pairwise $L^2$-orthogonal,
it suffices to consider the Fourier coefficients of $\mathcal E\phi$ separately.
We see that 
\[
  D_0\exp(-t|A|)Q_{(0,\infty)}\phi
  = \sigma_0 \big( \frac{\del}{\del t} + A \big)\exp(-tA)Q_{(0,\infty)}\phi=0,
\]
and hence that
\[
  D_0(\mathcal EQ_{(0,\infty)}\phi) = \chi'\sigma_0\exp(-tA)Q_{(0,\infty)}\phi .
\]
It follows that the graph norm $\|\mathcal EQ_{(0,\infty)}\phi\|_{D_0}$
can be bounded from above by $\|\exp(-tA)Q_{(0,\infty)}\phi\|_{L^2(Z_{[0,r)})}$.
Now there is some $\eps>0$ such that $|\lambda|\ge\eps>0$
for all nonzero eigenvalues $\lambda$ of $A$.
Hence, for $\phi = \sum_{j} a_j\phi_j$, we have 
\begin{align*}
  \|\exp(-t|A|)Q_{(0,\infty)}\phi\|_{L^2(Z_{[0,r)})}^2
  &=
  \sum_{\lambda_j\geq\eps} |a_j|^2 \int_0^r e^{-2t\lambda_j} \, dt\\
  &=
  \frac12\sum_{\lambda_j\geq\eps} |a_j|^2 \cdot \lambda_j^{-1}(1-e^{-2r\lambda_j}) \\
  &\le
  \frac12\sum_{\lambda_j\geq\eps} |a_j|^2\cdot \lambda_j^{-1}\\
  &\le
  \frac{(1+\eps^{-2})^{1/2}}{2}\sum_{\lambda_j\geq\eps} |a_j|^2 \cdot (1+\lambda_j^2)^{-1/2}\\
  &=
  \frac{(1+\eps^{-2})^{1/2}}{2} \|Q_{(0,\infty)}\phi\|_{H^{-1/2}(\dM,E)}^2\\
  &=
  \frac{(1+\eps^{-2})^{1/2}}{2} \|Q_{(0,\infty)}\phi\|_{\check H(A)}^2 .
\end{align*}
The claimed inequality for $\|\mathcal EQ_{(0,\infty)}\phi \|_{D_0}^2$ follows.
The estimate for $\|\mathcal EQ_{(-\infty,0]}\phi \|_{D_0}^2$
follows from similar considerations, where now
\[
  (D_0\exp(-t|A|)Q_{(-\infty,0]}\phi)(t)
  = 2\sigma_0A\exp(-t|A|)Q_{(-\infty,0]}\phi
\]
and hence 
\[
  (D_0\mathcal EQ_{(-\infty,0]}\phi)(t)
  = \sigma_0(2\chi A+\chi')\exp(tA)Q_{(-\infty,0]}\phi.
\]
Thus the graph norm $\|\mathcal EQ_{(-\infty,0]}\phi\|_{D_0}$ can be bounded
from above by
\[
  \|(\id+|A|)\exp(-t|A|)Q_{(-\infty,0]}\phi\|_{L^2(Z_{[0,r)})} .
\]
Now we have
{\allowdisplaybreaks 
\begin{align*}
  \|(\id+|A|)&\exp(-t|A|)Q_{(-\infty,0]}\phi\|_{L^2(Z_{[0,r)})}^2 \\
  &=
  \sum_{\lambda_j\leq 0} \int_0^r |a_j|^2
  (1+|\lambda_j|)^2 e^{2t\lambda_j} \, dt \\
  &=
  r\sum_{\lambda_j=0}|a_j|^2
  + \frac12\sum_{\lambda_j\leq-\eps} |a_j|^2 (1+|\lambda_j|)^2 |\lambda_j|^{-1}
  (1-e^{2r\lambda_j}) \\
  &\le
  r\sum_{\lambda_j=0}|a_j|^2
  + \frac12\sum_{\lambda_j\leq-\eps} |a_j|^2 (1+|\lambda_j|)^2 |\lambda_j|^{-1} \\
  &\le
  r\sum_{\lambda_j=0}|a_j|^2
  + C_1(\eps)\sum_{\lambda_j\leq-\eps} |a_j|^2 (1+|\lambda_j|^2)^{1/2} \\
  &\le
  C_2(\eps,r)\sum_{\lambda_j\leq 0} |a_j|^2 (1+|\lambda_j|^2)^{1/2} \\
  &=
  C_2(\eps,r)\|Q_{(-\infty,0]}\phi\|_{H^{1/2}(\dM,E)}^2 \\
  &= C_2(\eps,r)\|Q_{(-\infty,0]}\phi\|_{\check H(A)}^2 .
  \qedhere
\end{align*}
}
\end{proof}

\begin{remark}
With the same methods one can show that $\|\mathcal E\phi\|_{D_0}$
can also be bounded from below by $\|\phi\|_{\check H(A)}$.
\end{remark}


\section{The maximal domain}
\label{secMaxDom}


Throughout this Section, assume the Standard Setup~\ref{stase},
identify a neighborhood $U$ of the boundary $\dM$ with a cylinder $Z_{[0,r)}$
via an adapted diffeomorphism as in \lref{adapted},
and fix a normal form for $D$ and $D^*$ as in \lref{nf}.
Assume furthermore that $D$ and $D^*$ are complete.


\subsection{Regularity properties of the maximal domain}
\label{susemaxdom}

For some of our assertions concerning estimates over the cylinder $Z_{[0,r)}$,
it will be convenient to consider the operator $\sigma^{-1}D$ instead of $D$.
Its formal adjoint is given by
\begin{equation}\label{sigd}
 (\sigma^{-1}D)^* = - \left( \frac{\partial}{\partial t} - A - R_t^* \right)
\end{equation}
so that we can let $-A$ take over the role of $\tilde A$
in the normal form of $\sigma^{-1}D$ as in \lref{nf}.
Then we have $\check H(-A)=\hat H(A)$.

\begin{lemma}\label{extfin2}
For $\phi\in \Hfin(A)$, the section $\mathcal E\phi$ of $E$ over $Z_{[0,r)}$
belongs to the maximal domain of $D$.
Moreover, there is a constant $C>0$ such that
\[
  \|\mathcal E\phi \|_D 
  \le C \|\phi\|_{\check H(A)} 
  \quad\text{and}\quad
  \|\mathcal E\phi \|_{(\sigma ^{-1}D)^*}
  \le C \|\phi\|_{\hat H(A)} .
\]
\end{lemma}

\begin{proof}
The section  $\mathcal E\phi$ belongs to the maximal domain of $D$
because $\mathcal E\phi \in \Cuc(Z_{[0,r)},E) \subset \Cuc(M,E) \subset \dom(\Dmax)$.
The first estimate has been shown in \lref{extfin}
with the model operator $D_0$ instead of $D$.
By the definition of $D_0$, we have
\begin{align*}
  \|D(\mathcal E\phi)\|_{L^2(M)}
  &=
  \|D(\mathcal E\phi)\|_{L^2(Z_{[0,r)})}\\
  &\le
  C_1\cdot \|\sigma_0\sigma_t^{-1}D(\mathcal E\phi)\|_{L^2(Z_{[0,r)})}\\
  &=
  C_1\cdot \|(D_0+\sigma_0 R_t)(\mathcal E\phi)\|_{L^2(Z_{[0,r)})}\\
  &\le
  C_1\cdot \|D_0(\mathcal E\phi)\|_{L^2(Z_{[0,r)})} +
  C_2\cdot \|R_t(\mathcal E\phi)\|_{L^2(Z_{[0,r)})} \\
  &\le
  C_1\cdot \|\mathcal E\phi\|_{D_0} +
  C_2\cdot \|R_t(\mathcal E\phi)\|_{L^2(Z_{[0,r)})} .
\end{align*}
It remains to estimate $\|R_t(\mathcal E\phi)\|_{L^2(Z_{[0,r)})}$.
By \eqref{Rt}, we get
\begin{align*}
  \|R_t&(\mathcal E\phi)\|_{L^2(Z_{[0,r)})}^2 \\
  &\le
  C_3 \int_0^{r} \left( \|tA(\mathcal E\phi) \|_{L^2(\dM)}^2
  +  \|\mathcal E\phi \|_{L^2(\dM)}^2 \right) dt\\
  &=
  C_3 \int_0^{r} \chi(t)^2 \left( \|t A \exp(-t|A|)\phi \|_{L^2(\dM)}^2
  +  \|\exp(-t|A|)\phi \|_{L^2(\dM)}^2 \right) dt .
\end{align*}
If $A\phi_j = \lambda_j\phi_j$, then we compute,
substituting $\tilde{t}=t\cdot|\lambda_j|$,
{
\allowdisplaybreaks
\begin{align*}
  \int_0^{r} \chi(t)^2 &\left( \|t A \exp(-t|A|)\phi_j \|_{L^2(\dM)}^2
  +  \|\exp(-t|A|)\phi_j \|_{L^2(\dM)}^2 \right) dt \\
  &\le
  \int_0^r \exp(-2t|\lambda_j|)(t^2|\lambda_j|^2+1)\|\phi_j\|_{L^2(\dM)}^2 \, dt \\
  &=
  \int_0^{r|\lambda_j|} \exp(-2\tilde{t})(\tilde{t}^2+1)\|\phi_j\|_{L^2(\dM)}^2 \,
  \frac{d\tilde{t}}{|\lambda_j|} \\
  &\le
  \frac{\|\phi_j\|_{L^2(\dM)}^2}{|\lambda_j|}\int_0^\infty \exp(-2\tilde{t})
  (\tilde{t}^2+1)\, d\tilde{t} \\
  &=
  \frac34 \cdot\||\lambda_j|^{-1/2}\phi_j\|_{L^2(\dM)}^2 .
\end{align*}
}
Expanding $\phi$ in an orthonormal eigenbasis for $A$, this shows
\begin{multline*}
  \int_0^{r} \chi(t)^2 \left( \|t A \exp(-t|A|)\phi \|_{L^2(\dM)}^2
  +  \|\exp(-t|A|)\phi \|_{L^2(\dM)}^2 \right) dt \\
  \le
  C_4 \|\phi\|_{H^{-1/2}(\dM)}^2 \le C_4 \|\phi\|_{\check H(A)}^2 .
\end{multline*}
This concludes the proof of the first inequality.
For the proof of the second,
we recall from \eqref{sigd} that $-A$ is an adapted boundary operator
for $(\sigma ^{-1}D)^*$ and that $\check H(-A) = \hat H(A)$.
\end{proof}

\begin{lemma}\label{destimate}
There is a constant $C>0$ such that, for all $\Phi\in C^\infty_c(Z_{[0,r)},E)$,
\[
  \| \Phi_{|\dM} \|_{\check H(A)} \le C \| \Phi \|_{D} .
\]
\end{lemma}

\begin{proof}
Since the pairing between $\check H(A)$ and $\hat H(A)$ introduced
in \eqref{omega} is perfect  and since $\Hfin(A)$ is dense in $\hat H(A)$,
we have
\[
  \| \Phi_{|\dM} \|_{\check H(A)}
  = \sup \{ |(\Phi_{|\dM},\psi )| \mid
  \psi\in \Hfin(A) , 
  \|\psi\|_{\hat H(A)}=1 \} .
\]
Now \lref{lem:Green} and \lref{extfin2} imply that,
for any $\psi\in \Hfin(A)$ with $\|\psi\|_{\hat H(A)}=1$,
\begin{align*}
  |(\Phi_{|\dM}, &\psi)|
  =
  \left|(\sigma^{-1}D\Phi,\mathcal E\psi)_{L^2(M)} 
        - (\Phi,(\sigma^{-1}D)^*(\mathcal E\psi))_{L^2(M)}\right|  \\
  &\le 
  \LzM{\sigma^{-1}D\Phi}\LzM{\mathcal E\psi}
  + \LzM{\Phi}\LzM{(\sigma^{-1}D)^*(\mathcal E\psi)} \\
  &\le
  C'\cdot\| \Phi \|_D \cdot \|\mathcal E\psi \|_{(\sigma^{-1}D)^*} \\
  &\le
  C \cdot\| \Phi \|_D \cdot\| \psi \|_{\hat H(A)} \\
  &=
  C \cdot\| \Phi \|_D .
\qedhere
\end{align*}
\end{proof}

\begin{lemma}\label{duality}
Over $\dM$, the homomorphism field $(\sigma _0^{-1})^*:E\to F$
induces an isomorphism $\hat H(A)\to\check H(\tilde A)$. 
Here $\tilde A$ is an adapted boundary operator for $D^*$.
In particular, the sesquilinear form
\[
 \beta: \check H(A) \times \check H(\tilde A) \to \C , \quad
  \beta(\phi,\psi) := - (\sigma_0\phi,\psi) ,
\]
is a perfect pairing of topological vector spaces.
\end{lemma}

\begin{proof}
For $\phi\in C^\infty(\partial M,E)$,
we have $(\sigma_0^{-1})^*\phi\in C^\infty(\partial M,F)$ and
\begin{align*}
  \| (\sigma _0^{-1})^*\phi \|_{\check H(\tilde A)}^2
  &\le C_1 \|(\sigma ^{-1})^*\mathcal E\phi\|_{D^*}^2 \\
  &= C_1 \big( \| D^* (\sigma ^{-1})^*\mathcal E\phi \|_{L^2(Z_{[0,r)})}^2
  +  \| (\sigma ^{-1})^*\mathcal E\phi \|_{L^2(Z_{[0,r)})}^2 \big) \\
  &\le C_2 \big( \| (\sigma ^{-1}D)^*\mathcal E\phi \|_{L^2(Z_{[0,r)})}^2
  +  \| \mathcal E\phi \|_{L^2(Z_{[0,r)})}^2 \big) \\
  &= C_2 \|\mathcal E\phi\|_{(\sigma ^{-1}D)^*}^2 \\
  &\le C_3 \|\phi\|_{\hat H(A)}^2
\end{align*}
by Lemmas~\ref{destimate} (for $D^*$) and \ref{extfin2}, respectively.

Conversely, for $\psi\in C^\infty(\partial M,F)$,
we have $\sigma _0^*\psi\in C^\infty(\partial M,E)$.
Furthermore, if $\tilde{\mathcal E}$ denotes the extension operator
defined in \eqref{exfi} associated to the boundary operator $\tilde A$, 
we obtain
\begin{align*}
  \| \sigma _0^*\psi \|_{\hat H(A)}^2
  &=
  \|(\sigma ^* \tilde{\mathcal E} \psi)|_{\dM}\|_{\check H(-A)}^2 \\
  &\le 
  C_4\|\sigma ^* \tilde{\mathcal E} \psi\|_{(\sigma ^{-1}D)^*}^2 \\
  &=
  C_4 \big( \| (\sigma ^{-1}D)^*\sigma ^* \tilde{\mathcal E} \psi \|_{L^2(Z_{[0,r)})}^2
  + \| \sigma ^* \tilde{\mathcal E} \psi \|_{L^2(Z_{[0,r)})}^2 \big) \\
  &\le 
  C_5 \big( \| D^* \tilde{\mathcal E} \psi \|_{L^2(Z_{[0,r)})}^2
  + \| \tilde{\mathcal E} \psi \|_{L^2(Z_{[0,r)})}^2 \big) \\
  &=
  C_5 \|\tilde{\mathcal E} \psi\|_{D^*}^2 \\
  &\le 
  C_6 \| \psi \|_{\check H(\tilde A)}^2
\end{align*}
by \lref{destimate} applied to the operator $(\sigma ^{-1}D)^*$
and \lref{extfin2} applied to the operator $D^*$.

Since $C^\infty(\partial M,E)$ and $C^\infty(\partial M,F)$ are dense
in $\hat H(A)$ and $\check H(A)$, respectively,
and since $(\sigma _0^*)^{-1}=(\sigma _0^{-1})^*$,
we conclude that $(\sigma _0^{-1})^*$ induces a topological  isomorphism
between the Hilbert spaces $\hat H(A)$ and $\check H(A)$ as asserted.
The second claim is now an immediate consequence since the corresponding
pairing between $\check H(A)$ and $\hat H(A)$ is perfect.
\end{proof}

Consider the intersection
\begin{equation}
  H^1_D(M,E) := \dom(D_{\max})\cap H^1_{\loc}(E,F).
  \index{1HgDME@$H^1_D(M,E)$}
  \label{H1}
\end{equation}
The $H^1_D$-norm on $H^1_D(M,E)$ is defined by
\begin{equation}
  \| \Phi \|_{H^1_D(M)}^2 :=
  \| \chi\Phi \|_{H^1(M)}^2 + \LzM{\Phi}^2 + \LzM{D\Phi}^2 ,
  \label{eq:H1Ddef}\index{1normHgD@$\HlD{\cdot}$}
\end{equation}
where $\chi$ is as in \eqref{chi},
and turns $H^1_D(M,E)$ into a Hilbert space.
Since the support of $\chi$ as a function on $M$ is compact,
the specific choice of $H^1$-norm is irrelevant and leads to equivalent $H^1_D$-norms.
The $H^1_D$-norm is stronger than the graph norm for $D$;
it controls in addition $H^1$-regularity near the boundary.
The completeness of $D$ is responsible for the following properties of $H^1_D(M,E)$.

\begin{lemma}\label{density}
\begin{enumerate}[(i)]
\item\label{cdense}
$\Cuc(M,E)$ is dense in $H^1_D(M,E)$;
\item\label{ccdense}
$\Cucc(M,E)$ is dense in $\{\Phi\in H^1_D(M,E) \mid \Phi|_{\dM}=0 \}$.
\end{enumerate}
\end{lemma}

\begin{proof}
\eqref{cdense}
We have to show that any given $\Phi\in H^1_D(M,E)$ can be approximated
by compactly supported smooth sections in the $H^1_D$-norm.
Let $\chi\in\Cuc(M,\R)$ be the cut-off function used in the definition of $\HlD{\cdot}$,
that is, the function from \eqref{chi}.
Choose a second cut-off function $\chi_2\in\Cuc(M,\R)$ with $\chi_2\equiv 1$
on the support of $\chi$ and a third cut-off function $\chi_3\in\Cuc(M,\R)$
with $\chi_3\equiv 1$ on the support of $\chi_2$.

\begin{center}
\begin{pspicture}(-2,0)(2,2)
\psline{->}(-1.5,0)(2.5,0)
\psline{->}(-1.5,0)(-1.5,2)

\psline(-1.5,1.5)(1.0,1.5)
\psbezier(-1,1.5)(-0.1,1.5)(-0.9,0)(0,0)
\psbezier(0,1.5)(0.9,1.5)(0.1,0)(1,0)
\psbezier(1,1.5)(1.9,1.5)(1.1,0)(2,0)

\rput(-0.7,0.7){$\chi$}
\rput(0.25,0.7){$\chi_2$}
\rput(1.8,0.7){$\chi_3$}
\end{pspicture}
\nopagebreak

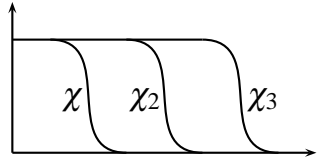
{\em Fig.~4}
\end{center}

Then $\chi_3\Phi\in H^1_D(M,E)$ because $\Phi\in H^1_{\loc}(M,E)$ and $\chi_3\Phi$
has compact support.
Therefore $\chi_3\Phi$ can be approximated by smooth compactly supported
sections in any $H^1$-norm, hence also in the $H^1_D$-norm.

It remains to approximate $(1-\chi_3)\Phi\in H^1_D(M,E)$.
Since $D$ is complete, there exists $\Phi_0\in\dom(D_{\max})$
with compact support such that
\[
  \|(1-\chi_3)\Phi-\Phi_0\|_D < \epsilon
\]
for any given $\epsilon>0$.
Since $\chi_2$ vanishes on the support of $1-\chi_3$, we have
\begin{align*}
  \HlD{(1-\chi_3)\Phi-(1-\chi_2)\Phi_0}
  &=
  \HlD{(1-\chi_2)((1-\chi_3)\Phi-\Phi_0)} \\
  &=
  \|(1-\chi_2)((1-\chi_3)\Phi-\Phi_0)\|_D \\
  &\le
  C\cdot \|(1-\chi_3)\Phi-\Phi_0\|_D \\
  &<
  C\cdot \epsilon .
\end{align*}
Hence it suffices to approximate $(1-\chi_2)\Phi_0$.
Since $(1-\chi_2)\Phi_0$ has compact support,
this is possible exactly like for $\chi_3\Phi$.

\eqref{ccdense}
Suppose that $\Phi\in H^1_D(M,E)$ vanishes along $\dM$.
We have to show that $\Phi$ can be approximated by sections from $\Cucc(M,E)$.
As in the first part of the proof, we may assume that the support of $\Phi$ is
contained in a neighborhood of $\dM$, say in $Z_{r/2}$.
For $n\in \N$ sufficiently large,
define sections $\Phi_n$ of $E$ over $Z_{[0,r)}$ by
\[
  \Phi_n(t,x) = \begin{cases}
  0 &\text{for $0 \le t \le 1/n$} , \\
  \Phi(t-1/n,x) \quad &\text{for $1/n \le t < r$} .
  \end{cases}
\]
Then $\Phi_n$ has compact support in $Z_{[1/n,r]}$,
$\Phi_n\in H^1(Z_{[0,r)},E)$, and $\Phi_n\to\Phi$ in $H^1(Z_{[0,r)},E)$,
therefore also in $H^1_D(M,E)$.
Since $\Phi_n$ has compact support in $Z_{(0,r)}$,
it can be approximated in $H^1_D(M,E)$
by smooth sections of $E$ with compact support in $Z_{(0,r)}$.
\end{proof}

\begin{lemma}\label{esth1}
There is a constant $C>0$ such that
\[
  \| \Phi \|_D 
  \le 
  \| \Phi \|_{H^1_D(M)} 
  \le 
  C \,\| \Phi \|_D
\]
for all $\Phi\in C^\infty_c(M,E)$ with $Q_{(0,\infty)}(\Phi_{|\dM})=0$.
\end{lemma}

\begin{proof}
We show that the $H^1_D$-norm is bounded from above by the graph norm,
the converse inequality being clear.
We write $\Phi=\Phi_1+\Phi_2$ where $\Phi_1 = \chi\Phi \in \Cuc(Z_{[0,r)},E)$
and $\Phi_2 = (1-\chi)\Phi \in \Cucc(M,E)$ has support disjoint from $Z_{r/3}$.
On the space of $\Phi_2$'s the graph norm and the $H^1_D$-norm are equivalent.
Therefore we can assume that $\Phi=\Phi_1$ has compact support
in the closure of $Z_{2r/3}$.

Since $Q_{(0,\infty)}\Phi_{|\dM}=0$,
we have $(\Phi_{|\dM},A\Phi_{|\dM})_{L^2(\dM)}\leq 0$ and hence
\begin{align}
  \| D_0\Phi \|_{L^2(Z_{[0,r)})}^2
  &=
  \| \sigma _0(\Phi'+ A\Phi) \|_{L^2(Z_{[0,r)})}^2\nonumber \\
  &\ge
  C_1\|\Phi' + A\Phi \|_{L^2(Z_{[0,r)})}^2 \nonumber \\
  &=
  C_1\{\| \Phi' \|_{L^2(Z_{[0,r)})}^2 + \| A\Phi \|_{L^2(Z_{[0,r)})}^2   
       - (\Phi_{|\dM},A\Phi_{|\dM})_{L^2(\dM)} \} \nonumber  \\
  &\ge
  C_1\{\| \Phi' \|_{L^2(Z_{[0,r)})}^2 + \| A\Phi \|_{L^2(Z_{[0,r)})}^2 \} ,
\label{eq:c1}
\end{align}
where we use \lref{intpart} to pass from line 2 to line 3.
Thus
\begin{align*}
\HlD{\Phi}^2 
& \le 
C_2\,\{\|\Phi\|_{L^2(Z_{[0,r)})}^2 + \|\Phi'\|_{L^2(Z_{[0,r)})}^2
+ \|A\Phi\|_{L^2(Z_{[0,r)})}^2\} \\
& \le
C_3\,\{\|\Phi\|_{L^2(Z_{[0,r)})}^2 + \|D_0\Phi\|_{L^2(Z_{[0,r)})}^2\} \\
& \le
C_4\,\|\Phi\|_D^2 .
\end{align*}
The first inequality follows from the ellipticity of $A$,
the second from \eqref{eq:c1},
and the third from \lref{lem:relbnd}.
\end{proof}

\begin{cor}\label{domdmin}
Assume the Standard Setup~\ref{stase} and that $D$ and $D^*$ are complete.
Then $\dom D_{\min}=\{\Phi\in H^1_D(M,E) \mid \Phi|_{\dM}=0 \}$.
\end{cor}

\begin{proof}
Sections in $\Cucc(M,E)$ vanish along $\dM$,
hence satisfy the boundary condition $Q_{(0,\infty)}(\Phi_{|\dM})=0$
required in \lref{esth1}, and hence $\dom D_{\min}$ is contained
in $\{\Phi\in H^1_D(M,E) \mid \Phi|_{\dM}=0 \}$.
Now \lref{density} \eqref{ccdense} concludes the proof.
\end{proof}

\begin{thm}\label{domdmax}
Assume the Standard Setup~\ref{stase} and that $D$ and $D^*$ are complete.
Then we have:
\begin{enumerate}[(i)]
\item\label{dense}
$C^\infty_{c}(M,E)$ is dense in $\dom(D_{\max})$ with respect to the graph norm.
\item\label{trace}
The trace map $\Cuc(M,E)\to \Cu(\dM,E)$, $\Phi \mapsto \Phi_{|\dM}$,
extends uniquely to a surjective bounded linear map
$\mathcal R:\dom(D_{\max})\to\check H(A)$.\index{1R@$\mathcal R$}
\item\label{reg}
$H^1_D(M,E)=\{\Phi\in\dom(D_{\max}) \mid\mathcal R\Phi\in H^{1/2}(\dM,E) \}$.
\end{enumerate}
The corresponding statements hold for $\dom((D^*)_{\max})$.
Furthermore, for all sections $\Phi\in\dom(D_{\max})$ and $\Psi\in\dom((D^*)_{\max})$,
we have 
\begin{equation}\label{eq:ParInt}
  (D_{\max}\Phi,\Psi)_{L^2(M)}
  - (\Phi,(D^*)_{\max}\Psi)_{L^2(M)}
  = -(\sigma _0\mathcal R\Phi,\mathcal R\Psi)_{L^2(\dM)} .
\end{equation}
\end{thm}

\begin{proof}
\eqref{dense}
Extend $(M,\mu)$ smoothly to a larger measured manifold $(\tilde M,\tilde\mu)$
with $\partial\tilde M=\emptyset$.
Do it in such a way that $E$ and $F$ extend smoothly to Hermitian vector bundles
$\tilde E$ and $\tilde F$ over $\tilde M$  and $D$ to an elliptic operator
$\tilde D:\Cu(\tilde M,\tilde E) \to \Cu(\tilde M,\tilde F)$.
This is possible, since we may choose $\tilde M\setminus M$
to be diffeomorphic to the product $\dM\times(-1,0)$.

\begin{center}
\begin{pspicture}(-4,-1.5)(4,1.5)
\psellipse[linestyle=dashed](-3,0)(.25,1.2)
\psellipticarc(-3,0)(.25,1.2){95}{265}
\psecurve(-4,2)(-3,1.2)(-2,1)(3,1.5)(4,2.5)
\psecurve(-4,-2)(-3,-1.2)(-2,-1)(3,-1.5)(4,-2.5)
\psecurve(4,1)(3,1)(1.5,0)(3,-1)(4,-1)
\psellipse[linestyle=dashed](-2,0)(.25,1)
\psellipticarc(-2,0)(.25,1){100}{260}

\rput(-2.2,1.4){$\tilde M$}
\rput(-1.8,0){\psframebox*{$\dM$}}
\rput(0,0){$M$}
\end{pspicture} 

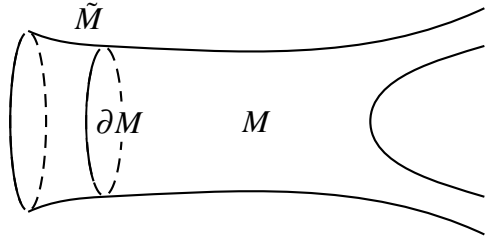
{\em Fig.~5}
\end{center}

Let $D_c$ be $D$ with domain $\dom(D_c)=C^\infty_{c}(M,E)$.
We have to show that the closure of $D_c$ equals $\Dmax$.
Let $\Psi\in L^2(M,F)$ be in the domain of the adjoint operator $(D_c)^{\ad}$.
Extend $\Psi$ and $(D_c)^{\ad}\Psi$ by the trivial section $0$ 
on $\tilde M\setminus M$ to sections $\tilde\Psi\in L^2(\tilde M,\tilde F)$
and $\tilde\Xi\in L^2(\tilde M,\tilde E)$, respectively.
Let $\tilde\Phi\in\Cuc(\tilde M,\tilde E)=\Cucc(\tilde M,\tilde E)$.
Then the restriction $\Phi$ of $\tilde\Phi$ to $M$ is in $\dom(D_c)$ and hence,
since $\tilde\Psi=\tilde\Xi=0$ on $\tilde M\setminus M$,
\begin{align*}
  (\tilde D\tilde\Phi,\tilde\Psi)_{L^2(\tilde M)}
  &=
  (D_c\Phi,\Psi)_{L^2(M)} \\
  &=
  (\Phi,(D_c)^{\ad}\Psi)_{L^2(M)}\\
  &=
  (\tilde\Phi,\tilde\Xi)_{L^2(\tilde M)} .
\end{align*}
Thus $\tilde\Psi$ is a weak solution of 
$\tilde D^*\tilde\Phi=\tilde\Xi\in L^2(\tilde M,\tilde E)$.
By interior elliptic regularity theory, $\tilde\Psi\in \Hlloc(\tilde M,\tilde F)$.
It follows that $\Psi$ is in $H^1_{D^*}(M,F)$ and that $\Psi_{|\dM}=0$.
By \cref{domdmin}, $\Psi\in\dom((D^*)_{\min})$,
the domain of the minimal extension of $(D^*)_{cc}$.
Hence $(D_c)^{\ad}\subset(D^*)_{\min}$,
and therefore the closure $\overline D_c$ of $D_c$ satisfies  
\[
\overline D_c
=
((D_c)^{\ad})^{\ad}
\supset
((D^*)_{\min})^{\ad}
=
D_{\max}
\supset
\overline D_c .
\]
Thus $\overline D_c=D_{\max}$ as asserted.

\eqref{trace}
By \lref{destimate},
the trace map $\Cuc(M,E) \to \Cu(\dM,E)$, $\Phi \mapsto \Phi_{|\dM}$,
extends to a bounded linear map $\dom(\Dmax) \to \check H(A)$.
By \eqref{dense}, this extension is unique.
By \lref{extfin2}, the map $\mathcal E: \Hfin(A) \to \Cuc(M,E)$
extends to a bounded linear map
$\mathcal E: \check H(A) \to \dom(\Dmax)$ with $\mathcal R \circ \mathcal E = \id$.
This proves surjectivity.

\eqref{reg}
The inclusion
$H^1_D(M,E)\subset\{\Phi\in\dom(D_{\max}) \mid\mathcal R\Phi\in H^{1/2}(\dM) \}$
is clear from the definition of $H^1_D(M,E)$ and the standard trace theorem.
To show the converse inclusion,
let $\Phi\in\dom(D_{\max})$ with $\mathcal R\Phi\in H^{1/2}(\dM)$.
Put $\phi:=Q_{(0,\infty)}\mathcal R\Phi\in H^{1/2}(\dM)$.
Expanding $\phi$ with respect to an eigenbasis of $A$,
one easily sees that $\mathcal E\phi\in H^1_D(M,E)$.
In particular, $\Phi-\mathcal E\phi\in \dom(D_{\max})$
and $Q_{(0,\infty)}\mathcal R(\Phi-\mathcal E\phi)=0$.
Now \eqref{dense} and \lref{esth1} imply $\Phi-\mathcal E\phi\in H^1_D(M,E)$.
Thus
\[
  \Phi = \mathcal E\phi + (\Phi-\mathcal E\phi) \in H^1_D(M,E).
\]
The asserted formula for integration by parts holds for all $\Phi, \Psi \in \Cuc(M,E)$,
hence for all $\Phi\in\dom(D_{\max})$ and all $\Psi\in \dom((D^*)_{\max})$,
by \eqref{dense}, \eqref{trace}, and \lref{duality}.
\end{proof}


\subsection{Higher regularity}
\label{susehire}

Fix $\rho>0$.
For $g\in L^2([0,\rho],\C)$ and $t\in[0,\rho]$ set
\begin{equation}\label{est4}
  (R_\lambda g)(t) :=
  \begin{cases}
  \int_0^t g(s) \, e^{\lambda(s-t)} ds
  &\mbox{if $\lambda\ge0$} , \\
  - \int_t^\rho g(s) \, e^{\lambda(s-t)} ds
  &\mbox{if $\lambda<0$} .
  \end{cases}
\end{equation}
For $\lambda>0$, we compute for $f=R_{\lambda}g$:
{\allowdisplaybreaks
\begin{align*}
  \| f \|_{L^2([0,\rho])}^2 
  &= 
  \int_0^\rho \left|\int_0^t g(s) e^{\lambda(s-t)}  ds \right|^2 dt \\
  &\le 
  \int_0^\rho \left(\int_0^t |g(s)|^2 e^{\lambda(s-t)} ds \right)
  \left( \int_0^t e^{\lambda(s-t)} ds \right) dt \\
  &\le 
  \frac{1}{\lambda} \int_0^\rho \int_s^\rho |g(s)|^2 e^{\lambda(s-t)} \,dt\,ds \\
  &= 
  \frac{1}{\lambda^2} \int_0^\rho |g(s)|^2 \, ds \\
  &=  
  \frac{1}{\lambda^2} \| g \|_{L^2([0,\rho])}^2 .
\end{align*}}
There is a similar computation in the case $\lambda<0$.
We obtain
\begin{equation}
  \|f\|_{L^2([0,\rho])}^2
  \le \begin{cases} 
  \frac1{\lambda^2} \|g\|_{L^2([0,\rho])}^2 &\text{for $\lambda\ne0$}, \\
  \frac{\rho^2}{2}\|g\|_{L^2([0,\rho])}^2 &\text{for $\lambda=0$} .
  \end{cases}
\label{eq:L2}
\end{equation}
Now $f$ satisfies 
\begin{equation}
f'+\lambda f=g.
\label{appest1}
\end{equation}
Hence
\[
  \|f'\|_{L^2([0,\rho])}^2 \le 
  \begin{cases} 
  4\|g\|_{L^2([0,\rho])}^2  &\text{for $\lambda\ne0$}, \\
  \phantom{4}\|g\|_{L^2([0,\rho])}^2 &\text{for $\lambda=0$} .
  \end{cases}
\]
In conclusion,
\begin{equation}
  \|f\|_{H^1([0,\rho])}^2  \le
  \begin{cases} 
  (4+\frac1{\lambda^2})\|g\|_{L^2([0,\rho])}^2 &\text{for $\lambda\ne0$}, \\
  (1+\frac{\rho^2}2)\|g\|_{L^2([0,\rho])}^2 &\text{for $\lambda=0$} .
  \end{cases}
\label{eq:H1}
\end{equation}

\begin{remarks}
(a)
The above considerations are an adaption and small simplification
of the corresponding discussion in \cite[Sec.~2]{APS}.

(b)
Note that the bound in \eqref{eq:H1} is independent of $\rho$ for $\lambda\ne 0$
and depends only on an upper bound on $\rho$ for $\lambda=0$.
\end{remarks}

 From now on assume that $\rho<r$ where $r<\infty$ is as in \lref{adapted}.
We apply the above discussion to sections of $E$
over the cylinder $Z_{[0,\rho]}=[0,\rho]\times \dM$.
We choose an orthonormal Hilbert basis of $L^2(\dM,E)$
consisting of eigensections $\phi_j$ of $A$ with corresponding
eigenvalues $\lambda_j$ as in \sref{secMOP}. 
For 
\[
 \Psi \in L^2(Z_{[0,\rho]},F)=L^2([0,\rho],L^2(\dM,F)), \quad
  \sigma _0^{-1}\Psi(t,x) = \sum_{j=-\infty}^\infty g_j(t)\phi_j(x)  , 
\]
we set
\begin{equation}
  S_0\Psi := \sum_{j=-\infty}^\infty  (R_{\lambda_j}g_j) \phi_j .
\label{eq:DefS0}
\end{equation}
Estimates \eqref{eq:L2} and \eqref{eq:H1} show that $\Phi=S_0\Psi$ satisfies
\begin{align*}
  \|\Phi\|_{H^1(Z_{[0,\rho]})}^2
  &\le C_1 \left(\|\Phi\|_{L^2(Z_{[0,\rho]})}^2
  + \|\Phi'\|_{L^2(Z_{[0,\rho]})}^2 + \|A\Phi\|_{L^2(Z_{[0,\rho]})}^2\right) \\
  &= C_1 \sum_{j=-\infty}^\infty \left(\|R_{\lambda_j}g_j\|_{H^1([0,\rho])}^2+\lambda_j^2\|R_{\lambda_j}g_j\|_{L^2([0,\rho])}^2\right) \\
  &\le C_1C_2\sum_{j=-\infty}^\infty \|g_j\|_{L^2([0,\rho])}^2 \\
  &= C_1C_2 \|\sigma _0^{-1}\Psi\|_{L^2(\Zrq)}^2 \\
  &\le C_1C_2C_3 \|\Psi\|_{L^2(\Zrq)}^2 .
\end{align*}
Here $C_1$ depends on the specific choice of $H^1$-norm, $C_2$ on $r$
(which is an upper bound for $\rho$) and a lower bound for the modulus
of the nonzero $\lambda_j$, and $C_3$ on an upper bound for $\sigma _0^{-1}$.
Differentiating \eqref{appest1} repeatedly,
we get
\begin{equation}
  \|\Phi\|_{H^{m+1}(\Zrq)}^2 \le C_4 \|\Psi\|_{H^m(\Zrq)}^2 ,
\label{eq:S}
\end{equation}
where $C_4$ depends on the specific choice of $H^k$-norms, on $m$, on $r$,
on a lower bound for the nonzero $|\lambda_j|$, and on an upper bound
for the derivatives of order up to $m$ of $\sigma _0^{-1}$.
By definition, $C_4$ is a bound for
\[
  S_0: H^m(\Zrq,F) \to H^{m+1}(\Zrq,E) .
\]
Moreover, with the model operator $D_0$ as in \eqref{moop}, we have
\begin{equation}
  D_0S_0\Psi = \Psi
\label{eq:DoS}
\end{equation}
and, by definition, $S_0\Psi$ satisfies the boundary condition
\begin{equation}
  Q_{[0,\infty)} (S_0\Psi)(0) = 0 
  \quad\text{and}\quad
  Q_{(-\infty,0)} (S_0\Psi)(\rho) = 0 .
 \label{eq:Randbed}
\end{equation}
By \tref{domdmax} \eqref{trace}, the boundary values of elements in the maximal
domain of $D_0$ over $\Zrq$ constitute the space $\check H(A)\oplus\hat H(A)$,
where $\check H(A)$ is responsible for the left part of the boundary,
$\{0\}\times \dM$, and $\hat H(A)$ for the right, $\{\rho\}\times \dM$.
Set
\begin{equation}
  B_0 := H_{(-\infty,0)}^{1/2}(A)\oplus H_{[0,\infty)}^{1/2}(A) 
\label{eq:DefB0}
\end{equation}
and, for $m\ge1$,
\[
H^m(\Zrq,E;B_0) := \{f\in H^m(\Zrq,E)\mid (f(0),f(\rho))\in B_0 \} .
\]
By \eqref{eq:Randbed},
$S_0$ is a bounded operator $H^m(\Zrq,F) \to H^{m+1}(\Zrq,E;B_0)$.

\begin{prop}\label{riso}
Let $S_0$ be as in \eqref{eq:DefS0} and $B_0$ as in \eqref{eq:DefB0}.
Then, for all $m\ge0$, the model operator
\[
  D_0: H^{m+1}(\Zrq,E;B_0) \to H^m(\Zrq,F)
\]
is an isomorphism with inverse $S_0$.
Furthermore, for elements $\Phi$ in the maximal domain of $D_0$
satisfying $Q_{(-\infty,0)}(\Phi(\rho))=0$, we have 
\[
\Phi-S_0D_0\Phi = \exp(-tA)(Q_{[0,\infty)}\Phi(0)) .
\]
\end{prop}

\begin{proof}
Surjectivity of $D_0$ follows from \eqref{eq:DoS} and \eqref{eq:Randbed}.
The Fourier coefficients $f_j$ of an element in the kernel of $D_0$
satisfy $f_j'+\lambda_j f_j=0$, thus the boundary condition $B_0$,
i.e.\ $f_j(0)=0$ or $f_j(\rho)=0$, implies that they vanish.
This shows injectivity and proves the first assertion.

Let $\Phi$ be in the maximal domain of $D_0$ with $Q_{(-\infty,0)}(\Phi(\rho))=0$.
Since $D_0(\Phi-S_0D_0\Phi)=0$,
the Fourier coefficients $f_j$ of $\Phi-S_0D_0\Phi$ satisfy $f_j'+\lambda_j f_j=0$.
From $Q_{(-\infty,0)}(\Phi(\rho))=0$ and $Q_{(-\infty,0)}(S_0D_0\Phi(\rho))=0$
we have $f_j(\rho)=0$ for all $\lambda_j<0$.
Hence $f_j\equiv 0$ whenever $\lambda_j<0$.
Thus 
\begin{equation}
(\Phi-S_0D_0\Phi)(t,x) = \sum_{\lambda_j\ge0} e^{-t\lambda_j}f_j(0)\phi_j(x).
\label{eq:riso1}
\end{equation}
On the other hand,
\begin{align}
 \exp(-tA)(Q_{[0,\infty)}\Phi(0))
&=
\exp(-tA)(Q_{[0,\infty)}(\Phi-S_0D_0\Phi)(0)) \nonumber\\
&=
\exp(-tA)\big(\sum_{\lambda_j\ge0} f_j(0)\phi_j\big) \nonumber\\
&=
\sum_{\lambda_j\ge0} e^{-t\lambda_j}f_j(0)\phi_j .
\label{eq:riso2}
\end{align}
Equating \eqref{eq:riso1} and \eqref{eq:riso2} concludes the proof.
\end{proof}

We return from the model operator $D_0$ to the original operator $D$.

\begin{lemma}\label{lem:DDo}
Let $(M,\mu)$ be a measured manifold with compact boundary
and let $D$ and $A$ be as in the Standard Setup~\ref{stase}.
Let $m\ge0$.
If $\rho>0$ is sufficiently small, then
\begin{enumerate}[(i)]
\item \label{DDoi}
$D_0$ and $D$ have the same maximal domain when regarded as unbounded
operators from $L^2(Z_{[0,\rho]},E)$ to $L^2(Z_{[0,\rho]},F)$.
\item\label{DDoii}
$(D-\sigma_0R_0): H^{m+1}(Z_{[0,\rho]},E;B_0) \to H^m(Z_{[0,\rho]},F)$
is an isomorphism.
\end{enumerate}
\end{lemma}

\begin{proof}
By \lref{lem:relbnd}, $D-D_0$ is relatively $D_0$-bounded with constant $C\rho$.
For $\rho$ so small that $C\rho<1$ the first assertion follows
from \cite[Thm.~1.1, p.~190]{Ka}.

If $\rho$ is small enough, we have by \lref{nf}, for given $\eps>0$,
\[
  \| (D_0 - (D-\sigma_0R_0))(\Phi) \|_{H^m(Z_{[0,\rho]},E)}
  \le \eps \|\Phi\|_{H^{m+1}(Z_{[0,\rho]},E)}
\]
for all $\Phi\in H^{m+1}(M,E)$.
Since $\rho$ is small,
the norm of the inverse $S_0:H^m(Z_{[0,\rho]},F)\to H^{m+1}(Z_{[0,\rho]},E;B_0)$
of $D_0$ is bounded by a constant $C$ independent of $\rho$,
see \eqref{eq:S}.
Thus if $\eps < 1/C$, then
\[
  (D-\sigma_0R_0):H^{m+1}(Z_{[0,\rho]},E;B_0)\to H^m(Z_{[0,\rho]},F)
\]
is also an isomorphism.
\end{proof}

\begin{thm}\label{charasob}
Assume the Standard Setup~\ref{stase} and that $D$ and $D^*$ are complete.
Then, for any $m\ge0$, 
\begin{multline*}
  \dom(D_{\max}) \cap H^{m+1}_{\loc}(M,E) \\
  = \{\Phi\in\dom(D_{\max}) \mid D\Phi\in H^m_{\loc}(M,F)
  \mbox{ and } Q_{[0,\infty)}(\mathcal R\Phi)\in H^{m+1/2}(\dM,E)\} .
\end{multline*}
\end{thm}

\begin{proof}
The case $m=0$ is assertion \eqref{reg} in \tref{domdmax},
so that we may assume $m>0$ in the following.

Clearly, if $\Phi\in\dom(D_{\max}) \cap H^{m+1}_{\loc}(M,E)$,
then $D\Phi\in H^m_{\loc}(M,F)$
and $\mathcal R\Phi\in H^{m+1/2}(\dM,E)$.
A fortiori, $Q_{[0,\infty)}(\mathcal R\Phi)\in H^{m+1/2}(\dM,E)$.

Conversely, let $\Phi\in\dom(D_{\max})$ satisfy $D\Phi\in H^m_{\loc}(M,F)$
and $Q_{[0,\infty)}(\mathcal R\Phi)\in H^{m+1/2}(\dM,E)$.
By interior elliptic regularity,
we may assume that $\Phi$ has support in $Z_{[0,\rho)}$,
where $\rho$ is as in \lref{lem:DDo}.
By induction on $m$, we may also assume that $\Phi\in H^m_{\loc}(M,E)$,
and then $\sigma_0R_0\in H^m(Z_{[0,\rho]},F)$.

Since now $\Phi|_{\{\rho\}\times\dM}=0$ and $\Phi$ is in the maximal domain of $D_0$
over $Z_{[0,\rho]}$ by \lref{lem:DDo} \eqref{DDoi}, we can apply \pref{riso} and get
\[
\Phi = S_0D_0\Phi + \exp(-tA)(Q_{[0,\infty)}(\mathcal R\Phi)) =: \Phi_0 + \Phi_1 .
\] 
Since $Q_{[0,\infty)}(\mathcal R\Phi)\in H^{m+1/2}(\dM,E)$,
we have $\Phi_1 \in H^{m+1}(Z_{[0,\rho]},E)$.
Furthermore, $\Phi_0 = S_0D_0\Phi \in H^1(Z_{[0,\rho]},E;B_0)$ and
\begin{equation*}
  (D-\sigma_0R_0)\Phi_0
  = (D-\sigma_0R_0)\Phi - (D-\sigma_0R_0)\Phi_1 \in H^m(Z_{[0,\rho]},F) .
\end{equation*}
By \lref{lem:DDo}~\eqref{DDoii} we have $\Phi_0\in H^{m+1}(Z_{[0,\rho]},E;B_0)$
and hence
\[
  \Phi = \Phi_0 + \Phi_1\in H^{m+1}(Z_{[0,\rho]},E) .
  \qedhere
\]
\end{proof}


\section{Boundary value problems}
\label{secbova}


Assume the Standard Setup~\ref{stase} and that $D$ and $D^*$ are complete.


\subsection{Boundary conditions}
\label{suseboco}

We will use the following notation.
For any subset $U\subset \bigcup_{s'\in\R} H^{s'}(A)$ and any $s\in\R$,
we let
\[
  U^s:=U\cap H^s(A) , \quad
  \check U:=U\cap\check H(A) , \quad
  \hat U:=U\cap\hat H(A) .
\]

\begin{definition}\label{defboucon}\index{boundary condition}
A closed linear subspace $B\subset\check H(A)$
will be called a \emph{boundary condition} for $D$.
\end{definition}

Later we will study the equation $D\Phi=\Psi$ with given $\Psi\in L^2(M,F)$
subject to the boundary condition $\phi:=\mathcal R\Phi\in B$ for $\Phi\in\dom(D_{\max})$.
This explains the terminology; compare also \pref{closedext} below.

For a boundary condition $B\subset\check H(A)$, 
we consider the operators $D_{B,\max}$\index{1DBmax@$D_{B,\max}$} and $D_B$\index{1DB@$D_B$} with domains
\begin{align*}
  \dom(D_{B,\max})
  &= \{\Phi\in\dom(D_{\max}) \mid \mathcal R\Phi\in B\} , \\
  \dom(D_{B})
  &= \{\Phi\in H^1_D(M,E) \mid \mathcal R\Phi\in B\} \\
  &= \{\Phi\in\dom(D_{\max}) \mid \mathcal R\Phi\in H^{1/2}(\dM,E)\},
\end{align*}
and similarly for the formal adjoint $D^*$.
By \eqref{trace} of \tref{domdmax}, $\dom(D_{B,\max})$ is a closed subspace
of $\dom(D_{\max})$.
Since the trace map extends to a bounded linear map $\mathcal R : H^1_D(M,E) \to H^{1/2}(\dM,E)$ and $H^{1/2}(\dM,E) \hookrightarrow \check H(A)$ is a continuous embedding, $\dom(D_{B})$ is also a closed subspace of $H^1_D(M,E)$. 

In particular, $D_{B,\max}$ is a closed operator.
Conversely, Proposition 1.50 of \cite{BBC} reads as follows:

\begin{prop}\label{closedext}
Any closed extension of $D$ between $D_{cc}$ and $D_{\max}$
is of the form $D_{B,\max}$,
where $B\subset\check H(A)$ is a closed subspace.
\end{prop}

\begin{proof}
Let $\bar D\subset D_{\max}$ be a closed extension of $D_{cc}$.
Then $\bar D$ extends the minimal closed extension $D_{\min}=D_{0,\max}$
of $D_{cc}$, that is,
\[
  \dom(D_{\min})
  = \{ \Phi \in \dom(D_{\max}) \mid \mathcal R(\Phi) = 0 \}
  \subset \dom(\bar D) ,
\]
compare \cref{domdmin}.
It follows that
\[
  \dom\bar D = \{ \Phi \in \dom(D_{\max}) \mid \mathcal R(\Phi) \in B \} ,
\]
where $B\subset\check H(A)$ is the space of all $\phi$ such that there exists
a $\Phi\in\dom\bar D$ with $\mathcal R(\Phi)=\phi$.
In particular, $\mathcal E(\phi)$ is contained in $\dom\bar D$,
for any $\phi\in B$.
Now let $(\varphi_j)$ be a sequence in $B$
converging to $\varphi$ in $\check H(A)$.
Then $(\mathcal E(\varphi_j))$ converges to $\mathcal E(\varphi)$
in $\dom(D_{\max})$, by \lref{extfin2}.
Since $\bar D$ is closed,
we have $\mathcal E(\varphi)\in\dom\bar D$.
It follows that $\varphi\in B$ and hence that $B$ is closed in $\check H(A)$.
\end{proof}

\begin{lemma}\label{h1ngn}
Let $B$ be a boundary condition.
Then $B\subset H^{1/2}(\dM,E)$ if and only if $D_B=D_{B,\max}$,
and then there exists a constant $C>0$ such that 
\[
  \|\Phi\|_{H^1_D(M)}^2
  \leq C \cdot \left( 
  \|\Phi\|_{L^2(M)}^2 + \|D\Phi\|_{L^2(M)}^2 \right) 
\]
for all $\Phi\in\dom(D_B)$.
\end{lemma}

\begin{proof}
If $D_{B,\max}=D_{B}$
then $B\subset H^{1/2}(\dM,E)$ by \eqref{reg} of \tref{domdmax}.
Conversely, if
$B\subset H^{1/2}(\dM,E)$, then
\begin{align*}
\dom(D_{B,\max})
&=
\{\Phi\in\dom(D_{\max}) \mid \mathcal R\Phi\in B\} \\
&\subset
\{\Phi\in\dom(D_{\max}) \mid \mathcal R\Phi\in H^{1/2}(\dM,E)\} \\
&=
H^1_D(M,E) ,
\end{align*}
again by \eqref{reg} of \tref{domdmax}, and therefore $D_{B,\max}=D_{B}$.

Suppose now that $D_{B,\max}=D_{B}$.
Since $\dom(D_{B,\max})$ is a closed subspace of $\dom(D_{\max})$
and $\dom(D_{B})$ is a closed subspace of $H^1_D(M,E)$,
we conclude that the $H^1_D$-norm and the graph norm for $D$
are equivalent on $\dom(D_{B,\max})=\dom(D_{B})$.
This shows the asserted inequality.
\end{proof}


\subsection{Adjoint boundary condition}
\label{suseabc}

For any boundary condition $B$,
we have $D_{cc} \subset D_{B,\max}$.
Hence the $L^2$-adjoint operators satisfy
\begin{equation}\label{dbadin}
  (D_{B,\max})^{\ad} 
  \subset(D_{cc})^{\ad} = (D^*)_{\max} .
\end{equation}
From \eqref{eq:ParInt}, we conclude that
\begin{multline}\label{dbaddo}
  \dom((D_{B,\max})^{\ad}) \\
  = \{ \Psi\in\dom((D^*)_{\max}) \mid
  \text{$(\sigma _0\mathcal R\Phi,\mathcal R\Psi)=0$
  for all $\Phi\in\dom(D_{B,\max})$} \} .
\end{multline}
Now for any $\phi\in B$ there is a $\Phi\in\dom(D_{B,\max})$ with $\mathcal R\Phi=\phi$.
Therefore
\begin{equation}\label{dba}
  (D_{B,\max})^{\ad} = (D^*)_{\Bad,\max} 
\end{equation}
with\index{1Bad@$\Bad$}
\begin{equation}\label{dbabc}  
  \Bad := \{ \psi\in\check H(\tilde A) 
  \mid \text{$(\sigma _0\phi,\psi)=0$ for all $\phi\in B$} \} .
\end{equation}
By \lref{duality}, $\Bad$ is a closed subspace of $\check H(\tilde A)$.
In other words, it is a boundary condition for $D^*$.

\begin{definition}\label{defabc}\index{adjoint boundary condition}
We call $\Bad$ the boundary condition {\em adjoint to} $B$.
\end{definition}

The perfect pairing between $H^{1/2}(\dM,E)$ and $H^{-1/2}(\dM,E)$
as in \eqref{perfpair} on page~\pageref{perfpair}
and the analogous perfect pairing between $\check H(A)$ and $\hat H(A)$
as in \eqref{omega} 
coincide on the intersection $H^{1/2}(\dM,E)\times \hat H(A)$.
For a boundary condition $B$ which is contained in $H^{1/2}(\dM,E)$,
it follows that
\begin{equation}\label{ba}
  \sigma_0^*(\Bad) = B^0 \cap \hat H(A) ,
\end{equation}
where the superscript $0$ indicates the annihilator in $H^{-1/2}(\dM,F)$.


\subsection{Elliptic boundary conditions}
\label{susebc}

\begin{definition}\label{defellbou}\index{elliptic boundary condition}
A linear subspace $B\subset H^{1/2}(\dM,E)$ is called
an {\em elliptic boundary condition} for $A$ (or for $D$)
if there is an $L^2$-orthogonal decomposition
\[
  L^2(\dM,E) = V_- \oplus W_- \oplus V_+ \oplus W_+ 
\]
such that
\begin{enumerate}[(i)]
\item\label{ebcw}
$W_-$ and $W_+$ are finite-dimensional and contained in $H^{1/2}(\dM,E)$;
\item\label{ebcv}
$V_- \oplus W_-\subset L^2_{(-\infty,a]}(A)$
and $V_+ \oplus W_+\subset L^2_{[-a,\infty)}(A)$ for some $a\ge 0$;
\item\label{ebcg}
there is a bounded (with respect to $\LzdM{\cdot}$) linear map $g:V_-\to V_+$ with 
\[
  g(V_-^{1/2}) \subset V_+^{1/2}
  \quad\text{and}\quad
  g^*(V_+^{1/2}) \subset V_-^{1/2}
\]
such that
\[
  B = W_+ \oplus \{ v+gv \mid v \in V_-^{1/2} \} .
\]
\end{enumerate}
\end{definition}

\begin{example}\label{aps}
If we put $V_-=L^2_{(-\infty,0)}(A)$, $V_+=L^2_{[0,\infty)}(A)$, 
$W_-=W_+=0$, and $g=0$, then we have
\[
  B = \BAPS := H_{(-\infty,0)}^{1/2}(A) .
\]
This is the well-known {\em Atiyah-Patodi-Singer boundary
condition}\index{Atiyah-Patodi-Singer boundary condition}
as introduced in \cite{APS}.
\end{example}

We will use the notation $\Gamma(g) := \{ v+gv \mid v \in V_- \}$
to denote the graph of $g$ and similarly for the restriction of $g$ to $V_-^{1/2}$,
$\Gamma(g)^{1/2} := \{ v+gv \mid v \in V_-^{1/2} \}$.

\begin{lemma}\label{lem:Bperp}
Let $B \subset H^{1/2}(\dM,E)$ be an elliptic boundary condition.
Then
\[
V_- \oplus V_+ = \Gamma(g) \oplus \Gamma(-g^*) ,
\]
where both decompositions are $L^2$-orthogonal.
\end{lemma}

\begin{proof}
The decomposition $V_- \oplus V_+$ is $L^2$-orthogonal by assumption.
With respect to the splitting $V_- \oplus V_+$ we have
\[
\Gamma(g) = \left\lbrace {v_- \choose gv_-}\,|\,v_-\in V_-\right\rbrace
\quad\text{and}\quad
\Gamma(-g^*) = \left\lbrace {-g^*v_+ \choose v_+}\,|\,v_+\in V_+\right\rbrace
\]
By the definition og $g^*$,
the decomposition on the right hand side is $L^2$-orthogonal.
An arbitrary element of $\Gamma(g) + \Gamma(-g^*)$ is given by
\[
{v_--g^*v_+ \choose gv_-+v_+} 
= \begin{pmatrix}\id & -g^* \cr g & \id\end{pmatrix}
{v_- \choose v_+} .
\]
Since
\[
\begin{pmatrix}\id & g^* \cr -g & \id\end{pmatrix}
\begin{pmatrix}\id & -g^* \cr g & \id\end{pmatrix}
=
\begin{pmatrix}\id & -g^* \cr g & \id\end{pmatrix}
\begin{pmatrix}\id & g^* \cr -g & \id\end{pmatrix}
= 
\begin{pmatrix}\id+g^*g & 0 \cr 0 & \id+gg^*\end{pmatrix}
\]
we see that $\begin{pmatrix}\id & -g^* \cr g & \id\end{pmatrix}$
is an isomorphism with inverse
$\begin{pmatrix}\id & g^* \cr -g & \id\end{pmatrix}
\begin{pmatrix}\id+g^*g & 0 \cr 0 & \id+gg^*\end{pmatrix}^{-1}$.
The decomposition follows.
\end{proof}

\begin{remark}\label{rem:Bperp}
With respect to the splitting $V_+ \oplus V_-$, 
the orthogonal projection onto $\Gamma(g)$ is given by 
\[
\begin{pmatrix}\id &  -g^*\cr g & \id\end{pmatrix}
{v_- \choose v_+} 
\mapsto
\begin{pmatrix}\id & 0 \cr g & 0\end{pmatrix}
{v_- \choose v_+} ,
\]
hence by the matrix
\begin{equation}
\begin{pmatrix}\id & 0 \cr g & 0\end{pmatrix}
\begin{pmatrix}\id &  -g^*\cr g & \id\end{pmatrix}^{-1} .
\label{eq:ProjGamma}
\end{equation}
\end{remark}

\begin{lemma}\label{lem:Bclosed}
Let $B \subset H^{1/2}(\dM,E)$ be an elliptic boundary condition.
Then
\begin{enumerate}[(i)]
\item \label{Bc1}
the spaces $B, V_\pm^{1/2}$, and $W_\pm$ are closed in $H^{1/2}(\dM,E)$ and
\[
H^{1/2}(\dM,E) = V_-^{1/2} \oplus W_- \oplus V_+^{1/2} \oplus W_+  ;
\]
\item \label{Bc2}
the spaces $V_\pm^{1/2}$ are dense in $V_\pm$;
\item \label{Bc3}
the map $g$ restricts to a continuous (w.\ r.\ t.\ $\HhdM{\cdot}$)
linear map $g:V_-^{1/2}\to V_+^{1/2}$, and similarly for $g^*$;
\item \label{Bc4}
the $L^2$-orthogonal projections
\[
  \pi_\pm:L^2(\dM,E) \twoheadrightarrow  V_\pm \subset L^2(\dM,E)
\]
restrict to continuous projections
\[
  \pi_\pm^{1/2}:H^{1/2}(\dM,E)
  \twoheadrightarrow  V_\pm^{1/2} \subset H^{1/2}(\dM,E) .
\]
\end{enumerate}
\end{lemma}

\begin{proof}
We start by proving \eqref{Bc4}.
Since $V_+\oplus W_+ \subset L^2_{[-a,\infty)}(A)$ is the orthogonal complement
of $V_-\oplus W_-$ in $L^2(\dM,E)$, we have
\[
L^2_{(-\infty,-a)}(A) \subset V_-\oplus W_- \subset L^2_{(-\infty,a]}(A) .
\]
Hence the orthogonal complement $F$ of $L^2_{(-\infty,-a)}(A)$ in $V_-\oplus W_-$
is contained in $L^2_{[-a,a]}(A)$.
In particular, $F$ is finite-dimensional and contained in $H^{1/2}(\dM,E)$.
Thus the orthogonal projection
$\pi_F:L^2(\dM,E) \twoheadrightarrow F \subset L^2(\dM,E)$
restricts to a continuous projection
$H^{1/2}(\dM,E) \twoheadrightarrow F \subset H^{1/2}(\dM,E)$,
and similarly for the orthogonal projection $\pi_{W_-}$ onto $W_-$.
Since the orthogonal projection
$\pi_-:L^2(\dM,E) \twoheadrightarrow V_- \subset L^2(\dM,E)$
is given by $\pi_- = Q_{(-\infty,-a)} + \pi_F - \pi_{W_-}$,
it restricts to a continuous projection
$H^{1/2}(\dM,E) \twoheadrightarrow V_-^{1/2} \subset H^{1/2}(\dM,E)$ as asserted.
The case $V_+$ is analogous. 

Clearly, \eqref{Bc4} implies \eqref{Bc2}
and shows that $V_\pm^{1/2}$ is closed in $H^{1/2}(\dM,E)$.
Since they are finite-dimensional, $W_+$ and $W_-$ are also closed in $H^{1/2}(\dM,E)$.
Moreover, we obtain the decomposition of $H^{1/2}(\dM,E)$ as asserted in \eqref{Bc1}.

We now show \eqref{Bc3}.
Let $\phi_j \to \phi$ in $V_-^{1/2}$ and $g\phi_j \to \psi$ in $V_+^{1/2}$ as $j\to\infty$.
Then $\phi_j \to \phi$ in $V_-$ and, 
since $g$ is continuous on $V_-$, $g\phi_j \to g\phi$ in $V_+$.
Thus $\psi=g\phi$.
This shows that the graph $\Gamma(g)^{1/2}$ of $g:V_-^{1/2}\to V_+^{1/2}$ is closed.
The closed graph theorem implies \eqref{Bc3}.

Since $W_+$ is closed in $H^{1/2}(\dM,E)$ and $\Gamma(g)^{1/2}$
is closed in $V_-^{1/2}\oplus V_+^{1/2}$,
 it follows that $B$ is closed in $H^{1/2}(\dM,E)$ as well.
This concludes the proof of \eqref{Bc1} and of the lemma.
\end{proof}

\begin{lemma}\label{lemHhHc}
Let $B \subset H^{1/2}(\dM,E)$ be an elliptic boundary condition.
Then there exists a constant $C>0$ such that 
\[
\HhdM{\phi} \leq C\cdot \HcdM{\phi}
\]
for all $\phi\in B$.
\end{lemma}

Since we always have $\HcdM{\cdot}\leq\HhdM{\cdot}$, Lemma~\ref{lemHhHc}
says that on $B$ the $H^{1/2}$ and $\check{H}$-norms are equivalent.
In particular, $B$ is also closed in $\check{H}(A)$ and hence a boundary condition
in the sense of \dref{defboucon}.

\begin{proof}
We apply \pref{HormPeet} with $X=B$ (with the $H^{1/2}$-norm), 
$Y=H_{(-\infty,\Lambda]}^{1/2}(A)$, and $Z=H^{-1/2}(\dM,E)$.
The linear map
\[
L: X=B \hookrightarrow H^{1/2}(\dM,E)
\xrightarrow{Q_{(-\infty,\Lambda]}} H^{1/2}_{(-\infty,\Lambda]}(A) = Y
\]
is bounded and the inclusion
\[
K: X=B \hookrightarrow H^{1/2}(\dM,E) \hookrightarrow H^{-1/2}(\dM,E) = Z
\] 
is compact.
Since $K$ is injective, $\ker K \cap \ker L = \{0\}$.
We need to show that the kernel of $L$ is finite-dimensional and that its range is closed.
Then the implication $\eqref{HP1} \Rightarrow \mbox{(iv)}$ of \pref{HormPeet}
yields the desired inequality.

Without loss of generality we assume that $\Lambda \geq a$,
where $\Lambda$ is as in \eqref {eq:defHc} and $a$ as in \dref{defellbou}.
Let $w+v+gv\in \ker(L)$, $w\in W_+$, $v\in V_-$.
Then
\[
0= Q_{(-\infty,\Lambda]}(w+v+gv)
= Q_{(-\infty,\Lambda]}(w) + v + Q_{[-a,\Lambda]}(gv),
\]
hence
\begin{align*}
v &= -Q_{(-\infty,\Lambda]}(w) - Q_{[-a,\Lambda]}(gv) \\
& \in Q_{(-\infty,\Lambda]}(W_+) + Q_{[-a,\Lambda]}(H^{1/2}(\dM,E)) =:F .
\end{align*}
Thus $\ker(L) \subset W_+ + \Gamma(g|_F)$ which is finite-dimensional.
Here $\Gamma(g|_F) = \{x+gx \mid x\in F\}$ denotes the graph of $g$
restricted to the finite-dimensional space $F$.

The image of $L$ is given by
\[
\im(L) = Q_{(-\infty,\Lambda]}(B) = 
\underbrace{Q_{(-\infty,\Lambda]}(W_+)}_{\mathrm{finite\,\, dimensional}}
+ \underbrace{\Gamma(Q_{(-\infty,\Lambda]}\circ g:V_-^{1/2}\to Y)}_{\mathrm{closed}}
\]
which is closed because $g$ (and hence $Q_{(-\infty,\Lambda]}\circ g$)
is $H^{1/2}$--bounded by \lref{lem:Bclosed}\eqref{Bc3}.
\end{proof}

\begin{thm}\label{ellbou}
Assume the Standard Setup~\ref{stase}.
Then, for a linear subspace $B\subset H^{1/2}(\dM,E)$,
the following are equivalent: 
\begin{enumerate}[(i)]
\item\label{ellbouc}
$B$ is closed in $\check H(A)$ and $B^{\ad}\subset H^{1/2}(\dM,F)$.
\item\label{ellboua}
For any $a\in\R$, we can choose orthogonal decompositions
\[
  L^2_{(-\infty,a)}(A) = V_-\oplus W_- 
  \quad\text{and}\quad
  L^2_{[a,\infty)}(A) = V_+\oplus W_+
\]
and $g:V_-\to V_+$ as in \dref{defellbou} such that
\[
  B = W_+ \oplus \{ v+gv \mid v \in V_-^{1/2} \} .
\]
\item\label{ellbouh}
$B$ is an elliptic boundary condition.
\end{enumerate}
Moreover, for an elliptic boundary condition $B$, the adjoint boundary condition
$\Bad\subset\check H(\tilde A)=(\sigma _0^*)^{-1}(\hat H(A))$
is also elliptic (for $D^*$) with
\[
  \sigma _0^*(\Bad) = 
  W_- \oplus \{ u-g^*u \mid u\in V_+^{1/2} \} .
\]
\end{thm}

\begin{remark}\label{remellbou}
In \cite{BBC}, the point of departure for elliptic boundary conditions
is Property~\eqref{ellbouc}.
The equivalence between \eqref{ellbouc} and \eqref{ellboua} was already observed there.
Since our setup is slightly different and more general, we repeat the argument. 
\end{remark}

\begin{proof}
It is trivial that \eqref{ellboua} implies \eqref{ellbouh}.
We show that both \eqref{ellbouc} and the assertion on the adjoint boundary condition
are implied by \eqref{ellbouh}.
By \lref{lemHhHc}, $B$ is closed in $\check H(A)$.
Denote the closure of $V_\pm$ with respect to $\HmhdM{\cdot}$ by $V_\pm^{-1/2}$.
Then we have
\[
  \hat H(A) = V_-^{-1/2} \oplus W_- \oplus V_+^{1/2} \oplus W_+ .
\]
By \eqref{ba}, an element 
\[
  v_-+w_-+v_++w_+ \in V_-^{-1/2}\oplus W_- \oplus V_+^{1/2}\oplus W_+
\]
lies in $\sigma _0^*(\Bad)$ if and only if
\[
  (\tilde w_+ + \tilde v_- + g\tilde v_-,v_-+w_-+v_++w_+) = 0
\]
for all $\tilde w_+\in W_+, \tilde v_-\in V_-^{1/2}$, i.e.\ if and only if
\[
  (\tilde w_+,w_+) + (\tilde v_-,v_-) + (g\tilde v_-,v_+)=0
\]
for all $\tilde w_+\in W_+, \tilde v_-\in V_-^{1/2}$, i.e., if and only if it is of the form
\[
  -g^*v_++w_-+v_+ , \quad\mbox{where $w_-\in W_-$, $v_+\in V_+^{1/2}$ .}
\]
Hence
\[
  \sigma _0^*(\Bad) = W_- \oplus \{ v_+-g^*v_+ \mid v_+\in V_+^{1/2} \}
\]
as asserted.
In particular, $\Bad \subset (\sigma _0^*)^{-1}(H^{1/2}(\dM,E)) = H^{1/2}(\dM,F)$.

It remains to show that \eqref{ellbouc} implies \eqref{ellboua}.
Fix $a\in\R$.
Consider the spaces $W_+$, $V_-^{1/2}$, $V_+^{1/2}$, $V_+$,
and $V_+^{-1/2}$ as in \lref{happylem} below.
In particular, 
\begin{equation}
B = W_+ \oplus \{v+gv\mid v\in V_-^{1/2}\}.
\label{eq:B}
\end{equation}
By \eqref{ba}, $C:=\sigma _0^*(B^{\ad})=B^0\cap\hat H(A)$.
In particular, $C$ is closed in $\hat H(A)=\check H(-A)$.
By assumption, $C\subset H^{1/2}(\dM,E)$,
hence we can apply \lref{happylem} with $-A$ instead of $A$
and $-a-\eps$ instead of $a$,
where $\eps>0$ is chosen so small that $A$ has no eigenvalues in $[a-\eps,a)$.
We obtain subspaces
\begin{enumerate}[ (a)]
\item 
$W_- = C\cap H^{1/2}_{[-a+\eps,\infty)}(-A)
= C\cap H^{1/2}_{(-\infty,a-\eps]}(A)
= C\cap H^{1/2}_{(-\infty,a)}(A)$;
\item
$U_+^{1/2} = Q_{[a,\infty)}(C)$;
\item
$U_-^{-1/2}$, the annihilator of $W_-$ in $H^{-1/2}_{(-\infty,a)}(A)$;
\item
$U_-=U_-^{-1/2}\cap L^2(\dM,E)$ and $U_-^{1/2}=U_-^{-1/2}\cap H^{1/2}(\dM,E)$.
\end{enumerate}
Moreover,
\begin{equation}
C=W_- \oplus \{u+hu\mid u\in U_+^{1/2}\} .
\label{eq:C}
\end{equation}
We have
\begin{align*}
W_- 
&=
C \cap H^{1/2}_{(-\infty,a)}(A) \\
&=
B^0 \cap \hat H(A) \cap H^{1/2}_{(-\infty,a)}(A) \\
&=
B^0 \cap \hat H(A) \cap H^{-1/2}_{(-\infty,a)}(A) ,
\end{align*}
where we use that $B^0\cap\hat H(A) \subset H^{1/2}(\dM,E)$
to pass from the second to the third line.
Now $\hat H(A) \cap H^{-1/2}_{(-\infty,a)}(A) = H^{-1/2}_{(-\infty,a)}(A)$
by the definition of $\hat H(A)$.
We conclude that
\begin{align*}
W_-
&=
B^0 \cap H^{-1/2}_{(-\infty,a)}(A)\\
&=
\{x\in H^{-1/2}_{(-\infty,a)}(A) \mid
\text{$(x,w+v+gv)=0$ for all $w\in W_+$, $v\in V_-^{1/2}$ }\}\\
&=
\{x\in H^{-1/2}_{(-\infty,a)}(A) \mid \text{$(x,v)=0$ for all $v\in V_-^{1/2}$} \} .
\end{align*}
Hence $W_-$ is the annihilator of $V_-^{1/2}$ in $H^{-1/2}_{(-\infty,a)}(A)$.
By \lref{happylem},
$W_-$ is also the annihilator of $U_-^{1/2}$ in $H^{-1/2}_{(-\infty,a)}(A)$.
Thus
\[
U_-^{1/2} = V_-^{1/2} .
\]
By interchanging the roles of $B$ and $C$, we also get
\[
U_+^{1/2} = V_+^{1/2} .
\]
We set
\[
V_- := U_- \mbox{ and } V_-^{-1/2} := U_-^{-1/2} 
\]
and get by \lref{happylem}~\eqref{t3}
\[
H^{-1/2}(\dM,E) = H_{(-\infty,a)}^{-1/2}(A) \oplus H_{[a,\infty)}^{-1/2}(A) =
W_- \oplus V_-^{-1/2} \oplus W_+ \oplus V_+^{-1/2}
\]
and similarly for $L^2(\dM,E)$ and $H^{1/2}(\dM,E)$.

It follows that the annihilators of $B$  and $C$ in $H^{-1/2}(\dM,E)$ are given by
\begin{align*}
  B^0 &= W_- \oplus \{ u- g'u \mid u\in V_+^{-1/2} \} ,\\
  C^0 &= W_+ \oplus \{ v- h'v \mid u\in V_-^{-1/2} \} ,
\end{align*}
where $g':V_+^{-1/2}\to V_-^{-1/2}$ is the dual map of $g$ and similarly for $h$.
Furthermore, we get that the restriction of $-h'$ to $V_-^{1/2}$ equals $g$
and the restriction of $-g'$ to $V_-^{1/2}$ coincides with $h$.
By interpolation, $-h'$ restricts to a continuous linear map $V_-\to V_+$,
again denoted by $g$, and $-g'$ restricts to $-g^*$.
This shows that \eqref{ellbouc} implies \eqref{ellboua}.
\end{proof}

\begin{lemma}\label{happylem}
Let $B\subset \check H(A)$ be a boundary condition
which is contained in $H^{1/2}(\dM, E)$.
Let $a\in\R$.
Define
\begin{enumerate}[ (a)]
\item 
$W_+:= B\cap H^{1/2}_{[a,\infty)}(A)$;
\item
$V_-^{1/2}:= Q_{(-\infty,a)}(B)$;
\item
$V_+^{-1/2}$ to be the annihilator of $W_+$ in $H^{-1/2}_{[a,\infty)}(A)$;
\item
$V_+:=V_+^{-1/2}\cap L^2(\dM,E)$ and $V_+^{1/2}:=V_+^{-1/2}\cap H^{1/2}(\dM,E)$.
\end{enumerate}
Then 
\begin{enumerate}[(i)]
\item\label{t1}
$W_+$ is finite-dimensional and equals the annihilator
of $V_+^{1/2}$ in $H_{[a,\infty)}^{-1/2}(A)$;
\item\label{t2}
$V_-^{1/2}$ is a closed subspace of $H^{1/2}_{(-\infty,a)}(A)$;
\item\label{t3}
$H_{[a,\infty)}^{-1/2}(A) = W_+ \oplus V_+^{-1/2}$, \\
$L^2_{[a,\infty)}(A) = W_+ \oplus V_+$,\\
$H_{[a,\infty)}^{1/2}(A) = W_+ \oplus V_+^{1/2}$,\\
where the second decomposition is $L^2$-orthogonal;
\item\label{t4}
the perfect pairing between $H^{1/2}(\dM,E)$ and $H^{-1/2}(\dM,E)$
restricts to a perfect pairing between $V_+^{1/2}$ and $V_+^{-1/2}$;
\item\label{t5}
there exists a continuous linear map $g:V_-^{1/2} \to V_+^{1/2}$ such that 
\[
B = W_+ \oplus \{v+gv\mid v\in V_-^{1/2}\}.
\]
\end{enumerate}
\end{lemma}

\begin{proof}
Since $B$ is closed in $\check H(A)$ and contained in $H^{1/2}(\dM,E)$,
there is an estimate
\begin{equation}
\begin{split}
\|\phi\|_{H^{1/2}(\dM)}
&\le
C \|\phi\|_{\check H(A)} \\
&\le
C'( \|Q_{(-\infty,a)}\phi\|_{H^{1/2}(\dM)} + \|Q_{[a,\infty)}\phi\|_{H^{-1/2}(\dM)} ) 
\label{pfellbou2}
\end{split}
\end{equation}
for all $\phi\in B$.
Since $B\subset H^{1/2}(\dM,E)$ is closed, the linear map
\[
  Q_{[a,\infty)}:B\to H^{1/2}_{[a,\infty)}(A) \hookrightarrow H^{-1/2}_{[a,\infty)}(A)
\]
is compact by Rellich's theorem.
It follows from \pref{HormPeet} that $Q_{(-\infty,a)}:B\to H^{1/2}_{(-\infty,a)}(A)$
has finite-dimensional kernel $W_+$ and closed image $V_-^{1/2}$.
This shows \eqref{t2} and the first part of \eqref{t1}.

Since $W_+\subset H_{[a,\infty)}^{1/2}(A)$ is finite-dimensional,
\lref{lem:Triple} implies \eqref{t3}, \eqref{t4}, and the second part of \eqref{t1}.

Let $G$ be the $L^2$-orthogonal complement of $W_+$ in $B$.
Then $B=W_+\oplus G$ and $Q_{(-\infty,a)}:G\to V_-^{1/2}$ is an isomorphism.
Compose its inverse $V_-^{1/2} \to G$ with $Q_{[a,\infty)}: G \to H^{1/2}_{[a,\infty)}(A)$
to obtain a continuous linear map $g:V_-^{1/2}\to H^{1/2}_{[a,\infty)}(A)$.
Since $G$ is $L^2$-orthogonal to $W_+$ and since $V_+$
is the $L^2$-orthogonal complement of $W_+$ in $L^2_{[a,\infty)}(A)$ by \eqref{t3},
the map $g$ takes values in $V_+^{1/2}$.
In conclusion,
\[
  B = W_+ \oplus \{ v+ gv \mid v\in V_-^{1/2} \} ,
\]
where $g$ is now considered as a continuous linear map $V_-^{1/2} \to V_+^{1/2}$.
\end{proof}


\subsection{Boundary regularity}
\label{suseregu}

\begin{lemma}\label{lem:Breg}
Let $B\subset H^{1/2}(\dM,E)$ be an elliptic boundary condition.
Let $s\ge 1/2$.
Then the following are equivalent:
\begin{enumerate}[(i)]
\item \label{Breg1}
There exist $V_\pm$, $W_\pm$, and $g$ for $B$ as in \dref{defellbou}
such that $W_+\subset H^s(\dM,E)$ and $g(V_-^s) \subset V_+^s$.
\item \label{Breg2}
For all $V_\pm$, $W_\pm$, and $g$ for $B$ as in \dref{defellbou},
we have that $W_+\subset H^s(\dM,E)$ and $g(V_-^s) \subset V_+^s$.
\item \label{Breg3}
For all $\phi\in B$ with $Q_{(-\infty,0)}(\phi)\in H^s(\dM,E)$,
we have $\phi\in H^s(\dM,E)$.
\end{enumerate}
\end{lemma}

\begin{proof}
Clearly, \eqref{Breg2} implies \eqref{Breg1}.
We show that \eqref{Breg1} implies \eqref{Breg3}.
Let $V_\pm$, $W_\pm$, and $g$ be as in \eqref{Breg1}.
Let $\phi\in B$ with $Q_{(-\infty,0)}(\phi)\in H^s(\dM,E)$.
We have to show $\phi\in H^s(\dM,E)$.
Write $\phi=w_+ + v_- + gv_-$ with $w_+\in W_+$ and $v_-\in V_-$.
Since $V_+ \subset L^2_{[-a,\infty)}(A)$ for some $a\ge 0$, we have
\[
Q_{(-\infty,-a)}(\phi) = Q_{(-\infty,-a)}(w_++v_-)
\]
and thus, using $V_-\subset L^2_{(-\infty,a]}(A)$,
\begin{align*}
v_- 
&= 
Q_{(-\infty,a]}(v_-)  \\
&=
Q_{(-\infty,-a)}(v_-) + Q_{[-a,a]}(v_-) \\
&=
Q_{(-\infty,-a)}(\phi) - Q_{(-\infty,-a)}(w_+) + Q_{[-a,a]}(v_-) \\
&=
Q_{(-\infty,0)}(\phi) - Q_{[-a,0)}(\phi) - Q_{(-\infty,-a)}(w_+) + Q_{[-a,a]}(v_-) .
\end{align*}
Since for bounded intervals $I$ the projection $Q_I$ takes values in smooth sections
and since $w_+\in H^s(\dM,E)$,
all four terms on the right hand side are contained in $H^s(\dM,E)$.
Hence $v_-\in H^s(\dM,E)$.
By the assumption on $g$,
we also have $gv_-\in H^s(\dM,E)$ and therefore $\phi\in H^s(\dM,E)$.

It remains to show that \eqref{Breg3} implies \eqref{Breg2}.
Let $V_\pm$, $W_\pm$, and $g$ be for $B$ as in \dref{defellbou}.
Since $W_+\subset L^2_{[-a,\infty)}(A)$, we have for $w_+\in W_+$ that
\[
  Q_{(-\infty,0)}(w_+) = Q_{[-a,0)}(w_+) \in \Cu(\dM,E) .
\]
By \eqref{Breg3}, $w_+\in H^s(\dM,E)$.
Hence $W_+\subset H^s(\dM,E)$.

Now let $v_-\in V_-^s$.
Then $v_-+gv_-\in B$ and
\[
  Q_{(-\infty,0)}(v_-+gv_-) = Q_{(-\infty,0)}(v_-)+Q_{[-a,0)}(gv_-) \in H^s(\dM,E) .
\]
Again by \eqref{Breg3}, $v_-+gv_- \in H^s(\dM,E)$ and therefore $gv_- \in H^s(\dM,E)$.
This shows $g(V_-^s) \subset V_+^s$.
\end{proof}

\begin{definition}\index{regular1@$s$-regular}\index{regular2@$\infty$-regular}
An elliptic boundary condition $B\subset H^{1/2}(\dM,E)$
is called $(s+1/2)$-\emph{regular}
if the assertions in \lref{lem:Breg} hold for $B$ and for $\Bad$.
If $B$ is $(s+1/2)$-regular for all $s\ge 1/2$,
then $B$ is called $\infty$-\emph{regular}.
\end{definition}

\begin{remarks}\label{remreg}
(a)
An elliptic boundary condition $B\subset H^{1/2}(\dM,E)$ is $(s+1/2)$-regular
if and only if for some (or equivalently all) $V_\pm$, $W_\pm$, and $g$
as in \dref{defellbou} we have $W_+\cup W_- \subset H^s(\dM,E)$
and $g(V_-^s) \subset V_+^s$ as well as $g^*(V_+^s) \subset V_-^s$.
In this case, $g$ and $g^*$ are continuous with respect to the $H^s$-norm.

(b)
If $B$ is $(s+1/2)$-regular, then so is $\Bad$.

(c)
By definition, every elliptic boundary condition is $1$-regular.
By interpolation one sees that if $B$ is $(s+1/2)$-regular,
then $B$ is also $t$-regular for all $1\le t \le s+1/2$.
\end{remarks}

\begin{thm}[Higher boundary regularity]\label{thm:regvc}
Assume the Standard Setup~\ref{stase} and that $D$ and $D^*$ are complete.
Let $m\in\N$ and $B\subset H^{1/2}(\dM,E)$ be an $m$-regular
elliptic boundary condition for $D$.

Then, for all $0\le k< m$ and $\Phi\in\dom(D_{\max})$, we have
\[
\mathcal{R}(\Phi) \in B \;\mbox{ and }\; D_{\max}\Phi\in H^k_{\loc}(M,F) 
\Longrightarrow \Phi \in H^{k+1}_{\loc}(M,E).
\]
\end{thm}

\begin{proof}
Since $B$ is an $m$-regular elliptic boundary condition,
\tref{ellbou}, \lref{lem:Breg}, and (c) in Remark~\ref{remreg} above
imply that we have an orthogonal decomposition
\[
  H_{(-\infty,0)}(A) = V_- \oplus W_- 
  \quad\text{and}\quad
  H_{[0,\infty)}(A) = V_+ \oplus W_+
\]
with $W_\pm\subset H^{m-1/2}(\dM,E)$ of finite dimension
and a bounded linear map $g:V_-\to V_+$ with $g(V_-^{k+1/2})\subset V_+^{k+1/2}$
and $g^*(V_+^{k+1/2})\subset V_-^{k+1/2}$ for all $0\le k<m$ such that
\begin{align*}
  B &= W_+ \oplus \{ v+gv \mid v\in V_-^{1/2} \} , \\
  \sigma _0^*(\Bad) &=  W_- \oplus \{ u-g^*u \mid u\in V_+^{1/2} \} .
\end{align*}
It is no loss of generality to enlarge $B$ by extending $g$ along $W_-$ by $0$.
Hence we may assume that $V_-=H_{(-\infty,0)}(A)$ and $W_-=\{0\}$.

We identify a small collar about the boundary with $\Zrq=[0,\rho]\times \dM$,
where $0<\rho<r$ and $r$ is as in \lref{adapted}.
Assume first that we have constant coefficients $D=D_0$ over $\Zrq$.

At the ``right end'' $t=\rho$ we impose the boundary condition $B_\rho=H_{[0,\infty)}(A)$, 
at the ``left end'' $t=0$ we impose $B$.
Since $V_-=H_{(-\infty,0)}(A)$,
we have for $\Phi$ with $\mathcal{R}(\Phi)=\Phi(0)\in B$,
\[
Q_{[0,\infty)}(\mathcal{R}(\Phi))
= gQ_{(-\infty,0)}(\mathcal{R}(\Phi)) + P_{W_+}(\mathcal{R}(\Phi)) ,
\]
where $P_{W_+}:L^2(\dM,E)\to W_+$ is the orthogonal projection.
Using this and \pref{riso} one easily checks that the maps
\begin{align*}
  \mathcal{F}_{B,0}:\, & H^{k+1}(\Zrq,E;B\oplus B_\rho) \to H^k(\Zrq,F) \oplus W_+ , \\
  &\mathcal{F}_{B,0}(\Phi) := (D_0\Phi, P_{W_+}\mathcal{R}(\Phi)) ,
\end{align*}
and
\begin{align*}
  & \mathcal{G}_{B,0}:\, H^k(\Zrq,F) \oplus W_+  \to H^{k+1}(\Zrq,E;B\oplus B_\rho)  , \\
  \mathcal{G}_{B,0} &(\Psi,\phi) := S_0\Psi
  + \exp(-tA)(gQ_{(-\infty,0)}(\mathcal{R}(S_0\Psi)) + \phi),
\end{align*}
are inverse to each other for all $k\in\N$.
Hence $\mathcal{F}_{B,0}$ is an isomorphism for all $k\in\N$.

Passing now to variable coefficients, we compare with constant coefficients.
We still let $B_\rho$ on the right side of $\Zrq$ be defined with respect to $A$.
Arguing as in the proof of \lref{lem:DDo}, we find that 
\begin{align*}
  \mathcal{F}_{B}: &H^{k+1}(\Zrq,E;B\oplus B_\rho) \to H^k(\Zrq,E) \oplus W_+ , \\
  &\mathcal{F}_{B}(\Phi) := ((D-\sigma_0R_0)\Phi, P_{W_+}\mathcal{R}(\Phi)) ,
\end{align*}
is also an isomorphism for $0\le k<m$ and $\rho$ small enough.
Denote the inverse of $\mathcal{F}_{B}$ by $\mathcal{G}_B$.

Suppose now that $\mathcal{R}(\Phi)\in B$ and $D\Phi\in H^k_{\loc}(M,F)$.
We only need to show that $\Phi$ is $H^{k+1}$ near $\dM$.
Multiplying by a smooth cut-off function, 
we can assume that the support of $\Phi$ lies inside $\Zrq$
and does not intersect the right part of the boundary of $\Zrq$.
Then $\Phi\in H^1(\Zrq,E;B\oplus B_\rho)$ by \tref{domdmax}~\eqref{reg}
because $B\subset H^{1/2}(\dM,E)$.
By induction on $k$, we can also assume that $\Phi\in H^k(\Zrq,E)$.
Then
\[
\mathcal{F}_B(\Phi)
= ((D-\sigma_0R_0)\Phi, P_{W_+}\mathcal{R}(\Phi)) \in H^k(\Zrq,F) \oplus W_+
\]
and therefore
\[
\Phi = \mathcal{G}_B(\mathcal{F}_B(\Phi)) \in H^{k+1}(\Zrq,E;B\oplus B_\rho) .
\qedhere
\]
\end{proof}

\begin{cor}
Assume the Standard Setup~\ref{stase} and that $D$ and $D^*$ are complete.
Let $B\subset H^{1/2}(\dM,E)$ be an $\infty$-regular elliptic boundary condition for~$D$.

Then each $\Phi\in\dom(D_{\max})$ with $D_{\max}\Phi=0$ and $\mathcal R\Phi\in B$
is smooth up to the boundary.
\end{cor}

\begin{proof}
By standard elliptic regularity theory, $\Phi$ is smooth in the interior of $M$.
\tref{thm:regvc} shows that $\Phi\in H^{m}_{\loc}(M,E)$ for all $m$.
The Sobolev embedding theorem now implies that $\Phi$ is smooth up to the boundary.
\end{proof}


\subsection{Local and pseudo-local boundary conditions}
\label{suseplbc}

\begin{definition}\index{local boundary condition}\index{pseudo-local boundary condition}
We say that $B \subset H^{1/2}(\dM,E)$ is a \emph{local boundary condition}
if there is a subbundle $E'\subset E_{|\dM}$ such that
\[
B = H^{1/2}(\dM,E').
\]
More generally, we call $B$ \emph{pseudo-local} if there is
a classical pseudo-differential operator $P$ of order $0$ acting on sections of $E$
over $\dM$ inducing an orthogonal projection on $L^2(\dM,E)$ such that
\[
B=P(H^{1/2}(\dM,E)) .
\]
\end{definition}

\begin{thm}\label{pseudo}
Assume the Standard Setup~\ref{stase} and that $D$ and $D^*$ are complete.
Let $P$ be a classical pseudo-differential operator of order zero
acting on section of $E$ over $\partial M$.
Suppose that $P$ induces an orthogonal projection in $L^2(\dM,E)$.
Then the following are equivalent:
\begin{enumerate}[(i)]
\item\label{lbcell}
$B=P(H^{1/2}(\dM,E))$ is an elliptic boundary condition for~$D$.
\item\label{lbcfop} 
For some (and then all) $a\in\R$,
$P-Q_{[a,\infty)}:L^2(\dM,E)\to L^2(\dM,E)$ is a Fredholm operator. 
\item\label{lbcpdo} 
For some (and then all) $a\in\R$,
$P-Q_{[a,\infty)}:L^2(\dM,E)\to L^2(\dM,E)$
is an elliptic classical pseudo-differential operator of order zero.
\item\label{lbcloc}
For all $\xi\in T^*_x\dM\setminus\{0\}$, $x\in\dM$, 
the principal symbol $\sigma_P(\xi):E_x\to E_x$ restricts to an isomorphism
from the  sum of the eigenspaces to the negative eigenvalues of $i\sigma_{A}(\xi)$
onto its image $\sigma_P(\xi)(E_x)$.
\end{enumerate}
\end{thm}

\begin{remarks}
(a)
Variants of the equivalence of \eqref{lbcell} with \eqref{lbcfop} are contained in \cite[pp.~75--77]{Gi}, \cite[Thm.~5.6]{BL2}, and \cite[Thm.~1.95]{BBC}.
A special case of the equivalence of \eqref{lbcfop} with \eqref{lbcloc}
is contained in \cite[Thm.~7.2]{BL2}.

(b)
Not every elliptic boundary condition is pseudo-local, i.e.,
the orthogonal projection onto an elliptic boundary condition $B$
is not always given by a pseudo-differential operator.
For example, let $V_-=L^2_{(-\infty,0]}(A)$, $W_-=0$, and $g=0$.
Choose $\phi \in H^{1/2}_{(0,\infty)}(A)$ which is not smooth,
put $W_+=\C\cdot\phi$ and let $V_+$ be the orthogonal complement
of $W_+$ in $L^2_{(0,\infty)}(A)$.
Each eigensection for $A$ to a positive eigenvalue which occurs
in the spectral decomposition of $\phi$ is a smooth section
which is mapped by $P$ to a nontrivial multiple of $\phi$.
If $P$ were pseudo-differential it would map smooth sections to smooth sections.
\end{remarks}

\begin{proof}[Proof of \tref{pseudo}]
The equivalence of \eqref{lbcfop} and \eqref{lbcpdo} follows
from Theorems~19.5.1 and 19.5.2 in \cite{Ho}.

It is known that the spectral projection $Q:=Q_{[a,\infty)}$, $a\in\R$,
is a classical pseudo-differential operator of order zero
and that its principal symbol $i\sigma_Q(\xi)$, $\xi\in T^*\partial M\setminus\{0\}$,
is the orthogonal projection onto the sum of the eigenspaces
of positive eigenvalues of $i\sigma_{A}(\xi)$,
see \cite[p.~48]{APS} together with \cite{Se} or \cite[Prop.~14.2]{BW}.
Both, $i\sigma_P(\xi)$ and $i\sigma_Q(\xi)$, are orthogonal projections.
Hence $i\sigma_{P-Q}(\xi) = i\sigma_P(\xi)-i\sigma_Q(\xi)$ is an isomorphism
if and only if $i\sigma_P(\xi)$ induces an isomorphism between its image
and the orthogonal complement of the image of $i\sigma_Q(\xi)$.
This shows the equivalence of \eqref{lbcpdo} and \eqref{lbcloc}.

It remains to show the equivalence of \eqref{lbcell} with the other conditions.
We show that \eqref{lbcpdo} implies \eqref{lbcell}.
We check the first characterization in \tref{ellbou}.
First we observe that if $\phi_j\in B$ converge in $H^{1/2}(\dM,E)$ to an element $\phi$,
then $\phi_j=P(\phi_j) \to P(\phi)\in B$, hence $B$ is closed in $H^{1/2}(\dM,E)$.
Since $P-Q$ is elliptic of order zero, it has a parametrix $R$, that is,
$R$ is a  classical pseudo-differential operator of order zero
such that $R(P-Q) = \id + S$, where $S$ is a smoothing operator.
For $\phi\in B$ we have
\begin{align*}
\HhdM{\phi}
&\le
\HhdM{(\id+S)\phi} + \HhdM{S\phi}\\
&=
\HhdM{R(P-Q)\phi} + \HhdM{S\phi}\\
&\le
C_1\cdot\left(\HhdM{(P-Q)\phi} + \HmhdM{\phi}\right)\\
&=
C_1\cdot\left(\HhdM{(\id-Q)\phi} + \HmhdM{\phi}\right)\\
&\le
C_2 \cdot \HcdM{\phi}.
\end{align*}
Thus the $H^{1/2}$ and $\check H$--norms are equivalent on $B$,
hence $B$ is closed in $\check H(A)$.

Now let $\psi\in\sigma _0^*(\Bad)$.
Then $\psi\in\hat H(A)$, and we have for all $\phi\in H^{1/2}(A)$:
\[
0
=
(P\phi,\psi)
=
(\phi,P\psi) .
\]
Hence $P\psi=0$ and thus $(P-Q)(\psi)=-Q(\psi)\in H^{1/2}(A)$
since $\psi\in \hat H(A)$.
Therefore
\[
\psi 
=
(\id + S)\psi - S\psi
=
\underbrace{R(P-Q)\psi}_{\in H^{1/2}} 
- \underbrace{S \psi}_{\in \Cu}
\in
H^{1/2}(\dM,E).
\]
Thus $\sigma _0^*(\Bad)\subset H^{1/2}(\dM,E)$.
Hence $\Bad\subset H^{1/2}(\dM,F)$ and \eqref{lbcell} follows.

Finally, we show that \eqref{lbcell} implies \eqref{lbcfop}.
Let $L^2(\dM,E) = V_- \oplus W_- \oplus V_+ \oplus W_+$
and $B=P(H^{1/2}(\dM,E)) = W_+ \oplus \Gamma(g)^{1/2}$,
hence $P(L^2(\dM,E)) = W_+ \oplus \Gamma(g)$.
Since the sum of a Fredholm operator and a finite-rank-operator
is again a Fredholm operator and since $W_+$ and $W_-$ are finite-dimensional,
we can assume without loss of generality that $W_+=W_-=0$.
Let $a\in\R$ such that $V_- \subset L^2_{(-\infty,a]}(A)$
and $V_+ \subset L^2_{[-a,\infty)}(A)$.
Since $L^2_{[-a,a]}(A)$ is finite-dimensional,
we can furthermore assume that $a=0$, $V_-=L^2_{(-\infty,0)}(A)$,
and $V_+ = L^2_{[0,\infty)}(A)$.
With respect to the splitting $L^2(\dM,E) = V_- \oplus V_+$ we have
\[
Q = \begin{pmatrix}
    0 & 0 \cr 0 & \id
    \end{pmatrix}
\quad\mbox{ and, by \eqref{eq:ProjGamma}},\quad
P = \begin{pmatrix}
    \id & 0 \cr g & 0
    \end{pmatrix}
    \begin{pmatrix}
    \id &  -g^*\cr g & \id
    \end{pmatrix}^{-1} ,
\]
hence
\[
P-Q = \begin{pmatrix}
      \id & 0 \cr 0 & -\id
      \end{pmatrix}
      \begin{pmatrix}
      \id &  -g^*\cr g & \id
      \end{pmatrix}^{-1}
\]
is an isomorphism.
\end{proof}

\begin{cor}\label{lopashap}
Assume the Standard Setup~\ref{stase} and that $D$ and $D^*$ are complete.
Let $E' \subset E_{|\dM}$ be a subbundle
and let $P:E_{|\dM}\to E'$ be the fiberwise orthogonal projection.

If $(D,\id-P)$ is an elliptic boundary value problem in the classical sense
of Lopatinsky and Shapiro, then $B=H^{1/2}(\dM,E')$
is a local elliptic boundary value condition in the sense of \dref{defellbou}.
\end{cor}

\begin{proof}
We use a boundary defining function $t$ as in \lref{adapted}.
Let $(D,\id-P)$ be an elliptic boundary value problem in the classical sense
of Lopatinsky and Shapiro, see e.~g.\ \cite[Sec.~1.9]{Gi}.
This means that the rank of $E'$ is half of that of $E$ and that, for any $x\in\dM$,
any $\eta\in T^*_x\dM \setminus \{0\}$, and any $\phi\in (E'_x)^\perp$,
there is a unique solution $f:[0,\infty)\to E_x$ to the ordinary differential equation
\begin{equation}
\left(i\sigma_{A}(\eta) +\frac{d}{dt}\right)f(t)=0
\label{eq:LopShap}
\end{equation}
subject to the boundary conditions
$$
(\id-P)f(0)=\phi
\quad\mbox{ and }\quad
\lim_{t\to\infty}f(t)=0 .
$$
Recall from \dref{defbs} that $i\sigma_{A}(\eta)$ is Hermitian,
hence diagonalizable with real eigenvalues.
The solution to \eqref{eq:LopShap} is given by $f(t)= \exp(-it\sigma_{A}(\eta))\phi$.
The condition $\lim_{t\to\infty}f(t)=0$ is therefore equivalent to $\phi$
lying in the sum of the eigenspaces to the positive eigenvalues of $i\sigma_{A}(\eta)$.
This shows criterion \eqref{lbcloc} of \tref{pseudo}.
\end{proof}

As a direct consequence of \tref{pseudo} \eqref{lbcloc} we obtain

\begin{cor}\label{cor:involution}
Assume the Standard Setup~\ref{stase} and that $D$ and $D^*$ are complete.
Let $E_{|\dM}=E'\oplus E''$ be a decomposition
such that $\sigma_{A}(\xi)$ interchanges $E'$ and $E''$ for all $\xi\in T^*\dM$.

Then $B':=H^{1/2}(\dM,E')$ and $B'':=H^{1/2}(\dM,E'')$
are local elliptic boundary conditions for $D$.\qed
\end{cor}

This corollary applies, in particular,
if $A$ itself interchanges sections of $E'$ and $E''$.

Solutions to elliptic equations under pseudo-local elliptic boundary conditions
have optimal regularity properties.
Namely, we have

\begin{prop}\label{prop:localregular}
Every pseudo-local elliptic boundary condition is $\infty$-regular.
\end{prop}

\begin{proof}
Let $B$ be a pseudo-local elliptic boundary condition, and let $s\ge1/2$.
We show that $B$ is $(s+1/2)$-regular
by checking criterion~\eqref{Breg3} of \lref{lem:Breg}.
Let $\phi\in B$ with $Q_{(-\infty,0)}(\phi)\in H^s(\dM,E)$.
Since $B$ is pseudo-local, there is a classical pseudo-differential operator $P$
of order $0$ inducing an orthogonal projection on $L^2(\dM,E)$
such that $B=P(H^{1/2}(\dM,E))$
and such that $P-Q_{[0,\infty)}:L^2(\dM,E)\to L^2(\dM,E)$ is elliptic,
see criterion \eqref{lbcpdo} in \tref{pseudo}.
Now, since $P\phi=\phi$,
\[
(P-Q_{[0,\infty)})(\phi) = \phi - Q_{[0,\infty)}(\phi)
= Q_{(-\infty,0)}(\phi) \in H^s(\dM,E).
\]
Ellipticity of $P-Q_{[0,\infty)}$ implies that $\phi\in H^s(\dM,E)$.
\end{proof}


\subsection{Examples}
\label{susexas}

We start with examples of well-known local boundary conditions
of great geometric significance.
They are all of the form described in \cref{cor:involution}.

\begin{example}[Differential forms]
Let $M$ carry a Riemannian metric $g$ and let 
\[
E= \bigoplus_{j=0}^n \Lambda^jT^*M \otimes_{\R} \C
= \Lambda^*T^*M \otimes_{\R} \C
\] 
be the complexification of the sum of the form bundles over $M$.
The operator is given by $D=d+d^*$,
where $d$ denotes exterior differentiation and $d^*$ is its formal adjoint
with respect to the volume element of $g$.

Let $T$ be the interior unit normal vector field along the boundary $\dM$
and let $\tau$ be the associated unit conormal one-form.
Then $\sigma_D$ is symmetric with respect to $T$ along the boundary.
For each $x\in\dM$ and $j$ we have a canonical identification 
\[
\Lambda^jT^*_xM
= \Lambda^jT^*_x\dM \oplus \tau(x)\wedge\Lambda^{j-1}T^*_x\dM, 
\quad
\phi = (\phi)^{\tan} + \tau(x)\wedge (\phi)^\nor.
\]
The local boundary condition corresponding to the subbundle
$E':=\Lambda^*\dM\otimes_\R\C \subset E_{|\dM}$ is called
the {\em absolute boundary condition}, \index{absolute boundary condition}
\[
B_\abs 
= 
\{\phi\in H^{1/2}(\dM,E) \mid (\phi)^\nor =0 \} ,
\]
while $E'':=\tau\wedge\Lambda^*\dM\otimes_\R\C \subset E_{|\dM}$
yields the {\em relative boundary condition}\index{relative boundary condition},
\[
B_\rel
=
\{\phi\in H^{1/2}(\dM,E) \mid (\phi)^{\tan} =0 \} .
\]
Both boundary conditions are known to be elliptic in the classical sense,
see e.~g.\ \cite[Lemma~4.1.1]{Gi}.
Indeed, for $\xi\in T^*\dM$,
$\sigma_D(\xi)$ leaves the subbundles $E'$ and $E''$ invariant
while $\sigma_D(\tau)$ interchanges them.
Hence, by \eqref{eq:symbDT}, $\sigma_{A}(\xi)$ interchanges $E'$ and $E''$.
\end{example}

\begin{example}[Chirality conditions]\label{chirco}\index{chirality condition}
Let $M$ be a Riemannian spin manifold
and $D$ be the Dirac operator acting on spinor fields, i.e.,
sections of the spinor bundle $E=F=\Sigma M$.
We say that a morphism $G:\Sigma M|_{\dM} \to \Sigma M|_{\dM}$
is a \emph{boundary chirality operator}\index{boundary chirality operator} 
if $G$ is an orthogonal involution of $E$ such that
\begin{align*}
\sigma_D(\tau)\circ G &= G \circ \sigma_D(\tau) \mbox{ and }  \\
\sigma_D(\xi)\circ G &= - G \circ \sigma_D(\xi),\mbox{ for all }\xi\in T^*\dM .
\end{align*}
Now let $G$ be a chirality operator of $E$ as above
and $E'$ and $E''$ be the subbundles of $E$
given by the $\pm 1$-eigenspaces of $G$.
The displayed properties of $G$ show that
$\sigma_A(\xi) = \sigma_D(\tau)^{-1}\circ\sigma_D(\xi)$ interchanges the two subbundles.
\cref{cor:involution} then implies that both subbundles
give rise to local elliptic boundary conditions for $D$.

For instance, Clifford multiplication with $i$ times the exterior unit normal field
yields a chirality operator along $\dM$.
The resulting boundary condition is sometimes called
the \emph{MIT-bag condition}.\index{MIT-bag condition}

Chirality conditions have been used to show the positivity of the ADM mass
on asymptotically flat manifolds \cite{GHHP,He} (using an idea of Witten \cite{Wi}).
Eigenvalue estimates for the Dirac operator under selfadjoint chirality boundary conditions
have been derived in \cite{HMZ,HMR}.
\end{example}

We now discuss examples of pseudo-local boundary conditions.

\begin{example}[Generalized Atiyah-Patodi-Singer boundary conditions]\index{generalized Atiyah-Patodi-Singer boundary condition}\label{gaps}
In Example~\ref{aps} we have seen that the Atiyah-Patodi-Singer boundary conditions
are elliptic in the sense of Definition~\ref{defellbou}.
More generally, fix $a\in\R$
and put $V_-=L^2_{(-\infty,a)}(A)$, $V_+=L^2_{[a,\infty)}(A)$, 
$W_-=W_+=0$, and $g=0$.
Then we have
\[
  B = B(a) := H_{(-\infty,a)}^{1/2}(A) .
\]
These elliptic boundary conditions are known as {\em
generalized Atiyah-Patodi-Singer boundary conditions}.
As remarked in the proof of \tref{pseudo} these boundary conditions
are pseudo-local by \cite[p.~48]{APS} together with \cite{Se}
or by \cite[Prop.~14.2]{BW}.
\end{example}

\begin{example}[Transmission conditions]\label{ex:TransCond}\index{transmission condition}
Let $(M,\mu)$ be a measured manifold.
For the sake of simplicity, assume that the boundary of $M$ is empty,
even though this is not really necessary.
Let $N\subset M$ be a compact hypersurface with trivial normal bundle.
Cut $M$ along $N$ to obtain a measured manifold $M'$ with boundary.
The boundary $\dM'$ consists of two copies $N_1$ and $N_2$ of $N$.
We may write $M' = (M\setminus N)\sqcup N_1 \sqcup N_2$.

\begin{center}
\begin{pspicture}(0,-7.171081)(12.721025,2.1710808)
\psbezier[linewidth=0.04](12.501025,2.1510808)(11.741025,1.8310809)(6.443285,2.1186054)(5.4610248,2.111081)(4.478765,2.1035564)(0.0,0.99085987)(0.021024996,-0.008919082)(0.042049993,-1.008698)(5.102699,-2.1510808)(6.161025,-2.1089191)(7.219351,-2.0667572)(11.741025,-0.46891907)(12.701025,-0.68891907)
\psbezier[linewidth=0.04](12.381025,1.6910809)(12.041025,1.6110809)(7.0244503,1.4894793)(7.081025,0.4910809)(7.1376,-0.5073174)(11.701025,-0.028919082)(12.541025,-0.108919084)
\psbezier[linewidth=0.04](3.101025,0.4910809)(3.101025,-0.30891907)(5.621025,-0.3889191)(5.621025,0.41108093)
\psbezier[linewidth=0.04](3.301025,0.16832179)(3.661025,0.7910809)(4.9610248,0.7310809)(5.2583065,-0.028919082)
\psbezier[linewidth=0.06](4.021025,-0.108919084)(3.301025,-0.008919082)(3.341025,-1.8089191)(3.811025,-1.6756418)
\psbezier[linewidth=0.06,linestyle=dashed,dash=0.16cm 0.16cm](3.925706,-0.08891908)(4.781025,-0.39566067)(4.2703867,-1.9089191)(3.781025,-1.6839752)
\rput(6.302431,1.2360809){$M$}
\rput(3.1524312,-0.8439191){$N$}

\rput(0,-5){
\psbezier[linewidth=0.04](12.501025,2.1510808)(11.741025,1.8310809)(6.443285,2.1186054)(5.4610248,2.111081)(4.478765,2.1035564)(0.0,0.99085987)(0.021024996,-0.008919082)(0.042049993,-1.008698)(5.102699,-2.1510808)(6.161025,-2.1089191)(7.219351,-2.0667572)(11.741025,-0.46891907)(12.701025,-0.68891907)
\psbezier[linewidth=0.04](12.381025,1.6910809)(12.041025,1.6110809)(7.0244503,1.4894793)(7.081025,0.4910809)(7.1376,-0.5073174)(11.701025,-0.028919082)(12.541025,-0.108919084)
\psbezier[linewidth=0.04](3.101025,0.4910809)(3.101025,-0.30891907)(5.621025,-0.3889191)(5.621025,0.41108093)
\psbezier[linewidth=0.04](3.301025,0.16832179)(3.661025,0.7910809)(4.9610248,0.7310809)(5.2583065,-0.028919082)

\psdots[dotsize=0.2,linecolor=white](3.65,-1.7)(3.85,-0.1)

\psbezier[linewidth=0.06](4.021025,-0.108919084)(3.301025,-0.008919082)(3.341025,-1.8089191)(3.811025,-1.6756418)
\psbezier[linewidth=0.06,linestyle=dashed,dash=0.16cm 0.16cm](3.925706,-0.08891908)(4.781025,-0.39566067)(4.2703867,-1.9089191)(3.781025,-1.6839752)
\psbezier[linewidth=0.06](3.761025,-0.068919085)(3.041025,0.031080918)(3.081025,-1.7689191)(3.551025,-1.6356418)
\psbezier[linewidth=0.06,linestyle=dotted,dash=0.16cm 0.16cm](3.685706,-0.04891908)(4.541025,-0.35566065)(4.030387,-1.8689191)(3.541025,-1.6439753)

\rput(6.302431,1.2360809){$M'$}
\rput(3.9,-2.1){$N_2$}
\rput(3.4,-2.0){$N_1$}
}
\end{pspicture} 

\label{fig:transmission}
{\em Fig.~6}
\end{center}

Let $E,F\to M$ be Hermitian vector bundles
and $D:\Cu(M,E) \to \Cu(M,F)$  a linear elliptic differential operator of first order.
We get induced bundles $E'\to M'$ and $F'\to M'$
and an elliptic operator $D':\Cu(M',E') \to \Cu(M',F')$.
For $\Phi\in H^1_{\loc}(M,E)$ we get $\Phi'\in H^1_{\loc}(M',E')$
such that $\Phi'|_{N_1}=\Phi'|_{N_2}$.
We use this as a boundary condition for $D'$ on $M'$.
We set
\[
B := \left\{(\phi,\phi)\in H^{1/2}(N_1,E)\oplus H^{1/2}(N_2,E) \mid \phi\in H^{1/2}(N,E)\right\}.
\]
Here we used the canonical identification
\[
  H^{1/2}(N_1,E)=H^{1/2}(N_2,E)=H^{1/2}(N,E) .
\]
Now we assume that $D'$ is boundary symmetric
with respect to some interior vector field $T$ along $N=N_1$.
This corresponds to a condition on the principal symbol of $D$ along $N$.
It is satisfied e.g.\ if $D$ is of Dirac type.
Let $A=A_0 \oplus -A_0$ be an adapted boundary operator for $D'$.
Here $A_0$ is a selfadjoint elliptic operator on $\Cu(N,E)=\Cu(N_1,E')$
and similarly $-A_0$ on $\Cu(N,E)=\Cu(N_2,E')$.
The sign is due to the opposite relative orientations of $N_1$ and $N_2$ in $M'$.

To see that $B$ is an elliptic boundary condition, put 
\begin{align*}
V_+ &:= L^2_{(0,\infty)}(A_0\oplus -A_0) 
      = L^2_{(0,\infty)}(A_0) \oplus  L^2_{(-\infty,0)}(A_0) , \\
V_- &:= L^2_{(-\infty,0)}(A_0\oplus -A_0) 
      = L^2_{(-\infty,0)}(A_0) \oplus L^2_{(0,\infty)}(A_0)  , \\
W_+ &:= \{(\phi,\phi)\in\ker(A_0)\oplus\ker(A_0)\}  , \\
W_- &:= \{(\phi,-\phi)\in\ker(A_0)\oplus\ker(A_0)\} ,
\end{align*}
and
\[
g:V_-^{1/2} \to V_+^{1/2}, \quad g=\begin{pmatrix}
                                   0 & \id \cr 
                                   \id & 0
                                   \end{pmatrix} .
\]
With these choices $B$ is of the form required in \dref{defellbou}.
We call these boundary conditions {\em transmission conditions}.
Clearly, they are $\infty$-regular.

If $M$ has a nonempty boundary and $N$ is disjoint from $\dM$,
let us assume that we are given an elliptic boundary condition for $\dM$.
Then the same discussion applies if one keeps the boundary condition on $\dM$
and adapts $B$ on $\dM' = \dM \sqcup N_1 \sqcup N_2$ accordingly.
\end{example}


\section{Index theory}
\label{inth}


Throughout this section, assume the Standard Setup~\ref{stase}
and that $D$ and $D^*$ are complete.
In addition to completeness,
we introduce a second property of $D$ at infinity:
coercivity, see \dref{def:coerc} below.
This property will be crucial for the Fredholm property of $D$
so that we can speak of its index.

\begin{remark}\label{nomax}
Recall that we have $D_B=D_{B,\max}$ with domain \index{1HgDMEB@$H^1_D(M,E;B)$}
\[
  H^1_D(M,E;B) := \{\Phi\in H^1_D(M,E) \mid \mathcal R\Phi\in B\}
\]
if $B$ is an elliptic boundary condition for $D$, compare \dref{ellreg} and \lref{h1ngn}.
In particular, if $\dM$ is empty, then $D=D_{\max}$ with domain $H^1_D(M,E)$.
This explains the difference in notation between the introduction and this section.
\end{remark}


\subsection{The Fredholm property}
\label{susefrepro}

\begin{definition}\label{def:coerc}
We say that $D$ is {\em coercive at infinity}\index{coercive at infinity}
if there is a compact subset $K\subset M$ and a constant $C$ such that
\begin{equation}\label{coerc}
  \| \Phi \|_{L^2(M)}
  \le C \| D\Phi \|_{L^2(M)} ,
\end{equation}
for all smooth sections $\Phi$ of $E$
with compact support in $M\setminus K$.
\end{definition}

\begin{examples}\label{ex:coercive}
(a)
By definition, if $M$ is compact, then $D$ is coercive at infinity.

(b)
Coercivity at infinity of a Dirac type operator $D:\Cu(M,E) \to \Cu(M,F)$
can be shown using a \emph{Weitzenb\"ock formula}.\index{Weitzenb\"ock formula}
This is a formula of the form
$$
D^*D = \nabla^*\nabla + \mathcal{K}
$$
where $\nabla$ is a connection on $E$ and $\mathcal{K}$
is a symmetric endomorphism field on $E$.
Now if, outside a compact subset $K\subset M$, all eigenvalues of $\mathcal{K}$
are bounded below by a constant $c>0$, then we have, for all $\Phi\in\Cucc(M,E)$
with support disjoint from $K$,
$$
\LzM{D\Phi}^2 = \LzM{\nabla\Phi}^2 + (\mathcal{K}\Phi,\Phi)_{L^2(M)} 
\geq c\LzM{\Phi}^2 .
$$
Hence $D$ is coercive at infinity.
\end{examples}

\begin{lemma}\label{coerc2}
The operator $D$ is coercive at infinity if and only if there is no
sequence $(\Phi_n)$ in $\Cucc(M,E)$ such that
\[
  \| \Phi_n \|_{L^2(M)} = 1 
  \quad\text{and}\quad 
  \lim_{n\to\infty}\| D\Phi_n \|_{L^2(M)} = 0 
\]
and such that, for any compact subset $K\subset M$, we have 
\[
\supp\Phi_n\cap K = \emptyset
\] 
for all sufficiently large $n$.
\end{lemma}

\begin{proof}
Choose an exhaustion of $M$ by compact sets $K_1 \subset K_2 \subset \cdots \subset M$. 
If $D$ is not coercive at infinity,
then for each $n$ there exists $\Phi_n$
with $\supp\Phi_n \cap K_n=\emptyset$ and $\LzM{\Phi_n} > n\LzM{D\Phi_n}$.
Normalizing $\Phi_n$, we obtain a sequence as in \lref{coerc2}.

Conversely, a sequence as in \lref{coerc2} directly violates the condition in \dref{def:coerc}.
\end{proof}

\begin{thm}\label{thmfred}
Assume the Standard Setup~\ref{stase} and that $D$ and $D^*$ are complete.
Let $B\subset H^{1/2}(\dM,E)$ be an elliptic boundary condition for $D$.

Then $D$ is coercive at infinity if and only if
\begin{align*}
  D_B&: H^1_D(M,E;B) \to L^2(M,F)
\end{align*}
has finite-dimensional kernel and closed image.
In this case
\[
  \ind D_B = \dim\ker D_B - \dim\ker (D^*)_{\Bad} \in \Z \cup \{-\infty\}.
\]
\end{thm}

\begin{proof}
Suppose first that $D$ is coercive at infinity.
We want to apply \pref{HormPeet}.
Let $(\Phi_n)$ be a bounded sequence in $H^1_D(M,E;B)$
such that $D\Phi_n\to\Psi\in L^2(M,F)$.
We have to show that $(\Phi_n)$ has a convergent subsequence in $H^1_D(M,E)$.

Passing to a subsequence if necessary, we can assume,
by the Rellich embedding theorem, that there is a section $\Phi\in L^2_{\loc}(M,E)$
such that  $\Phi_n\to\Phi$ in $L^2_{\loc}(M,E)$.
Choose a compact subset $K\subset M$ and a constant $C$ as in \eqref{coerc}.
Let $\chi\in\Cuc(M,\R)$ be a cut-off function which is equal to $1$ on $K$
and $0\leq \chi \le 1$ everywhere.
Set $K':=\supp(\chi)$.
By \eqref{coerc},
\begin{align*}
\LzM{\Phi_n-\Phi_m} 
&\le
\LzM{\chi(\Phi_n-\Phi_m)}+\LzM{(1-\chi)(\Phi_n-\Phi_m))} \\
&\le
\Lz{\Phi_n-\Phi_m}{K'} + C\LzM{D((1-\chi)(\Phi_n-\Phi_m)))} \\
&\le
\Lz{\Phi_n-\Phi_m}{K'} + C\LzM{-\sigma_D(d\chi)(\Phi_n-\Phi_m)}\\
&\phantom{==}+C\LzM{(1-\chi)(D\Phi_n-D\Phi_m)} \\
&\le
C'\Lz{\Phi_n-\Phi_m}{K'} + C\Lz{D\Phi_n-D\Phi_m}{M} \to 0.
\end{align*}
It follows that $(\Phi_n)$ is a Cauchy sequence in $L^2(M,E)$,
hence $\Phi\in L^2(M,E)$ and $\Phi_n\to\Phi$ in $L^2(M,E)$.

Since $B$ is an elliptic boundary condition, 
the $H^1_D$-norm and the graph norm of $D$ are equivalent on $H^1_D(M,E;B)$.
Now $(\Phi_n)$ and $(D\Phi_n)$ converge in $L^2(M,E)$ and $L^2(M,F)$, respectively.
Hence $(\Phi_n)$ converges in the graph norm of $D$, and therefore in $H^1_D(M,E)$.
By the implication $\eqref{HP3}\Rightarrow\eqref{HP1}$ of \pref{HormPeet}, 
$D_B$ has finite-dimensional kernel and closed image. 

Suppose now that $D$ is not coercive at infinity.
Let $(\Phi_n)$ be a sequence as in \lref{coerc2}.
By the assumption on the support of $\Phi_n$,
$\Phi_n\rightharpoonup 0$ in $L^2(M,E)$.
Since, on the other hand, $\LzM{\Phi_n}=1$ for all $n$,
no subsequence converges in $L^2(M,E)$.
In particular, no subsequence converges in $H^1_D(M,E)$.
Hence $D_B$ does not satisfy criterion \eqref{HP3} in \pref{HormPeet},
and hence $D_B$ has infinite-dimensional kernel or nonclosed image.
\end{proof}

\begin{cor}\label{corfred1}
Assume the Standard Setup~\ref{stase}
and that $D$ and $D^*$ are complete and coercive at infinity.
Let $B\subset H^{1/2}(\dM,E)$ be an elliptic boundary condition for $D$.
Then
\[
  D_B : H^1_D(M,E;B) \to L^2(M,F)
\]
is a Fredholm operator and
\[
  \hspace{2.3cm}
  \ind D_B = \dim\ker D_B - \dim\ker (D^*)_{\Bad} \in \Z .
  \hspace{2.3cm}\qed
\]
\end{cor}

The following corollary is immediate from \pref{propFredholm}.

\begin{cor}\label{corfred2}
Assume the Standard Setup~\ref{stase}
and that $D$ and $D^*$ are complete and coercive at infinity.
Let $B \subset H^{1/2}(\dM,E)$ be an elliptic boundary condition for $D$.
Let $\check C$ and $C$ be closed complements of $B$
in $\check H(A)$ and in $H^{1/2}(\dM,E)$, respectively.
Let $\check P:\check H(A)\to\check H(A)$ and $P:H^{1/2}(\dM,E)\to H^{1/2}(\dM,E)$
be the projections with kernel $B$ and image $\check C$ and $C$, respectively.
Then
\begin{align*}
\check L: \dom(D_{\max}) \to L^2(M,F)\oplus\check C,
  \quad &\check L\Phi = (D_{\max}\Phi,\check P\mathcal R\Phi),
\intertext{and}
   L: H^1_D(M,E) \to L^2(M,F)\oplus C,
  \quad &L\Phi = (D\Phi,P\mathcal R\Phi), 
\end{align*}
are Fredholm operators with the same index as $D_{B,\max}=D_B$.
\qed
\end{cor}

\begin{cor}\label{cor:AgraDy}
Assume the Standard Setup~\ref{stase}
and that $D$ and $D^*$ are complete and coercive at infinity.
Let $B_1 \subset B_2 \subset H^{1/2}(\dM,E)$ be elliptic boundary conditions for $D$.

Then $\dim(B_2/B_1)$ is finite and
\[
\ind(D_{B_2}) = \ind(D_{B_1}) + \dim(B_2/B_1).
\]
\end{cor}

\begin{proof}
Choose closed complements $C_1$ and $C_2$ in $H^{1/2}(\dM,E)$ of $B_1$ and $B_2$, 
respectively, such that $C_2 \subset C_1$.
Then we have the following commutative diagram:
\[
\xymatrix{
 && L^2(M,E) \oplus C_1  \ar@{->>}[dd]^{\id\oplus \pi}\\
H^1_D(M,E) \ar[rru]^{(D,Q_1)} \ar[drr]_{(D,Q_2)} &&\\
&& L^2(M,E) \oplus C_2
}
\]
where $Q_j$ is $\mathcal R$ composed with the projection onto $C_j$
as in \cref{corfred2} and $\pi:C_1 \rightarrow C_2$ is the projection.
Since $(D,Q_1)$ and $(D,Q_2)$ are Fredholm operators, $\id\oplus \pi$ is one too.
In particular, 
\[
\dim(B_2/B_1) = \dim(C_1/C_2) =  \ind(\pi) =  \ind(\id\oplus \pi)
\]
is finite.
Moreover,
\begin{align*}
\ind(D_{B_1}) 
&=
\ind(D,Q_1) \\
&=
\ind(D,Q_2) - \ind(\id\oplus \pi) \\
&=
\ind(D_{B_2}) - \dim(B_2/B_1) .
\qedhere
\end{align*}
\end{proof}

\begin{example}\label{exaAgDy}
Assume the Standard Setup~\ref{stase}
and that $D$ and $D^*$ are complete and coercive at infinity.
Then we have, for generalized Atiyah-Patodi-Singer boundary conditions
$B(a)$ and $B(b)$ with $a<b$,
\[
  \ind D_{B(b)} = \ind D_{B(a)} +  \dim L^2_{[a,b)}(A) .
\]
\end{example}


\subsection{Deformations of boundary conditions}
\label{deformbc}

\begin{definition}\label{def:contfam}\index{continuous family of boundary conditions}
A family of boundary conditions $B_s\subset H^{1/2}(\dM,E)$, $0\le s\le 1$,
is called a \emph{continuous family of boundary conditions} if there exist isomorphisms
\[
k_s:B_0\to B_s ,
\]
where $k_0=\id$ and where the map
$[0,1]\to \mathcal L(B_0,H^{1/2}(\dM,E))$, $s\mapsto k_s$, is continuous.
Here $\mathcal L(B_0,H^{1/2}(\dM,E))$ denotes the Banach space of bounded
operators from $B_0$ to $H^{1/2}(\dM,E)$ equipped with the operator norm.
\end{definition}

The following auxiliary isomorphisms allow us to consider families of operators
between fixed spaces, i.e., spaces independent of the deformation parameter.
This will be useful since it will allow us to apply the standard fact that the index
of a continuous family of Fredholm operators has constant index.

\begin{lemma}\label{defolem}
The map
\begin{align*}
  J_E: H^1_D(M,E;0) &\oplus H^{1/2}(\dM,E) \to H^1_D(M,E) ,  \\
  & J_E(\Phi,\phi) = \Phi + \mathcal E\phi ,
\end{align*}
is an isomorphism.
If $B\subset H^{1/2}(\dM,E)$ is an elliptic boundary condition, then the restriction
\[
  J_{E,B}: H^1_D(M,E;0) \oplus B \to H^1_D(M,E;B)
\]
is again an isomorphism.
\end{lemma}

\begin{proof}
For $\phi\in H^{1/2}(\dM,E)$,
we have $\mathcal E \phi\in \dom(D_{\max})$ by \lref{extfin2}.
By \tref{domdmax}~\eqref{reg}, $\mathcal E\phi\in H^1_D(M,E)$.
Thus $J_E$ maps indeed to $H^1_D(M,E)$.
The inverse map of $J_E$ is given by
\begin{align*}
K_E :  H^1_D(M,&E) \to  H^1_D(M,E;0) \oplus H^{1/2}(\dM,E), \\
&K_E(\Psi) = (\Psi-\mathcal{E}\mathcal{R}\Psi,\mathcal{R}\Psi).
\end{align*}
%
Clearly, $J_E(\Phi,\phi)\in H^1_D(M,E;B)$ if and only if 
$$
\phi=\mathcal{R}(J_E(\Phi,\phi)) \in B .
$$
Hence $J_E$ restricts to an isomorphism
between $H^1_D(M,E;0) \oplus B$ and $H^1_D(M,E;B)$.
\end{proof}

\begin{thm}\label{thmdefo2}
Assume the Standard Setup~\ref{stase} and
that $D$ and $D^*$ are complete and coercive at infinity.
Let $B_s$ be a continuous family of elliptic boundary conditions for $D$.

Then, for all $0\le s\le 1$,
\[
  \ind D_{B_s} = \ind D_{B_0} .
\]
\end{thm}

\begin{proof}
Let the $k_s:B_0 \to B_s$ be as in \dref{def:contfam}.
Consider the commutative diagram of isomophisms:
\[
\xymatrix{
H^1_D(M,E;0)  \oplus B_0 \ar[rr]^{J_{E,B_0}} \ar[d]_{\id\oplus k_s}
&& H^1_D(M,E;B_0) \ar[d]_{K_s}\\
H^1_D(M,E;0) \oplus B_s \ar[rr]^{J_{E,B_s}} && H^1_D(M,E;B_s) ,
} 
\]
where $K_s(\Psi)=\Psi+\mathcal{E}\big(k_s(\mathcal{R}(\Psi))-\mathcal{R}(\Psi)\big)$.
Then the composition
\[
  H^1_D(M,E;B_0) 
  \xrightarrow{K_s} H^1_D(M,E;B_s)
  \xrightarrow{D_{B_s}} L^2(M,F)
\]
is a continuous family of operators.
Here continuity refers to the operator norm in $\mathcal{L}(H^1_D(M,E;B_0),L^2(M,F))$.
Thus the index of $D_{B_s}\circ K_s$ is constant.
Since $K_s$ is an isomorphism, $\ind D_{B_s}=\ind (D_{B_s}\circ K_s)$.
\end{proof}

\begin{example}\label{defob}
Writing an elliptic boundary condition $B_1$ as in \dref{defellbou},
we let $B_0:=W_+\oplus V_-^{1/2}$ and obtain a continuous deformation
by keeping $V_\pm$ and $W_\pm$ constant
and replacing the map $g:V_-\to V_+$ by $sg$, $0\le s\le 1$.
In other words,
\[
  k_s: B_0 \to B_s , \quad k_s(\phi+\psi) = \phi + \psi + sg\psi ,
\]
where $\phi\in W_+$ and $\psi\in V_-$.
This allows us to reduce index computations to the case $g=0$.
\end{example}


\subsection{Fredholm pairs}
\label{susefp}

Recall the generalized Atiyah-Patodi-Singer boundary conditions $B(a)$, $a\in\R$,
from Example~\ref{exabc} (a) or Example~\ref{gaps}.
The next result illustrates the central role of this kind of boundary condition.

\begin{thm}\label{thmad}
Assume the Standard Setup~\ref{stase}
and that $D$ and $D^*$ are complete and coercive at infinity.

Let $B \subset H^{1/2}(\dM,E)$ be an elliptic boundary condition for $D$.
Let $V_\pm$ and $W_\pm$ be for $B$ as in \tref{ellbou}~\eqref{ellboua}.
In other words, $L^2_{(-\infty,a)}(A) = V_- \oplus W_-$
and $B=W_+ \oplus \Gamma(g)^{1/2}$, where $g:V_-\to V_+$.
Then we have
\[
  \ind D_B = \ind D_{B(a)} + \dim W_+ - \dim W_- .
\]
\end{thm}

\begin{proof}
As explained in \eref{defob}, we can assume without loss of generality that $g=0$,
i.e., $B=W_+ \oplus V_-^{1/2}$.
Consider one further boundary condition, 
\[
B' := W_- \oplus W_+ \oplus V_-^{1/2}
= H^{1/2}_{(-\infty,a)}(A) \oplus W_+ = B(a) \oplus W_+ .
\]
Applying \cref{cor:AgraDy} twice we conclude
\begin{align*}
\ind(D_B) 
&= 
\ind(D_{B'}) - \dim W_- \\
&=
\ind(D_{B(a)}) + \dim W_+ - \dim W_- . \qedhere
\end{align*}
\end{proof}

We recall the notion of Fredholm pairs of subspaces.
Let $H$ be a Hilbert space and let $X, Y \subset H$ be closed subspaces.
Then $(X,Y)$ is called a \emph{Fredholm pair}\index{Fredholm pair}
if $X\cap Y$ is finite-dimensional and $X+Y$ has finite codimension.
In particular, $X+Y$ is closed.\symbolfootnote[1]{Note that, in general, the sum of two closed subspaces in a Hilbert space need not be closed.}
Intuitively, constituting a Fredholm pair means that $X$ and $Y$
are complements in $H$ up to finite-dimensional errors.
The number
\begin{equation}
\ind(X,Y) := \dim(X\cap Y) - \dim(H/(X+Y)) \in\Z
\label{indpair}
\end{equation}
is called the \emph{index}\index{index of Fredholm pair} of the pair $(X,Y)$.
If $(X,Y)$ is a Fredholm pair, then the orthogonal complements
also form a Fredholm pair $(X^\perp,Y^\perp)$ and one has
\begin{equation}
\ind(X^\perp,Y^\perp) = -\ind(X,Y) .
\label{indopair}
\end{equation}
See \cite[Ch.~IV, \S~4]{Ka} for a detailed discussion.
Under the assumptions of \tref{thmad}, the pair $(L^2_{[a,\infty)},\bar B)$,
where $\bar B$ denotes the $L^2$-closure of $B$,
is a Fredholm pair and the index formula says that 
\begin{equation}
\ind D_B - \ind D_{B(a)}
= \ind (L^2_{[a,\infty)},\bar B)
= - \ind(L^2_{(-\infty,a)},\bar B^\perp)
 \label{indind}
\end{equation}

\begin{thm}\label{thm:FredholmPairs}
Assume the Standard Setup~\ref{stase} and
that $D$ and $D^*$ are complete and coercive at infinity.

Let $B_1,B_2 \subset H^{1/2}(\dM,E)$ be elliptic boundary conditions for $D$.
Let $g_1$ and $g_2$ be maps for $B_1$ and $B_2$ respectively,
as in \tref{ellbou}~\eqref{ellboua} for the same number $a\in\R$.

If the $L^2$-operator norms satisfy
\begin{equation}
\|g_1\|\cdot\|g_2\| < 1,
\label{eq:FredVor}
\end{equation}
then the $L^2$-closures $(\bar B_1,\bar B_2^\perp)$
form a Fredholm pair in $L^2(M,E)$ and
\begin{equation}
\ind(D_{B_1}) - \ind(D_{B_2}) = \ind(\bar B_1,\bar B_2^\perp) .
\label{eq:IndDiff}
\end{equation}
\end{thm}

\begin{proof}
In the notation of \tref{defellbou} we have
$$
\bar B_1 = W_{1,+} \oplus \Gamma(g_1)
$$
where $g_1:V_{1,-}\to V_{1,+}$ and, by \lref{lem:Bperp},
$$
\bar B_2^\perp = W_{2,-} \oplus \Gamma(-g_2^*).
$$
A general element of $\bar B_1 \cap \bar B_2^\perp$ is therefore of the form
\begin{equation}
w_{1,+} + v_{1,-} + g_1v_{1,-} = w_{2,-} + v_{2,+} - g_2^*v_{2,+}
\label{eq:FredPair1}
\end{equation}
where $w_{1,+}\in W_{1,+}$, $w_{2,-}\in W_{2,-}$, $v_{1,-}\in V_{1,-}$,
and $v_{2,+}\in V_{2,+}$.
Projecting to $L^2_{(-\infty,a)}(A)$ and to $L^2_{[a,\infty)}(A)$,
we see that \eqref{eq:FredPair1} is equivalent to the two equations
\begin{align}
v_{1,-} &= w_{2,-}  - g_2^*v_{2,+} ,\label{eq:FredPair2}\\
v_{2,+} &= w_{1,+}  + g_1v_{1,-},\label{eq:FredPair3}.
\end{align}
We claim that the linear map $\bar B_1 \cap \bar B_2^\perp \to  W_{1,+} \oplus W_{2,-}$
mapping the element in \eqref{eq:FredPair1} to $(w_{1,+},w_{2,-})$ is injective.
Namely, if $w_{1,+}=w_{2,-}=0$, then \eqref{eq:FredPair2} and \eqref{eq:FredPair3} say
$$
v_{2,+} = g_1v_{1,-} = - g_1g_2^*v_{2,+} ,
$$ 
hence
$$
(\id + g_1g_2^*)v_{2,+} = 0.
$$
Here we have extended $g_1$ to $L^2_{(-\infty,a)}(A)$ by setting it $=0$ on $W_{1,-}$
and similarly for $g_2^*$.
Since 
\begin{equation}
\|g_1g_2^*\| \leq \|g_1\|\cdot\|g_2^*\| = \|g_1\|\cdot\|g_2\| < 1,
\label{eq:FredPair4}
\end{equation}
we see that $\id+g_1g_2^*$ is an isomorphism on $L^2_{[a,\infty)}(A)$.
Hence $v_{2,+}=0$ and, by \eqref{eq:FredPair2}, also $ v_{1,-} = 0$.
This implies 
$$
\dim(\bar B_1 \cap \bar B_2^\perp) \le \dim(W_{1,+} \oplus W_{2,-}) < \infty.
$$
Next we show that $\bar B_1 + \bar B_2^\perp$ is closed in $L^2(M,E)$.
Since $W_{1,+}$ and $W_{2,-}$ are finite-dimensional,
it suffices to show that $\Gamma(g_1) + \Gamma(-g_2^*)$ is closed.
Let $u_k + g_1u_k + v_k - g_2^*v_k$ converge in $L^2(M,E)$ as $k\to\infty$
with $u_k \in V_{1,-}$ and $v_k \in V_{2,+}$.
Projecting onto $L^2_{(-\infty,a)}(A)$ and onto $L^2_{[a,\infty)}(A)$ we see that 
\begin{align*}
u_k - g_2^*v_k &\to x_- \in L^2_{(-\infty,a)}(A) \mbox{ and}\\
v_k + g_1u_k   &\to x_+ \in L^2_{[a,\infty)}(A) .
\end{align*}
From
$$
g_1(u_k - g_2^*v_k) = v_k + g_1u_k  - (\id + g_1g_2^*)v_k
$$
we have
$$
v_k = (\id + g_1g_2^*)^{-1}(v_k + g_1u_k  - g_1(u_k - g_2^*v_k)) .
$$
This shows that $(v_k)$ converges with 
$$
\lim_{k\to\infty} v_k = (\id + g_1g_2^*)^{-1}(x_+  - g_1x_-).
$$
Then $(u_k)$ also converges and $u_k + g_1u_k + v_k - g_2^*v_k$
converges to an element in $\Gamma(g_1) + \Gamma(-g_2^*)$.

Now that we know that $\bar B_1 + \bar B_2^\perp$ is closed, we compute
$$
\frac{L^2(M,E)}{\bar B_1 + \bar B_2^\perp} 
\cong 
(\bar B_1 + \bar B_2^\perp)^\perp 
=
\bar B_1^\perp \cap \bar B_2 
\hookrightarrow
W_{2,+} \oplus W_{1,-}.
$$
Hence $\dim(L^2(M,E)/(\bar B_1 + \bar B_2^\perp)) < \infty$
and $(\bar B_1,\bar B_2^\perp)$ is a Fredholm pair.

It remains to prove \eqref{eq:IndDiff}.
If we replace $g_1$ and $g_2$ by $sg_1$ and $sg_2$, respectively, with $s\in[0,1]$,
then \eqref{eq:FredVor} clearly remains valid,
so that the corresponding boundary conditions form Fredholm pairs for all~$s\in[0,1]$.
Since the index of $D$ subject to these boundary conditions remains constant
by \tref{thmdefo2} and the index of the Fredholm pairs by \cite[Thm.~4.30, p.~229]{Ka},
we may without loss of generality assume that $g_1=g_2=0$.

In this case \eqref{eq:FredPair2} and \eqref{eq:FredPair3} tell us that 
\begin{equation}
\bar B_1 \cap \bar B_2^\perp = (W_{1,+}\cap V_{2,+}) \oplus (W_{2,-}\cap V_{1,-})
\label{eq:nerv1} 
\end{equation}
and similarly
\begin{equation}
\bar B_2 \cap \bar B_1^\perp = (W_{2,+}\cap V_{1,+}) \oplus (W_{1,-}\cap V_{2,-}) .
\label{eq:nerv2}  
\end{equation}
Let $\tilde\perp$ denote the orthogonal complement in $L^2_{(\infty,a)}(A)$
rather than in $L^2(M,E)$.
Then
\begin{align}
\dim(W_{2,-}\cap V_{1,-})
&=
\dim\left(\left(W_{2,-}^{\tilde\perp}+ V_{1,-}^{\tilde\perp}\right)^{\tilde\perp}\right) \nonumber\\
&=
\dim\left(\left(V_{2,-}+ W_{1,-}\right)^{\tilde\perp}\right) \nonumber\\
&=
\dim\left(\left(V_{2,-}\oplus \frac{W_{1,-}}{W_{1,-}\cap V_{2,-}}\right)^{\tilde\perp}\right)
\nonumber\\
&=
\dim(W_{2,-}) - \dim(W_{1,-}) + \dim(W_{1,-}\cap V_{2,-}).
\label{eq:nerv3}
\end{align}
Similarly, one sees
\begin{equation}
\dim(W_{2,+}\cap V_{1,+}) - \dim(W_{1,+}\cap V_{2,+})= \dim(W_{2,+}) - \dim(W_{1,+}) .
\label{eq:nerv4} 
\end{equation}
Equations \eqref{eq:nerv1} - \eqref{eq:nerv4} combine to give
\begin{equation}
\ind(\bar B_1,\bar B_2^\perp)
= \dim(W_{1,+}) + \dim(W_{2,-}) - \dim(W_{1,-}) - \dim(W_{2,+}).
\label{eq:FredPair5}
\end{equation}
\tref{thmad} yields
\begin{equation}
\ind D_{B_1} = \ind D_{\BAPS} + \dim W_{1,+} - \dim W_{1,-}
\label{eq:FredPair6}
\end{equation}
and
\begin{equation}
\ind D_{B_2} = \ind D_{\BAPS} + \dim W_{2,+} - \dim W_{2,-} .
\label{eq:FredPair7}
\end{equation}
Subtracting \eqref{eq:FredPair6} and \eqref{eq:FredPair7}
and inserting \eqref{eq:FredPair5} concludes the proof.
\end{proof}

\begin{remark}
Assumption \eqref{eq:FredVor} is sharp.
For instance, we can choose $B_1$ and $B_2$ such that they have the same $W_\pm$ 
and $V_\pm$, namely $W_\pm=0$, $V_-=L^2_{(-\infty,0)}(A)$, and $V_+=L^2_{[0,\infty)}(A)$.
For $g_1$ we choose a unitary isomorphism $g_1:V_- \to V_+$ which is also continuous
with respect to the $H^{1/2}$-norm.
For $g_2$ we choose $g_2=-g_1$.
Then 
$$
\bar B_2^\perp = \Gamma(-g_2^*) = \Gamma(g_1^*)
= \Gamma(g_1^{-1}) = \Gamma(g_1) = \bar B_1 .
$$
Thus $(\bar B_1,\bar B_2^\perp)$ does not form a Fredholm pair.
In this case, $\|g_1\|=\|g_2\|=1$.

If one of the two boundary conditions is a generalized Atiyah-Patodi-Singer
boundary condition $B(b)$, $b\in\R$,
then the corresponding $g$ vanishes so that assumption \eqref{eq:FredVor}
holds without any restriction on the other boundary condition.
In this case, \tref{thm:FredholmPairs} reduces to \tref{thmad}.
\end{remark}


\subsection{Relative index theory}
\label{suserelind}

Throughout this subsection, assume the Standard Setup~\ref{stase}
and that $D$ and $D^*$ are complete and coercive at infinity.
For convenience assume also that $M$ is connected and $\dM=\emptyset$.

For what follows, compare Example~\ref{ex:TransCond}.
Let $N$ be a closed and two-sided hypersurface in $M$.
Cut $M$ along $N$ to obtain a manifold $M'$, possibly connected,
whose boundary $\dM'$ consists of two disjoint copies $N_1$ and $N_2$ of $N$,
see Figure~6 on page~\pageref{fig:transmission}.
There are natural pull-backs $\mu'$, $E'$, $F'$, and $D'$
of $\mu$, $E$, $F$, and $D$ from $M$ to $M'$.
Assume that there is an interior vector field $T$ along $N=N_1$
such that $D'$ is boundary symmetric with respect to $T$ along $N_1$.
Then $D'$ is also boundary symmetric with respect to the interior
vector field $-T$ of $M'$ along $N=N_2$.
Choose an adapted operator $A$ for $D'$ along $N_1$
as in the Standard Setup~\ref{stase}.
Then $-A$ is an adapted operator for $D'$ along $N_2$
as in the Standard Setup~\ref{stase}
and will be used in what follows.

\begin{thm}[Splitting Theorem]\label{split}\index{splitting theorem}
Let $M$, $M'$, and notation be as above.
Assume that $D'$ is  boundary symmetric
with respect to an interior vector field $T$ of $M'$ along $N$
and choose an adapted operator $A$ as above.

Then, $D$ and $D^*$ are complete and coercive at infinity
if and only if $D'$ and $(D')^*$ are complete and coercive at infinity.
In this case, $D$ and $D'_{B_1\oplus B_2}$ are Fredholm operators with
\[
  \ind D = \ind D'_{B_1\oplus B_2} ,
\]
where $B_1=B(a)=H^{1/2}_{(-\infty,a)}(A)$ and $B_2=H^{1/2}_{[a,\infty)}(A)$,
considered as boundary conditions along $N_1$ and $N_2$, respectively.
More generally,
we may choose any elliptic boundary condition $B_1\subset H^{1/2}(N,E)$
and its $L^2$-orthogonal complement $B_2\subset H^{1/2}(N,E)$.
\end{thm}

\begin{remarks}
(a)
If $M$ is a complete Riemannian manifold and $D$ is of Dirac type, then $M'$
is also a complete Riemannian manifold and $D'$ is of Dirac type as well.
Completeness of the metrics implies completeness of $D$ and $D'$, see \eref{diracom}.
Moreover, $D'$ is boundary symmetric
with respect to the interior unit normal field $T$ of $M'$ along $N$.
Finally, if a curvature condition as in Example~\ref{ex:coercive} (b)
ensures coercivity at infinity of $D$ and $D'$,
then the index formula of \tref{split} applies.

(b)
The Decomposition Theorem~\ref{deco} is the special case of \tref{split}
where $N$ decomposes $M$ into two components, compare Remark~\ref{nomax}.
\end{remarks}

\begin{proof}[Proof of the Splitting Theorem~\ref{split}]
The first assertion,
namely that $D$ and $D^*$ are complete and coercive at infinity if and only if
$D'$ and $(D')^*$ are complete and coercive at infinity,
is immediate from \dref{complete}, \dref{coer} (resp.\ \ref{def:coerc}),
and the compactness of $N$.

Assume now that $D$ and $D^*$ are complete and coercive at infinity.
Then $D$ and $D'_{B_1\oplus B_2}$ are Fredholm operators,
by what we just said and \tref{fred}.

Under the canonical identifications $E|_{N_1} = E|_{N} = E|_{N_2}$,
the transmission condition from Example~\ref{ex:TransCond} reads
\begin{equation*}
  B = \{ (\phi,\psi) \in H^{1/2}(N_1,E) \oplus H^{1/2}(N_2,E) \mid \phi = \psi \} .
\end{equation*}
Recall that $B$ is an elliptic boundary condition.
Furthermore,
with respect to the canonical pull-back of sections from $E$ to $E'$, we have
\begin{equation*}
  \dom D = H^1_D(M,E) = H^1(M',E';B) = \dom D'_B
\end{equation*}
and
\begin{equation}
  \ind D = \ind D'_{B} ,
  \label{one}
\end{equation}
It follows from the discussion in Example~\ref{ex:TransCond} that $B$ is homotopic
to the boundary condition $W_+\oplus V_-^{1/2}$ (that is, $g=0$) with
\begin{align*}
  V_- &= L^2_{(-\infty,0)}(A)\oplus L^2_{(-\infty,0)}(-A)
  = L^2_{(-\infty,0)}(A)\oplus L^2_{(0,\infty)}(A) \\
  W_+ &= \{ (\phi,\phi) \mid \phi \in \ker A \} .
\end{align*}
We note that here $V_-^{1/2}=B(0)=H^{1/2}_{(-\infty,0)}(A)\oplus H^{1/2}_{(0,\infty)}(A)$,
hence
\begin{equation}
\begin{split}
  \ind D'_{B} &= \ind D'_{B(0)} + \dim W_+ \\
  &=  \ind D'_{B(0)} + \dim\ker A \\
  &= \ind D'_{H^{1/2}_{(-\infty,0)}(A)\oplus H^{1/2}_{[0,\infty)}(A)} ,
  \label{two}
\end{split}
\end{equation}
by applying \cref{cor:AgraDy} twice.
If $B_1$ is an elliptic boundary condition,
\begin{equation*}
  B_1 = W_{1,+} \oplus \Gamma(g)^{1/2} ,
\end{equation*}
where the notation is as in \tref{tell} with $a=0$, then
\begin{equation*}
  B_2 = W_{1,-} \oplus \Gamma(-g^*)^{1/2}
\end{equation*}
is the $L^2$-orthogonal complement of $B_1$ in $H^{1/2}(N,E)$.
By \eref{defob}, we may assume that $g=0$.
Then
\begin{equation}\label{aaa}
  B_1=W_{1,+}\oplus V_{1,-}^{1/2}
  \quad\text{and}\quad
  B_2=W_{1,-}\oplus V_{1,+}^{1/2}
\end{equation}
with
\begin{equation}\label{bbb}
  V_{1,-}^{1/2} \oplus W_{1,-} = H^{1/2}_{(-\infty,0)}(A)
  \quad\text{and}\quad
  V_{1,+}^{1/2} \oplus W_{1,+} = H^{1/2}_{[0,\infty)}(A) .
\end{equation}
Applying \cref{cor:AgraDy} to \eqref{aaa} and \eqref{bbb}, respectively,
we get
\begin{equation}\label{three}
\begin{split}
  \ind D'_{B_1\oplus B_2}
  &= \ind D'_{V_{1,-}^{1/2}\oplus V_{1,+}^{1/2}} + \dim W_{1,+} + \dim W_{1,-} \\
  &= \ind D'_{H^{1/2}_{(-\infty,0)}(A)\oplus H^{1/2}_{[0,\infty)}(A)} .
\end{split}
\end{equation}
The claimed index formula is now immediate
from \eqref{one}, \eqref{two}, and \eqref{three}.
\end{proof}

Let $M_1$ and $M_2$ be complete Riemannian manifolds without boundary.
Let $D_i: \Cu(M_i,E_i) \to \Cu(M_i,F_i)$ be Dirac type operators which agree
outside closed subsets $K_i\subset M_i$
and choose $f$ and $\mathcal I_E$ as in \dref{def:agree}.
For $i=1,2$, choose a decomposition $M_i=M_i'\cup M_i''$ such that
$N_i=M_i'\cap M_i''$ is a compact hypersurface in $M_i$,
$K_i$ is contained in the interior of $M_i'$, $f(M_1'')=M_2''$, and $f(N_1)=N_2$. 
Denote the restriction of $D_i$ to $M_i'$ by $D_i'$.
The following result is a general version of the $\Phi$-relative index theorem
of Gromov and Lawson \cite[Thm. 4.35]{GL}.

\begin{thm}\label{firelind}
Under the above assumptions,
let $B_1\subset H^{1/2}(N_1,E_1)$ and $B_2\subset H^{1/2}(N_2,E_2)$
be elliptic boundary conditions which correspond to each other
under the identifications given by $f$ and $\mathcal I_E$ from \dref{def:agree}.
Assume that $D_1$ and $D_2$ and their formal adjoints are coercive at infinity.

Then $D_1$, $D_2$, $D'_{1,B_1}$, and $D'_{2,B_2}$
are Fredholm operators such that
\[
  \ind D_{1} - \ind D_{2}
  = \ind D'_{1,B_1} - \ind D'_{2,B_2} .
\]
\end{thm}

\begin{proof}
The Decomposition Theorem implies that $D_1$, $D_2$, $D'_{1,B_1}$,
and $D'_{2,B_2}$ are Fredholm operators.

Let $T$ be the normal vector field along $N_1$ pointing into $M_1'$.
Since $D_1$ is of Dirac type, $D_1$ is boundary symmetric with respect to $T$.
Let $A$ be an adapted operator on $E_1|_{N}$.
By the Decomposition Theorem, we have
\begin{equation}
  \ind D_{1} = \ind D'_{1,B'} + \ind D''_{1,B''} ,
  \label{x}
\end{equation}
where $B'=B_1$ and $B''$ is the $L^2$-orthogonal complement of $B_1$ in $H^{1/2}(\dM,E_1)$.

Use $f$ and $\mathcal{I}_{E}$ from \dref{def:agree}
to identify $M_1\setminus K_1$ with $M_2\setminus K_2$
and $E_1$ over $M_1\setminus K_1$ with $E_2$ over $M_2\setminus K_2$.
Identify $N_1$ and $M_1''$ with their images $N_2$ and $M_2''$ under $f$.
Then, with the same choice of $T$ and $A$ as above,
we obtain a corresponding index formula
\begin{equation}
  \ind D_{2} = \ind D'_{2,B'} + \ind D''_{2,B''} ,
  \label{y}
\end{equation}
where $B'=B_2$ and $B''$ correspond to the above $B'$ and $B''$
under the chosen identifications.
Now $\ind D''_{1,B''} = \ind D''_{2,B''}$ since $D_1$ outside $K_1$
agrees with $D_2$ outside $K_2$,
and therefore the asserted formula follows by subtracting \eqref{y} from \eqref{x}.
\end{proof}

\begin{proof}[Proof of the Relative Index Theorem~\ref{relind}]
The first assertion of \tref{relind}, 
namely that $D_{1}$ is a Fredholm operator if and only if $D_{2}$ is a Fredholm operator,
is immediate from \tref{fred} and the compactness of $K_1$ and $K_2$.

Assume now that $D_1$ and $D_2$ are Fredholm operators.
Choose a compact hypersurface $N=N_1$ in $M_1\setminus K_1$
which decomposes $M_1$ into $M_1=M_1'\cup M_1''$,
where $M_1'$ is compact and $K_1$ is contained in the interior of $M_1'$,
and choose a corresponding decomposition of $M_2$.
Then we get
\begin{equation}
\ind D_{1} - \ind D_{2} = \ind D'_{1,B_1} - \ind D'_{2,B_2} ,
\label{eq:praeGL}
\end{equation}
in the setup of and the index fromula in by \tref{firelind}.

Now choose a compact Riemannian manifold $X$ with boundary $N$
equipped with a Dirac type operator such that gluing along $N$
yields a smooth closed Riemannian manifold $\tilde M_1=M_1' \cup_{N} X$
together with a smooth extension\symbolfootnote[2]{Such an extension exists. The manifold $X$ can be chosen to be diffeomorphic to $M_1'$, for instance.} $\tilde D_1$ of $D_1$.
Since $D_1$ outside $K_1$ agrees with $D_2$ outside $K_2$,
$\tilde M_2=M_2' \cup_N X$ is also a smooth closed Riemannian manifold
and comes with a smooth extension $\tilde D_2$ of $D_2$.
As above, we get
\begin{equation}
\ind \tilde D_{1} - \ind \tilde D_{2} = \ind D'_{1,B_1} - \ind D'_{2,B_2} .
\label{eq:praeGL2}
\end{equation}
Now, the Atiyah-Singer index theorem for Dirac type operators on closed manifolds gives
\begin{equation}
  \ind \tilde D_i
  = \int_{\tilde M_i}\alpha_{\tilde D_i} ,
  \label{eq:praeGL3}
\end{equation}
where $\alpha_{\tilde D_i}$ denotes the index density of $\tilde D_i$.
We apply \eqref{eq:praeGL} -- \eqref{eq:praeGL3},
the fact that $\alpha_{\tilde D_1}=\alpha_{\tilde D_2}$
on $M_1\setminus K_1 = M_2\setminus K_2$,
and that $\alpha_{\tilde D_i} = \alpha_{D_i}$ on $K_i$, and obtain
\begin{align*}
\ind D_{1} - \ind D_{2}
&= 
\ind \tilde D_{1} - \ind \tilde D_{2}\\
&=
\int_{\tilde M_1}\alpha_{\tilde D_1} - \int_{\tilde M_2}\alpha_{\tilde D_2} \\
&=
\int_{K_1}\alpha_{D_1} - \int_{K_2}\alpha_{D_2}.
\qedhere
\end{align*}
\end{proof}

\begin{remark}
In Theorems~\ref{firelind} and \ref{relind},
it is also possible to deal with the situation that $M_1$ and $M_2$
have compact boundary and elliptic boundary conditions $B_1$ and $B_2$
along their boundaries are given.
One then chooses the hypersurface $N=N_i$ such that it does not intersect
the boundary of $M_i$ and such that the boundary of $M_i$ is contained in $M_i'$.
The same arguments as above yield
$$
\ind D_{1,B_1} - \ind D_{2,B_2}
= 
\ind D'_{1,B_1\oplus B_1'} - \ind D'_{2,B_2\oplus B_2'} ,
$$
where $B_i'$ and $B_2'$are elliptic boundary condition along $N_1$ and $N_2$
which correspond to each other under the identifications
given by $f$ and $\mathcal I_E$ as in \dref{def:agree}.

Using \tref{thmad},
one can reduce to the case of Atiyah-Patodi-Singer boundary conditions.
If the domains $M_i'$ are compact,
one can then express the indices on the right hand side
as integrals over the index densities plus boundary contributions.
There are two kinds of boundary contribution:
the eta invariant of the adapted boundary operator
and the boundary integral of the transgression form,
compare \cite[Thm.~3.10]{APS} and \cite[Cor.~5.3]{G1}.

Since the boundary contributions along $N_1$ and $N_2$ agree,
they cancel each other when taking the difference
$\ind D'_{1,B_1\oplus B_1'} - \ind D'_{2,B_2\oplus B_2'}$.
This observation leads to another proof of the relative index formula
(in the case where $\dM_i=\emptyset$),
not using the auxiliary manifold $X$,
but the local index theorem for compact manifolds with boundary.
\end{remark}


\subsection{The cobordism theorem}\label{coth}


Assume that $D:\Cu(M,E)\to\Cu(M,E)$ is a formally selfadjoint operator of Dirac type
over a complete Riemannian manifold $M$ with compact boundary $\dM$.
Then $\sigma_D(\xi)$ is skew-Hermitian, for any $\xi\in T^*M$.

Let $\tau$ be the interior unit conormal field along $\dM$.
Then, by the Clifford relations \eqref{Cliff1} and \eqref{Cliff2},
$i\sigma_0=i\sigma_D(\tau)$ is a field of unitary involutions of $E|_{\dM}$
which anticommutes with $\sigma_A(\xi)=\sigma_0^{-1}\circ\sigma_D(\xi)$,
for all $\xi\in T^*\dM$.
The eigenspaces of $i\sigma_0$ for the eigenvalues $\pm1$
split $E|_{\dM}$ into a sum of fiberwise perpendicular subbundles $E^\pm$.
By \cref{cor:involution} and \pref{prop:localregular},
$B^\pm:=H^{1/2}(\dM,E^\pm)$ are $\infty$-regular elliptic boundary conditions for $D$. 

With respect to the splitting $E|_{\dM}=E^+\oplus E^-$,
any adapted boundary operator takes the form
$$
A = 
\begin{pmatrix}
A_{++} & A^{-} \cr A^{+} & A_{--}
\end{pmatrix} .
$$
Since the differential operators $A^\pm$ are elliptic of order one,
they define Fredholm operators $A^\pm:H^{1}(\dM,E^\pm) \to L^2(\dM,E^\mp)$.
The cobordism theorem is concerned with the index of theses operators.

Since $\sigma_A(\xi)$ interchanges the bundles $E^+$ and $E^-$,
the operators $A_{++}$ and $A_{--}$ are of order zero.
Thus we may, without loss of generality, assume that $A_{++} = A_{--} = 0$.
Moreover, $Z=(A^{-})^*-A^{+}$ is also of order zero.
Since the addition of a zero order term
does not change the index of an operator of order one,
we may replace $A^{+}$ by $A^{+}+Z$
and arrive at the normal form for $A$ which we use from now on,
\begin{equation}
A = 
\begin{pmatrix}
0 & A^- \cr A^+ & 0
\end{pmatrix} ,
\label{nfa}
\end{equation}
where $(A^+)^*=A^-$.
The adapted operator $A$ is still not uniquely determined by these requirements; 
we have the freedom to replace $A^+$ by $A^+ + V$ and $A^-$ by $A^- + V^*$,
where $V:E^+ \to E^-$ is any zero-order term.

In the notation of Theorems~\ref{tell} and \ref{ellbou},
the boundary conditions $B^+$ and $B^-$ can be written as
\begin{equation*}
  B^+=W_+\oplus\Gamma(g)^{1/2}
  \quad\text{and}\quad
  B^-=W_-\oplus\Gamma(-g)^{1/2} ,
\end{equation*}
where
\begin{equation*}
  W_\pm = \ker A^\pm , \quad
  V_- = L^2_{(-\infty,0)} (A) , \quad
  V_+ = L^2_{(0,\infty)} (A) , \quad
\end{equation*}
and $gx = i\sigma_0x$.
Since $\sigma_0^*=\sigma_0^{-1}=-\sigma_0$, we have $g^*=g$ and hence
\begin{equation*}
  B^- = \sigma_0(B^-) = (\sigma_0^{-1})^*(B^-)
\end{equation*}
is the adjoint boundary condition of $B^+$.

\begin{lemma}\label{cobolem}
Let $A$, $B^+$, and $B^-$ be chosen as above and assume that each connected component of $M$ has a nonempty boundary.
Then
\[
  \ker D_{B^+} = \ker D_{B^-} = 0 .
\]
\end{lemma}

\begin{proof}
Let $\Phi\in\ker D_{B^\pm,\max}$.
By \eqref{eq:ParInt}, we have
\begin{align*}
  0
  &= (D_{\max}\Phi,\Phi)_{L^2(M)} - (\Phi,D_{\max}\Phi)_{L^2(M)} \\
  &= - (\sigma_0\mathcal R\Phi,\mathcal R\Phi)_{L^2(\dM)} \\
  &= \pm i \|\mathcal R\Phi \|_{L^2(\dM)}^2 ,
\end{align*}
and hence $\mathcal R\Phi=0$.

Now extend $M$, $E$, and $D$ beyond the boundary $\dM$
to a larger manifold $\tilde M$ (without boundary) equipped with a bundle $\tilde E$
and Dirac type operator $\tilde D$ as in the proof of \tref{domdmax}.
Moreover, extend $\Phi$ to a section of $\tilde E$
by setting $\tilde\Phi=0$ on $\tilde M \setminus M$.
For any test section $\Psi\in\Cucc(\tilde M,\tilde E)$ we have, again by \eqref{eq:ParInt},
\begin{align*}
(\tilde\Phi,\tilde D^*\Psi)_{L^2(\tilde M)}
&=
(\Phi,D^*\Psi)_{L^2(M)} \\
&=
(D_{\max}\Phi,\Psi)_{L^2(M)} + (\sigma_0\mathcal R\Phi,\mathcal R\Psi)_{L^2(\dM)} \\
&= 0.
\end{align*}
Thus $\tilde\Phi$ is a weak (and, by elliptic regularity theory, smooth) solution of $\tilde D\tilde\Phi=0$.
Hence $\tilde\Phi$ is a solution to the Laplace type equation $\tilde D^2\tilde\Phi=0$
which vanishes on an nonempty open subset of $\tilde M$.
The unique continuation theorem of Aronszajn \cite{Ar} implies $\tilde\Phi=0$
and hence $\Phi=0$.
\end{proof}

We have
\begin{equation}
  \ind A^+ = \dim W_+ - \dim W_- = - \ind A^-
\label{eq:indA}
\end{equation}
since $A^-$ is the adjoint operator of $A^+$.

\begin{proof}[Proof of \tref{cobothm}]
By \eqref{dba}, $D_{B^-}$ is the adjoint of $D_{B^+}$.
\lref{cobolem} therefore yields
\begin{equation*}
  - \ind D_{B^-}
  = \ind D_{B^+}
  = \dim \ker D_{B^+} - \dim \ker D_{B^-}
  = 0 .
\end{equation*}
Without loss of generality we may assume $g=0$, see Example~\ref{defob}.
Then \cref{cor:AgraDy} yields
\begin{align*}
  \ind D_{B^+}
  &= \ind D_{\BAPS} + \dim W_+ , \\
  \ind D_{B^-}
  &= \ind D_{\BAPS} + \dim W_- .
\end{align*}
We obtain that $\dim W_+ = \dim W_-$.
Equation \eqref{eq:indA} concludes the proof.
\end{proof}


\appendix


\section{Some functional analytic facts}
\label{secFA}


We collect some functional analytic facts, which have been used in the main body of the text, 
and which are not easily found in the standard literature.

\begin{prop}
\label{propFredholm}
Let $H$ be a Hilbert space, let $E$ and $F$ be Banach spaces and let 
$L:H \to E$ and $P:H \to F$ be bounded linear maps.
We assume that $P:H \to F$ is onto.

Then the following hold:
\begin{enumerate}[(i)]
\item 
The kernel of $L_{|\ker(P)}: \ker(P) \to E$ and the kernel
of $L\oplus P: H \to E\oplus F$ have equal dimension.
\item
The range of $L_{|\ker(P)}: \ker(P) \to E$ is closed if and only if
the range of $L\oplus P: H \to E\oplus F$ is closed.
\item
The cokernel of $L_{|\ker(P)}: \ker(P) \to E$ and the cokernel
of $L\oplus P: H \to E\oplus F$ have equal dimension.
\item
$L_{|\ker(P)}: \ker(P) \to E$ is Fredholm of index $k$ if and only
if $L\oplus P: H \to E\oplus F$ is Fredholm of index $k$.
\end{enumerate}
\end{prop}

\begin{proof}
Write $L_1 := L_{|\ker(P)}$, $L_2 := L_{|\ker(P)^\perp}$, and 
$P_2 := P_{|\ker(P)^\perp}$.
With respect to the splittings $H = \ker(P) \oplus \ker(P)^\perp$ and
$E\oplus F$ the operator $L \oplus P$ takes the matrix form
$$
L \oplus P = 
\begin{pmatrix}
  L_1 & L_2 \cr 
  0 & P_2
\end{pmatrix}.
$$
By the open mapping theorem $P_2 : \ker(P)^\perp \to F$ is an isomorphism
(with a bounded inverse), so that $\begin{pmatrix} \id_{\ker(P)} & 0 \cr 
0 & P_2\m\end{pmatrix}$ is an isomorphism. 
Moreover, $\begin{pmatrix} \id_E & -L_2 P_2\m \cr 0 & \id_F\end{pmatrix}$ is an
isomorphism with inverse $\begin{pmatrix} \id_E & L_2 P_2\m \cr 0 & \id_F
\end{pmatrix}$.
Therefore the kernel of $L\oplus P$ has the same dimension as the kernel of 
$$
\begin{pmatrix} \id_E & -L_2 P_2\m \cr 0 & \id_F\end{pmatrix}
\begin{pmatrix} L_1 & L_2 \cr 0 & P_2 \end{pmatrix}
\begin{pmatrix} \id_{\ker(P)} & 0 \cr 0 & P_2\m\end{pmatrix}
=
\begin{pmatrix} L_1 & 0 \cr 0 & \id_F \end{pmatrix} 
$$
which is the same as the dimension of the kernel of $L_1$.
The other statements follow similarly.
\end{proof}

\begin{remark}
The proposition also holds and the proof works without change if $H$ is only a 
Banach space and one assumes that $\ker(P)$ has a closed complement.
\end{remark}

\begin{prop}\label{HormPeet}
Let $X$ and $Y$ be Banach spaces and $L:X \to Y$ be a bounded linear map.
Then the following are equivalent:
\begin{enumerate}[(i)]
\item \label{HP1}
The operator $L$ has finite-dimensional kernel and closed image.
\item\label{HP2}
There is a Banach space $Z$, a compact linear map $K:X\to Z$,
and a constant $C$ such that
\[
  \|x\|_X \leq C\cdot \left(\|Kx\|_Z + \|Lx\|_Y \right) ,
\]
for all $x \in X$. 
In particular, $\ker K\cap \ker L=\{0\}$.
\item\label{HP3}
Every bounded sequence $(x_n) $ in $X$ such that $(Lx_n)$ converges in $Y$
has a convergent subsequence in $X$.
\end{enumerate}
Moreover, these equivalent conditions imply
\begin{enumerate}[(iv)]
\item 
For any Banach space $Z$ and compact linear map $K:X \to Z$
such that $\ker K\cap\ker L=\{0\}$, there is a constant $C$ such that
\[  \| x\|_X \le C \big(\|Kx\|_Z + \|Lx\|_Y \big)  \]
for all $x \in X$. 
\end{enumerate}
\end{prop}

The equivalence of \eqref{HP1} and \eqref{HP3} is Proposition 19.1.3 in \cite{Ho}.

\begin{proof}[Proof of \pref{HormPeet}]
We show $\eqref{HP1}\Rightarrow\eqref{HP2}$.
Assume \eqref{HP1}.
Since $\ker L$ is finite-dimensional, it has a closed complement $U$ in $X$.
Since the image of $L$ is closed, $L|_U:U\to\im L$ is an isomorphism of Banach spaces.
Moreover, since $U$ is closed and $\ker L$ is finite-dimensional,
the projection $K:X\to\ker L$ along $U$ is a compact operator.
It follows that
\[
  F: X \to \ker L \oplus \im L , \quad Fx=(Kx,Lx) ,
\]
is an isomorphism of topological vector spaces, hence we have \eqref{HP2} with $Z=\ker L$.

Next we show $\eqref{HP2}\Rightarrow\eqref{HP3}$.
Assume \eqref{HP2}, and let $K$ be a compact linear map as assumed in \eqref{HP2}.
Let $(x_n)$ be a bounded sequence in $X$ such that $(Lx_n)$ converges in $Y$.
Since $K$ is compact, we may assume, by passing to a subsequence if necessary,
that $(Kx_n)$ converges in $Z$.
The estimate in \eqref{HP2} then implies that $(x_n)$ is a Cauchy sequence in $X$,
hence \eqref{HP3}.

Now we show $\eqref{HP3}\Rightarrow\eqref{HP1}$.
Assume \eqref{HP3}, and let $(x_n)$ be a bounded sequence in $\ker L$.
Then $(x_n)$ subconverges in $\ker L$, by \eqref{HP3}.
It follows that $\dim\ker L<\infty$.

Let $U$ be a complement of $\ker L$ in $X$.
Let $(x_n)$ be a sequence in $X$ such that $Lx_n\to y\in Y$.
We have to show that $y$ lies in the image of $L$.
For this, it is sufficient to show that $(x_n)$ subconverges.
Without loss of generality we assume that $x_n\in U$, for all $n$.

If $\|x_{n_k}\|_X\to\infty$ for some subsequence $(x_{n_k})$,
then $u_k:=x_{n_k}/\|x_{n_k}\|_X$ has norm $1$ and $Lu_k\to0$.
By \eqref{HP3}, the sequence of $u_k$ has a convergent subsequence.
The limit $u$ is a unit vector in $U$ with $Lu=0$.
This is a contradiction, because $U$ is complementary to $\ker L$.
Hence we may assume that the sequence of $(x_n)$ is bounded.
But then it subconverges, by \eqref{HP3}.

Finally, we show $\eqref{HP2}\Rightarrow \mbox{(iv)}$.
Assume \eqref{HP2} and let $K_0:X\to Z_0$ be a compact linear map as in \eqref{HP2}.
Let $K:X\to Z$ be any compact linear map such that $\ker K\cap \ker L=\{0\}$.
If the assertion does not hold, then there is a sequence $(x_n)$ of unit vectors
in $X$ such that $\|Kx_n\| + \|Lx_n\| \to 0$.
This implies, in particular, that $\|Lx_n\| \to 0$ and hence that $\|K_0x_n\|\ge\delta$
for some $\delta>0$.
Now $K_0$ is compact, hence, up to passing to a subsequence, $(K_0x_n)$ is a
Cauchy sequence.
Since $Lx_n\to0$, this implies that $(x_n)$ is a Cauchy sequence, by \eqref{HP2}.
If $x:=\lim x_n\in X$, then $x$ is a unit vector with $Kx=Lx=0$.
This contradicts $\ker K\cap\ker L=\{0\}$.
\end{proof}

\begin{lemma}\label{lem:Triple}
Let $B_+$ and $B_-$ be Banach spaces, let $(\cdot,\cdot): B_+ \times B_- \to \C$ be a perfect pairing.
Let $H$ be a Hilbert space
and let there be continuous imbeddings $B_+ \subset H \subset B_-$.
Assume that the restriction of $(\cdot,\cdot)$ to $B_+ \times H$ coincides
with the restriction of the scalar product on $H$.

Let $W\subset B_+$ be a finite-dimensional subspace and let $V_-$
be the annihilator of $W$ in $B_-$, i.~e.\ $V_- = \{x\in B_- \mid (w,x)=0 \,\,\forall x\in W\}$.
Put $V_0 := V_- \cap H$ and $V_+ := V_- \cap B_+$.

Then
$$
B_- = W \oplus V_-, \quad
H = W \oplus V_0, \quad
B_+ = W \oplus V_+.
$$
Moreover, the second decomposition is orthogonal
and $W$ is the annihilator of $V_+$ in $B_-$.
The pairing $(\cdot,\cdot)$ restricts to a perfect pairing of $V_+$ and $V_-$.
\end{lemma}

\begin{proof}
For $w\in W \cap V_-$ we have $0=(w,w)=\|w\|_H^2$, thus $w=0$.
This shows $W \cap V_- = 0$ and hence also $W \cap H = W \cap V_+ = 0$.
Since the codimension of the annihilator of $W$ in any space can at most be $\dim(W)$
we conclude $B_- = W \oplus V_-$ and similarly for $H$ and $B_+$.

Orthogonality of the second decomposition is clear because the restrictions
of the pairing and of the scalar product coincide.
Clearly, $W$ is contained in the annihilator of $V_+$ in $B_-$.
Conversely, let $v=w+v_- \in B_- = W \oplus V_-$ be in the annihilator of $V_+$.
Then for all $v_+ \in V_+$ we have $0= (v_+,w+v_-) = (v_+,v_-)$, 
thus $v_-=0$, i.~e.\ $v=w\in W$.

Nondegeneracy of the pairing of $V_+$ and $V_-$ also follows easily.
\end{proof}



\printindex


\end{document}